\documentclass[11pt,reqno,english]{article}

\usepackage{amsmath,amsfonts,amssymb,graphicx,amsthm,url}
\usepackage[noadjust]{cite}
\usepackage{stmaryrd}
\usepackage{mathrsfs,booktabs,tabularx}
\usepackage{xifthen,xcolor,tikz,setspace}
\usetikzlibrary{decorations.pathmorphing,patterns,shapes,calc,decorations}
\usetikzlibrary{decorations.pathreplacing}
\usepackage{mathtools,upgreek,paralist}
\usepackage[final]{showkeys} 

\definecolor{refkey}{gray}{.75}
\definecolor{labelkey}{gray}{.5}
\usepackage[algo2e,boxed,vlined,algoruled]{algorithm2e}
\usepackage{comment}
\usepackage[shortlabels]{enumitem}

\usepackage{extarrows}

\usepackage[letterpaper]{geometry}
\usepackage[colorinlistoftodos]{todonotes}
\presetkeys{todonotes}{inline, color=green}{}

\usepackage{accents}
\usepackage[algo2e,boxed,vlined,algoruled]{algorithm2e}

\usepackage{authblk}

\usepackage[small]{caption}
\usepackage[colorlinks=true]{hyperref}
\colorlet{DarkGreen}{green!50!black}
\colorlet{DarkGray}{gray!60!black}

\numberwithin{equation}{section}

\renewcommand{\restriction}{\mathord{\upharpoonright}}
\renewcommand{\epsilon}{\varepsilon}

\usepackage{color}
 \definecolor{refkey}{gray}{.5}
 \definecolor{labelkey}{gray}{.5}
\definecolor{light}{gray}{.9}

\usepackage{soul}

\usepackage[capitalise]{cleveref}

\newtheorem{thm}{Theorem}

\newtheorem{theorem}{Theorem}[section]
\newtheorem*{theorem*}{Theorem}
\newtheorem{lemma}[theorem]{Lemma}
\newtheorem{lem}[theorem]{Lemma}

\newtheorem{claim}[theorem]{Claim}
\crefname{claim}{Claim}{Claims}

\newtheorem{prop}[theorem]{Proposition}

\newtheorem{observation}[theorem]{Observation}

\newtheorem{cor}[theorem]{Corollary}

\newtheorem{assumption}[theorem]{Assumption}

\theoremstyle{definition}{

\newtheorem{definition}[theorem]{Definition}

\newtheorem*{definition*}{Definition}

\newtheorem{remark}[theorem]{Remark}
\newtheorem*{remark*}{Remark}

}

\newcommand{\E}{\mathbb E}

\renewcommand{\P}{\mathbb P}

\newcommand{\R}{\mathbb R}
\newcommand{\Z}{\mathbb Z}

\newcommand{\cI}{\ensuremath{\mathcal I}}

\newcommand{\SE}{\textsc{se}}
\newcommand{\SW}{\textsc{sw}}

\newcommand{\tmix}{t_{\textsc{mix}}}

 \renewcommand{\epsilon}{\varepsilon}

\DeclareMathOperator{\var}{Var}

\DeclareMathOperator{\diam}{diam}

\newcommand{\tv}{{\textsc{tv}}}

\title{Rapid phase ordering of Ising dynamics on $\mathbb Z^2$}
\date{}

\author{Reza Gheissari\thanks{Department of Mathematics, Northwestern University. Evanston, IL. \url{gheissari@northwestern.edu}}   { and } Allan Sly\thanks{Department of Mathematics, Princeton University. Princeton, NJ. \url{allansly@princeton.edu}}}

\begin{document}

\maketitle

\vspace{-1cm}
\begin{abstract}
    We consider the phase ordering problem for the low-temperature Ising dynamics initialized from a biased and disordered initialization. Work of Fontes, Schonmann, Sidoravicius (2002) showed that at zero-temperature, Ising Glauber dynamics on $\mathbb Z^d$ for $d\ge 2$ initialized from i.i.d.\ spins on each vertex that are $+1$ with sufficiently large probability, absorbs into the all-plus configuration quickly. We prove that analogous behavior holds throughout the low-temperature regime of the Ising model in two dimensions. Namely, there exists $p_0 <1$ such that Ising Glauber dynamics initialized from i.i.d.\ spins that are $+1$ with probability $p>p_0$, run at any low temperature $\beta>\beta_c$ converges rapidly to the plus phase  measure $\pi^+$. 
    
    The result is proved using a spacetime multiscale coupling valid in any $d\ge 2$, that boosts a uniform-in-$\beta$ quasi-polynomial bound on the mixing time of Ising dynamics with plus boundary conditions, into rapid phase ordering from biased initializations with no boundary conditions. 
\end{abstract}

\section{Introduction}

The Ising model is one of the simplest and best-studied models of phase transitions in statistical physics. 
It is the distribution over assignments of $\{-1,+1\}$ spins to the vertices of a finite graph $G= (V,E)$ given by 
\begin{equation}\label{eq:Ising-distribution}
    \pi_{G,\beta}(\sigma) \propto \exp\Big(  \beta \sum_{v\sim w} \sigma_v \sigma_w\Big)\,,
\end{equation}
where we use $v\sim w$ to mean $\{v,w\}$ forms an edge in $E$.  
In this paper, we are predominantly interested in the most physical case where the underlying graph is an $n\times n$ box in $\mathbb Z^d$ for $d\ge 2$, or is the infinite-volume limit of this distribution. 

The study of out-of-equilibrium dynamics of the Ising model date back to Glauber~\cite{Glauber}. 
We consider the continuous-time Ising Glauber dynamics on $\mathbb Z^d$ for $d\ge 2$. That is the continuous-time Markov chain which is initialized with $X_0 = x_0$, assigns every vertex $v$ a rate-$1$ Poisson clock, and when the clock at vertex $v$ rings at time $t$, the Markov chain $(X_t)_{t\ge 0}$ updates 
\begin{align}\label{eq:Glauber-update-rule}
    X_t(v)  = \begin{cases}
        +1 & \text{w.\ prob.\ }\propto \exp( 2\beta \sum_{w\sim v} X_{t^-}(w))\\ - 1 & \text{w.\ prob.\ }\propto \exp( - 2\beta \sum_{w\sim v} X_{t^-}(w))
    \end{cases}\,,
\end{align}
and $X_t(w) = X_{t^-}(w)$ for all $w\ne v$. It is easy to check that this dynamics satisfies the detailed balance equations with respect to the Ising Gibbs distribution~\eqref{eq:Ising-distribution}, and therefore for fixed $n$, converges as $t\to\infty$ to this distribution. 
It is well known that on boxes of side-length $n$ in $\mathbb Z^d$, the mixing time from worst-case initialization (denoted $\tmix$) undergoes a phase transition at the critical point of the system $\beta_c(d)$. When $\beta<\beta_c(d)$ the mixing time from any initial state is $O(\log n)$~\cite{MaOl1}, with even the exact constant in front of $\log n$ having been identified~\cite{LS-information-percolation}. On the other hand, when $\beta>\beta_c(d)$ the mixing time is $\exp(\Theta(n^{d-1}))$, because it takes an exponential time to go from the all-plus initialization to a majority minus configuration (though by spin-flip symmetry, the latter has probability $1/2$ at equilibrium: see~\cite{Thomas,Pisztora96,Bodineau05}).

A central question of interest in this low-temperature regime is of understanding the metastable behavior of the plus and minus phases (the Gibbs measure $\pi$ conditioned on positive or negative magnetization). Dating back to early physics work of Lifshitz~\cite{Lifshitz} and later Huse and Fisher~\cite{FisherHuse}, this has been studied in the context of questions of the following form: started ``out-of-equilibrium" how quickly and with what probabilities (depending on the initialization) does an Ising system order into one of the two phases? The study of such questions goes under the umbrella term of ``phase ordering kinetics" as per e.g., the important monograph of Bray~\cite{bray1994theory}.

In the mathematics literature, Fontes, Schonmann, and Sidoravicius~\cite{FoScSi02} studied the following setup of this question: 
``\emph{the behavior of a magnetic system which is initially at
high temperature under a strong external magnetic field, and from time 0 on is suddenly subject to a very low temperature and to no external field}." They then simplified to zero-temperature dynamics, and showed that in the zero-temperature limit where the chain follows majority dynamics on $\mathbb Z^d$, if the initialization is i.i.d.\ coin flips on the vertices with probability $p\ge 1-\epsilon$ of being plus (denoted $\bigotimes \text{Rad}(p)$), then the configuration converges to the all-plus configuration at stretched exponential rate. Morris~\cite{Morris-zero-temp} showed that the minimum parameter can be at least $\frac{1}{2} + \epsilon_d$ for $\epsilon_d \downarrow 0$ as $d \to\infty$. 

In this paper, we study the phase ordering problem at low but positive temperatures: Prepare a disordered Ising state $X_0$ with some bias towards plus spins, and then run low-temperature Ising Glauber dynamics from that initialization. Is the initial bias enough to ensure rapid equilibration to the metastable plus phase measure? This question has a long history, and a version was posed by Liggett~\cite[Open Problem 7, Chapter IV]{Liggett-book}. 
A simulation of this process is depicted in Figure~\ref{fig:phase-ordering-simulations}. 

From the statistical physics side, this has been a rich question because unlike situations where there is a weak external field or boundary conditions far away, the \emph{only} source of the symmetry breaking in the long-time dynamics is the bias in the initialization. From the Markov chain mixing time perspective, understanding fast (quasi-)convergence from certain ``nice" initial configurations, when correlations do not decay and the overall mixing time is exponentially slow is a problem of much recent interest for which there do not exist many tools. 

On  graph families other than $\mathbb Z^d$, there has been some progress on low-temperature phase ordering. On the complete graph, the Ising Glauber dynamics is essentially fully described by the magnetization process which forms a birth and death chain, and fast equilibration to the plus phase from biased initializations was shown in~\cite{LLP,DLP-censored-Glauber}. 
Caputo and Martinelli~\cite{CaMaTree} showed that on the infinite $d$-regular tree, from the i.i.d.\ $\bigotimes \text{Rad}(p)$ initialization for $p$ sufficiently close to $1$, the Glauber dynamics quickly converges to the infinite-volume plus measure on the tree. Recently~\cite{gheissari2025rapid} showed the analogous result for the Ising model on random $d$-regular graphs. 

Our main result is such a phase ordering result for the low-temperature Ising dynamics on~$\mathbb Z^2$. 
We note that whereas in the other geometries mentioned, e.g., the random regular graph result of~\cite{gheissari2025rapid}, rapid phase ordering holds from \emph{any} initialization with sufficiently large magnetization, on $(\Z/n\Z)^d$ there exist ``interface configurations" with arbitrarily large magnetization from which low-temperature Ising dynamics is slow to escape. Therefore in the context of this paper, beyond the magnetization bias, randomness in the initialization is essential.

\subsection{Main results}

There are two closely related settings in which we phrase our main result: a finite-volume one and an infinite-volume one. For the former, we consider the Ising Glauber dynamics on the torus of side-length $n$, denoted $\mathbb T_n^d = (\mathbb Z/n\mathbb Z)^d$, and for the latter, the Glauber dynamics on Ising configurations on the infinite $\mathbb Z^d$ graph. By convergence to the plus phase, on the finite torus, we mean the Gibbs distribution conditioned on having a majority of its spins be plus, i.e., $\pi_{\mathbb T_n^d}^+ = \pi_{\mathbb{T}_n^d}(\cdot \mid \sum_v \sigma_v \ge 0)$, while on $\mathbb Z^d$ we mean the extremal infinite-volume Gibbs measure $\pi^+_{\mathbb Z^d}$. 

\begin{figure}
\centering
    \includegraphics[width=.23\textwidth]{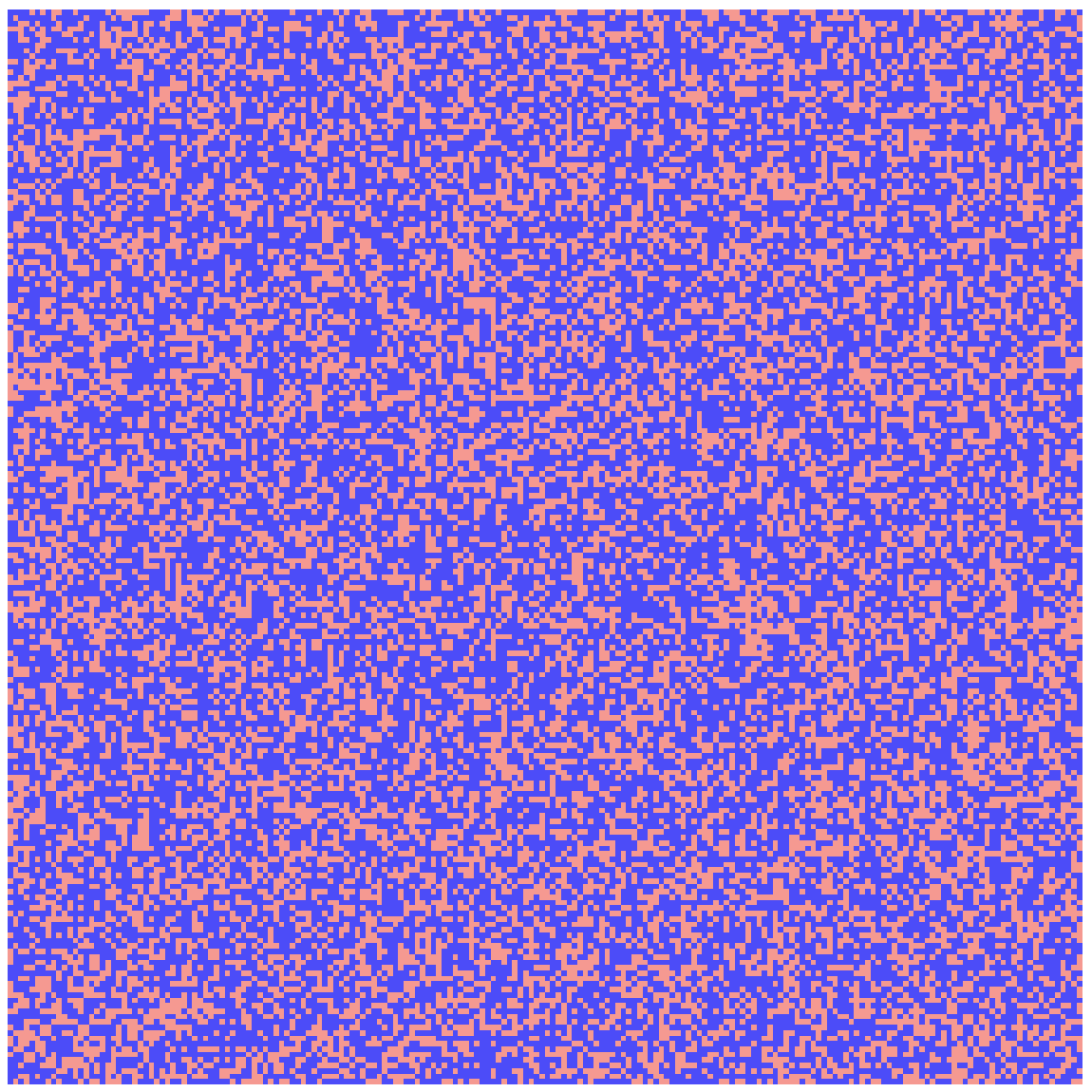}
    \includegraphics[width=.23\textwidth]{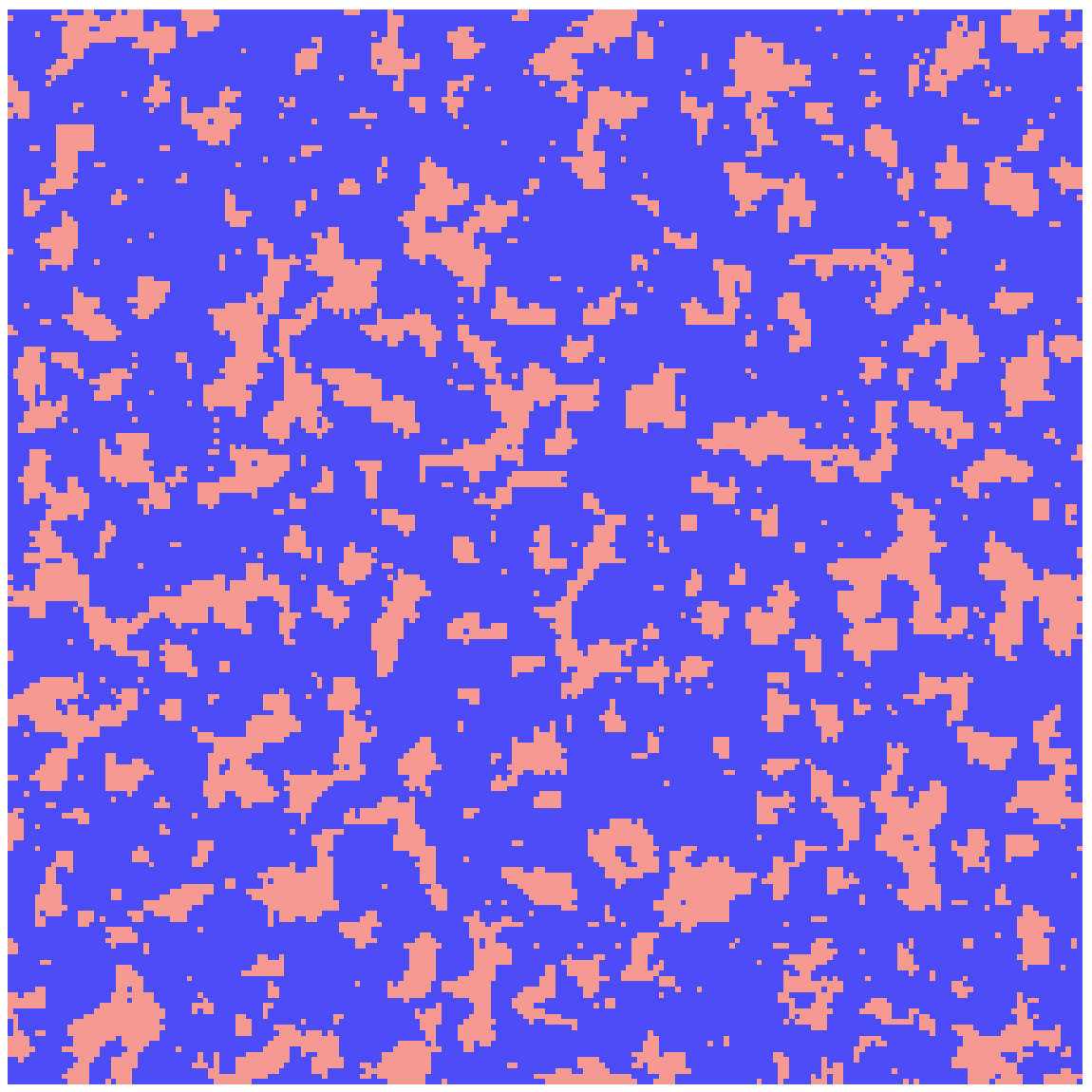}
        \includegraphics[width=.23\textwidth]{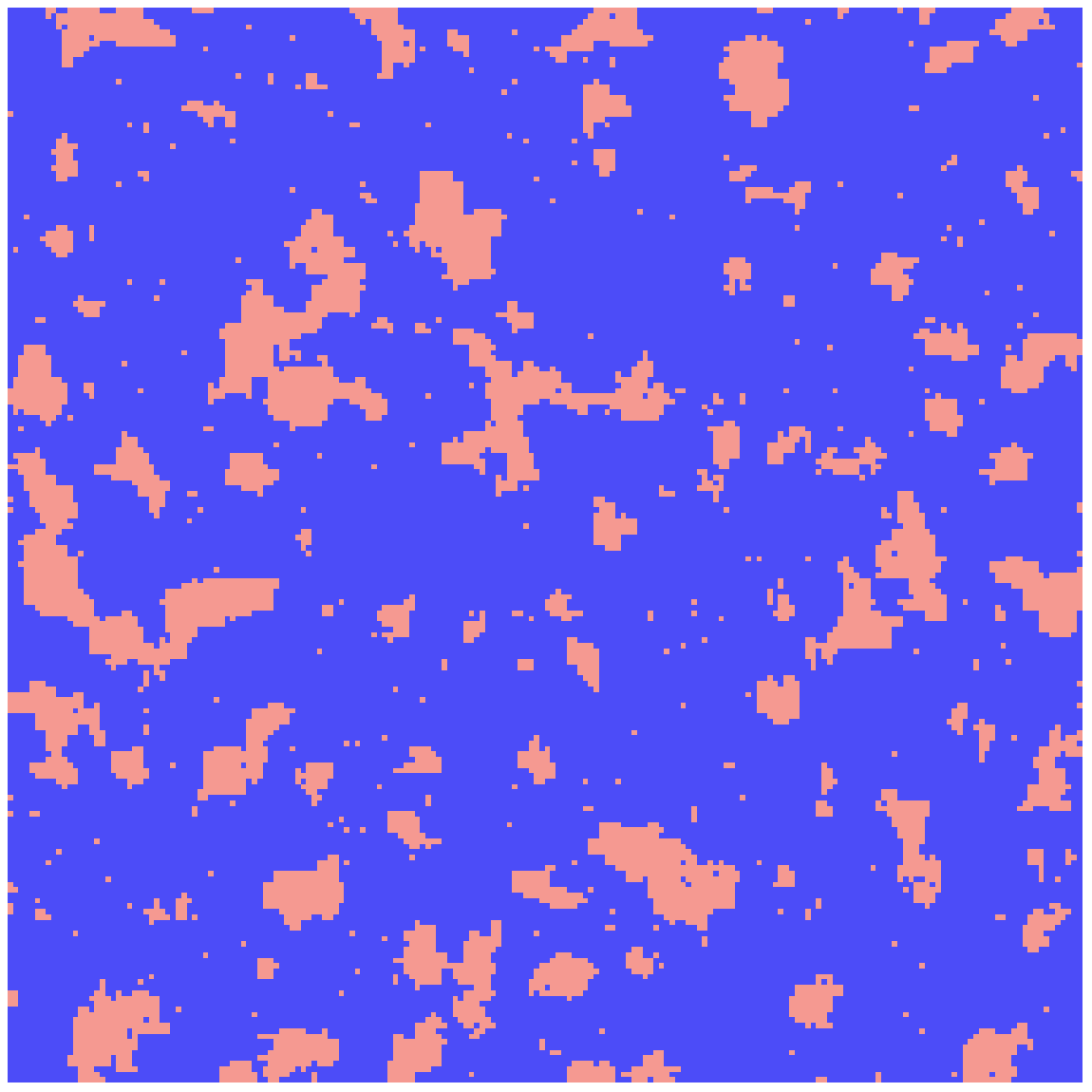}
    \includegraphics[width=.23\textwidth]{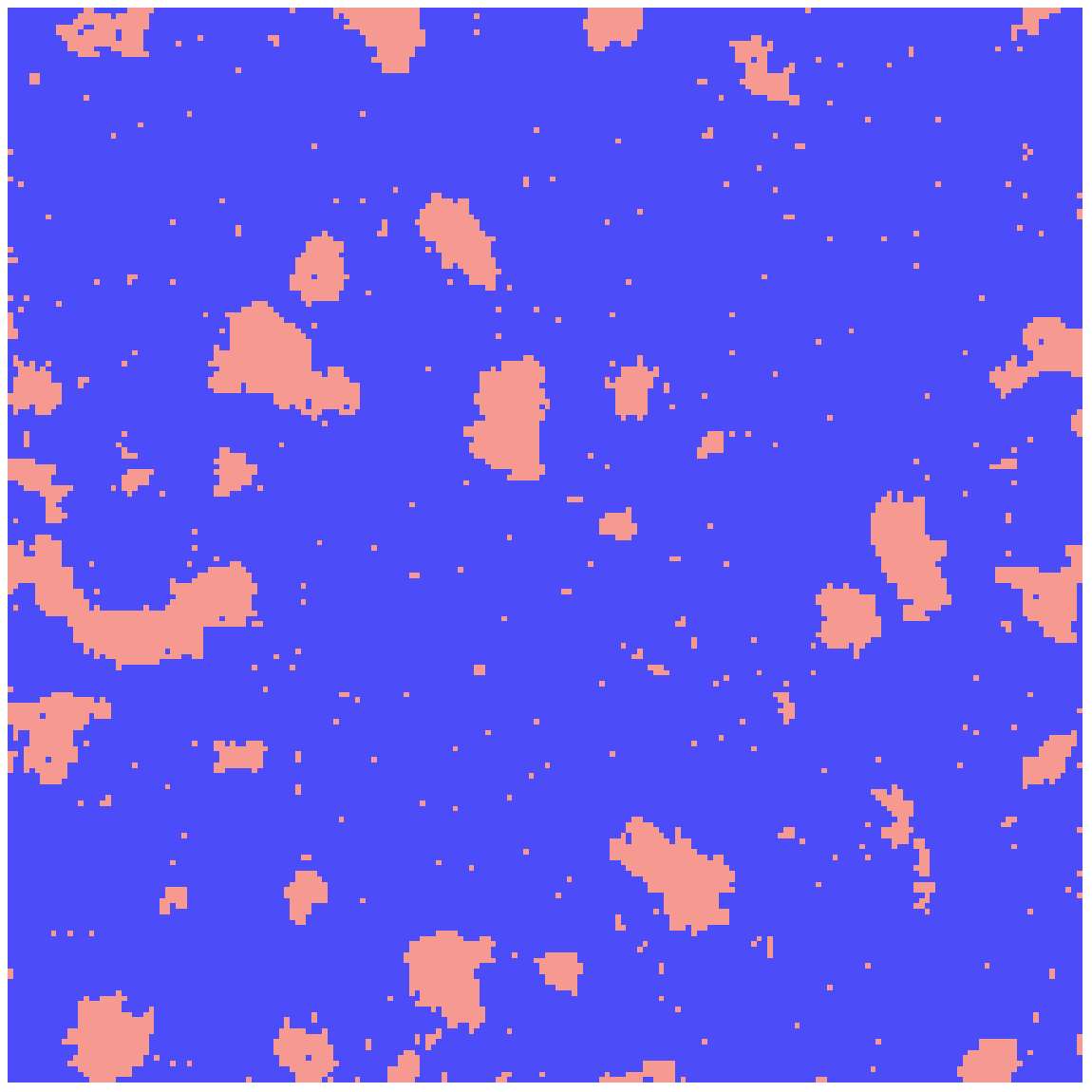}
\caption{Snapshots over time of low-temperature Ising dynamics on $\mathbb T_n^2$ initialized from i.i.d.\ spins with a slight bias towards plus (blue). The configuration first locally coarsens, then the small regions where minus regions dominate shrink by motion by mean-curvature while their complement gets close to the plus phase.}\label{fig:phase-ordering-simulations}
\end{figure}

\begin{thm}\label{thm:main} There exists $p_0<1$ such that for all $p>p_0$ and all $\beta>\beta_c(2)$ the following holds. 

\smallskip
\noindent \textbf{Finite volume.}  Suppose $X_t$  is the continuous-time Ising Glauber dynamics on $\mathbb{T}_n^2$ initialized from $x_0 \sim \bigotimes_{\mathbb{T}^2_n} Rad(p)$. With probability $1-o(1)$, $x_0$ is such that 
\begin{align*}
    \|\mathbb P_{x_0}( X_t\in \cdot ) - \pi_{\mathbb{T}_n^2}^+ \|_{\tv} \le n^{-10} \qquad \text{for all }t\in [n^{o(1)} ,  e^{\Omega(n)}]\,.
\end{align*}
\textbf{Infinite volume.} If $X_t$ is the continuous-time Glauber dynamics on all of $\Z^2$ initialized from $x_0 \sim \bigotimes_{\Z^2} \text{Rad}(p)$, then with probability going to $1$ as $R\to\infty$, $x_0$ is such that for $\Lambda_R = \{-\frac{R}{2},...,\frac{R}{2}\}^2$, for all $t \ge R$, we have 
\begin{align*}
        \|\mathbb P_{x_0}(X_t(\Lambda_R)\in \cdot) - \pi^+_{\mathbb Z^2}(\sigma(\Lambda_R)\in \cdot)\|_{\tv} \le \exp( - \Omega((\log t)^{10}))\,.
\end{align*}
\end{thm}

\begin{remark}
    In the above, the sequence $n^{o(1)}$ in the time needed to quasi-equilibrate could be taken to be $t_0 = \exp(O(\log n)^{1/100})$. All hidden constants that depend on $\beta$ in $o(\cdot),\Omega(\cdot)$ terms may deteriorate as $\beta \downarrow \beta_c$, but are uniformly bounded as $\beta\uparrow \infty$.   
\end{remark}

The above result is expected to hold for all biases, i.e., $p_0 > 1/2$, but we require a sufficiently biased initialization. Even for zero-temperature dynamics, showing absorption to the all-plus configuration started from $\bigotimes \text{Rad}(\frac{1}{2}+\epsilon)$ remains a folklore open problem.  
We emphasize here that $p_0$ is fixed before $\beta$ and therefore the initialization can be significantly ``less plus" than the plus measure to which it is equilibrating. Indeed, most of the work in our paper fixes $p$ and studies the sufficiently low temperature regime of $\beta>\beta_0$ for a large fixed $\beta_0$. In the complementary regime of $\beta\in (\beta_c,\beta_0]$, the minimal bias $p_0$ can be taken to be large enough that the initialization stochastically dominates the plus phase measure, and the bound already follows by sandwiching it between $\pi^+$ and the all-plus initialization, which is known to converge rapidly~\cite{MaTo,GhSi22}.

The initialization in Theorem~\ref{thm:main} need not have been product; any initialization with a sufficiently high density of pluses, and exponential decay of correlations would yield the same result of rapid convergence to the plus phase. This generalization is relevant in the physics context of cooling (cf.\ simulated annealing) where one is interested in the equilibration rate at low-temperature, from an initialization in a higher temperature Gibbs measure. The following gives analogous results to Theorem~\ref{thm:main} when the initialization is any Ising measure with a strong external field, or a different low-temperature Ising measure in its corresponding plus phase. For an external field $h\in \R$, we write $\pi_{\beta,h}$ to denote the Gibbs measure which, in addition to~\eqref{eq:Ising-distribution} has a factor of $\exp( h\sum_{v} \sigma_v)$, so the case without a subscript $h$ is understood as $h=0$.

\begin{thm}\label{cor:mixing-from-other-temp}
    There exists $h_0$ such that for all $\beta'$ and all $h\ge h_0$, if $X_t$ is continuous-time Glauber dynamics run at $\beta>\beta_c$ on $\mathbb{T}_n^2$, initialized from $x_0 \sim \pi_{\beta',h}$, with probability $1-o(1)$, $x_0$ is such that 
\begin{align*}
    \|\mathbb P_{x_0}( X_t\in \cdot ) - \pi_{\mathbb{T}_n^2,\beta}^+ \| \le n^{-10}\,, \qquad \text{for all }t\in [n^{o(1)} ,  e^{\Omega( n)}]\,.
\end{align*}
The analogous infinite-volume statement as in Theorem~\ref{thm:main} also holds. 
\end{thm}

Likewise, if $X_t$ is continuous-time Glauber dynamics run at sufficiently low temperatures $\beta>\beta_0$ on $\mathbb T_n^2$, and initialized from a different low-temperature plus-phase measure $x_0 \sim \pi_{\mathbb T_n^2,\beta'}^+$ for $\beta'>\beta_0$, then with probability $1-o(1)$, $x_0$ is such that  $$\|\mathbb P_{x_0}(X_t \in \cdot) - \pi_{\mathbb{T}_n^2,\beta}^+\|_{\tv}  \qquad \text{for all } t \in  [n^{o(1)}, e^{\Omega(n)}]\,.$$
By averaging over $x_0$ and using spin-flip symmetry, we conclude that the mixing time for Glauber dynamics run at low temperature $\beta$, initialized from another low temperature $\pi_{\mathbb T_n^2,\beta'}$, is $n^{o(1)}$. 

\subsection{Proof sketch}

Our high-level strategy is to perform a multi-scale argument in space and time to bootstrap a quasi-polynomial input on the mixing time with plus boundary conditions, into Theorem~\ref{thm:main}. This part of the argument will actually apply in any dimension $d\ge 2$. We thus present the following assumption and theorem to fully separate out which parts of our argument require at present dimension two.

\begin{assumption}\label{assump:uniform-mixing-time}
    The mixing time of continuous-time Glauber dynamics on $\Lambda_L = \{-\frac{L}{2},...,\frac{L}{2}\}^d$ with all-plus boundary conditions satisfies the following: there exists $\beta_0$ and a $\beta$-independent constant $C(d)>0$ such that for all $\beta>\beta_0$ and all $L \ge 1$, 
    \begin{align*}
        \tmix \le C \exp( (\log L)^C)\,.
    \end{align*}
\end{assumption} 

It is a long-standing open problem (see Open Problems, Question 1 in~\cite{LP}, as well as~\cite{Martinelli-phase-coexistence,Martinelli-notes}) in the study of Markov chain mixing times that Glauber dynamics on boxes of side-length $L$ in $\mathbb Z^d$ with plus boundary conditions have polynomial (in fact, $\tilde O(L^2)$ continuous-time) mixing time. The above assumption only asks for this mixing time to be quasi-polynomial, but it does require that all hidden constants in the quasi-polynomial be uniform over large $\beta$. Such uniformity in large $\beta$ is reasonable to hope for, as in the $\beta =\infty$ limit, the $\tilde O(L^2)$ polynomial bound holds~\cite{lacoin2014zero-2D,Lacoin-Lifshitz-any-dimension}.

\begin{thm}\label{thm:main-general}
    Fix $d\ge 2$. Suppose Assumption~\ref{assump:uniform-mixing-time} holds. There exists $p_0(d)<1$ such that for all $p>p_0$ and all $\beta > \beta_0$, if $X_t$ is continuous-time Glauber dynamics on $\mathbb T^d_n = (\mathbb Z/n\mathbb Z)^d$, with probability $1-o(1)$, $x_0 \sim \bigotimes_{\mathbb T_n^d}\text{Rad}(p)$ is such that 
    \begin{align*}
        \|\mathbb P_{x_0} (X_t \in \cdot) - \pi^+_{\mathbb T^d_n}\|_{\tv} \le n^{-10} \qquad \text{ for all $t\in [n^{o(1)}, e^{ \Omega(n^{d-1})}]$}\,.
    \end{align*}
    If $X_t$ is the infinite volume Glauber dynamics on $\mathbb Z^d$, with probability going to $1$ as $R\to\infty$, $x_0\sim \bigotimes _{\Z^d} \text{Rad}(p)$ is such that for $\Lambda_R = \{-\frac{R}{2},...,\frac{R}{2}\}^2$, once $t \ge R$, 
    \begin{align*}
        \|\mathbb P_{x_0}(X_t(\Lambda_R)\in \cdot) - \pi^+_{\mathbb Z^d}(\sigma(\Lambda_R)\in \cdot)\|_{\tv} \le \exp( - \Omega((\log t)^{10}))\,.
    \end{align*}
\end{thm}

    In words, Theorem~\ref{thm:main-general} says that if in any dimension $d \ge 2$ the Ising model has a uniform-in-large-$\beta$ quasi-polynomial bound on its mixing with plus boundary conditions, then without boundary conditions where mixing times are exponentially slow, it exhibits rapid quasi-equilibration to the plus phase from disordered biased initializations. To get Theorem~\ref{thm:main} from Theorem~\ref{thm:main-general}, we prove  that Assumption~\ref{assump:uniform-mixing-time} holds in $d=2$ (\cite{LMST} had shown a quasi-polynomial bound on the mixing time, but with its $\beta$-dependencies being non-explicit, and in fact poorly behaved as $\beta \uparrow \infty$). 

   The remainder of the proof sketch will be separated into two parts: The first overviews the multiscale spacetime coupling to establish Theorem~\ref{thm:main-general};  the second describes the technical ingredients we need to establish that Assumption~\ref{assump:uniform-mixing-time} holds in $d=2$.

\subsubsection{The multiscale setup}

We prove our main theorems by putting the Markov chains initialized from the disordered initialization and from stationarity in the same probability space, and coupling them in such a way that they agree with high probability on $\mathbb T_n^2$ after time $n^{o(1)}$. 
Let $X_t^Q$ be the Glauber dynamics initialized from a disordered biased initialization $Q$ and $X_t^\pi$ be the stationary Markov chain initialized at the quasi-stationary distribution $\pi^+$.

At a high level, we follow a multiscale recursion in space and time, showing that if we tile the lattice by blocks of side-length $\ell_k$, the   density of such blocks on which the two chains agree at time $T_k$, i.e., $X_{T_k}^Q(B) = X_{T_k}^\pi(B)$, is bounded by a threshold $q_k$ that improves suitably with $k$. The precise relationships between the spatial scale $(\ell_k)_k$, the time scale $(T_k)_k$ and disagreement probabilities $(q_k)_k$ is important, but we defer the exact expressions to Section~\ref{sec:scales}. For now, we mention that the length scales grow quasi-polynomially as $\ell_k = e^{(\log\ell_{k-1})^C}$ (with the constant $C$ depending on the constant in Assumption~\ref{assump:uniform-mixing-time}), the time-scales are linearly related to the length scales, and the probability thresholds decay quasi-polynomially in $\ell_k$.

As with many such renormalization schemes, we will define a suitable notion of block $B$ of side-length $\ell_k$ being bad such that 
\begin{enumerate}
	\item The event $\{B \text{ is bad}\}$ is measurable with respect to the randomness on vertices at distance at most $O(\ell_k)$ from $B$ at times $[0,T_k]$; 
	\item If the block $B$ is not bad, then $X_{T_k}^Q(B) = X_{T_k}^\pi(B)$; 
	\item The probability of $B$ being bad is at most $q_k$. 
\end{enumerate}
By the second and third items, once $q_k \le o(1/n^d)$, then by a union bound, we have coupling on all of $\mathbb T_n^d$, yielding fast mixing to the plus phase.

The key in such spacetime renormalization schemes is in the specifics of the bad event so that it satisfies the desired inductive argument: namely that if at time $T_k$, all the scale-$k$ blocks interior to a scale-$k+1$ block $B'$ satisfy (1)--(3) above, then so does the event $\{B' \text{ is bad}\}$ at time $T_{k+1}$. 

This will go by reasoning that so long as the bad scale-$k$ blocks interior to $B'$  are sufficiently sparse, these ``bad regions" can be locally cured by boxes with all-plus boundary conditions surrounding them between times $[T_k,T_{k+1}]$. Importantly, unlike in high temperature (e.g.~\cite{DSVW}) local mixing timescales are larger than the speed at which information travels, so pretending that the boundary conditions is all-plus may pose a problem. But because the ``bad regions" are surrounded by good blocks, we use buffer annuli $\mathcal A$ separating the plus boundary from the bad region the chain with fixed plus boundary condition and the original process $X_t^\pi$ agree for all times $[T_k,T_{k+1}]$. This allows us to argue that if the local chain cures the bad region, then so does the original process $X_t^\pi$. 

With that description in hand, let us be a bit more mathematically precise about what we mean by a block being good. We also accompany the definition with Figure~\ref{fig:intro-proof-sketch}, which depicts the coupling between the local and global Markov chains which enables the curing of the bad regions.

\begin{definition}\label{def:informal-good-block}
    [Informal version of Definition~\ref{def:Dtilde}]
	A block $B$ at scale $k+1$ is called good (the complement of bad), if the following hold: 
	\begin{itemize}
		\item \emph{Local coupling}: If $V_t^{R}$ denotes Glauber dynamics on $R$, a box surrounding a connected set of scale-$k$ bad blocks with all-$+$ boundary conditions, then $V_t^R$ mixes (couples all initializations) in the time $[T_k,T_{k+1}]$;
		\item \emph{Coupling local and global Markov chains using buffers}: There is a buffer annulus $\mathcal A_R$ separating the boundary of $R$ from the bad blocks internal to it, such that $V_t^R(\mathcal A_R) = X_t^\pi(\mathcal A_R)$ for all $t\in [T_k,T_{k+1}]$;
		\item \emph{Linear speed of information propagation}: There does not exist  $v_1,...,v_N\in B$ for $N \ge \ell_{k+1}$ such that $v_i \sim v_{i+1}$ and their clocks ring at increasing times $t_{i_1}<...<t_{i_N} \in [T_k,T_{k+1}]$.
	\end{itemize}
\end{definition}

\begin{figure}[t]
\centering
\begin{tikzpicture}[
    scale=0.6, 
    gridline/.style={thin, black!80},
    blueoutline/.style={very thick, blue!80!black},
    redcell/.style={fill=red!40},
    greycell/.style={fill=black!15},
    orangehatch/.style={pattern=north east lines, pattern color=orange}
]
\begin{scope}[xshift=0cm]

    \def\gridSize{9}
    \def\cellSize{1cm}

    \draw[fill=green!5, thick] (-1,-1) rectangle (9,9);

    \tikzset{block/.style={fill=red!40!white}}
    
    \fill[block] (3 * \cellSize, 6 * \cellSize) rectangle (4 * \cellSize, 7 * \cellSize);
    \fill[block] (6 * \cellSize, 6 * \cellSize) rectangle (7 * \cellSize, 7 * \cellSize);
    \fill[block] (1 * \cellSize, 2 * \cellSize) rectangle (2 * \cellSize, 3 * \cellSize);
    \fill[block] (2 * \cellSize, 1 * \cellSize) rectangle (3 * \cellSize, 2 * \cellSize);
    
    \draw[<->] (-1, \gridSize * \cellSize + 0.55cm) -- (\gridSize * \cellSize, \gridSize * \cellSize + 0.55cm)
        node[midway, fill=white] {$\ell_{k+1}$};
        
    \draw[<->] (6 * \cellSize - 0.9 * \cellSize, 0 * \cellSize - 1.2 * \cellSize)-- (6 * \cellSize - 0.1 * \cellSize, 0 * \cellSize - 1.2 * \cellSize);
    \node[font=\tiny] at (6 * \cellSize - 0.5 * \cellSize, 0 * \cellSize - 1.5 * \cellSize) {$\ell_k$};

    \draw (-1,-1) grid (\gridSize * \cellSize, \gridSize * \cellSize);

\end{scope}

\begin{scope}[xshift=14cm]

    \draw[fill=green!5, thick] (-1,-1) rectangle (9,9);

    \draw[gridline] (-1,-1) grid (9,9);
    \fill[orangehatch, even odd rule] 
        (1.35, 5.35) rectangle (3.65, 7.65) 
        (1.65, 5.65) rectangle (3.35, 7.35); 

    \fill[orangehatch, even odd rule] 
        (5.35, 5.35) rectangle (7.65, 7.65) 
        (5.65, 5.65) rectangle (7.35, 7.35);

    \fill[orangehatch, even odd rule] 
        (0.35, 0.35) rectangle (3.65, 3.65) 
        (0.65, 0.65) rectangle (3.35, 3.35);

    \fill[redcell] (2,6) rectangle (3,7); 
    \fill[redcell] (6,6) rectangle (7,7); 
    \fill[redcell] (1,2) rectangle (2,3); 
    \fill[redcell] (2,1) rectangle (3,2); 

    \draw[gridline] (1,5) grid (4,8);
    \draw[gridline] (5,5) grid (8,8);
    \draw[gridline] (0,0) grid (4,4);

    \draw[blueoutline] (1,5) rectangle (4,8);
    \draw[blueoutline] (5,5) rectangle (8,8);
    \draw[blueoutline] (0,0) rectangle (4,4);
\end{scope}
\end{tikzpicture}
\caption{ Left: a scale $k+1$ block tiled by scale $k$ blocks. The scale $k$ blocks are bad (red) with small probability, finitely dependently, and thus the bad blocks are with high probability sparse. \\ Right: The local coupling event on the block $B$ asks that dynamics restricted to $R$ (a blue square) with $+1$ boundary conditions couple in their mixing time. This then implies coupling on the full scale $k+1$ block if the local and global chains equal each other on the orange shaded buffer regions $\mathcal A_R$. 
}\label{fig:intro-proof-sketch}
\end{figure}
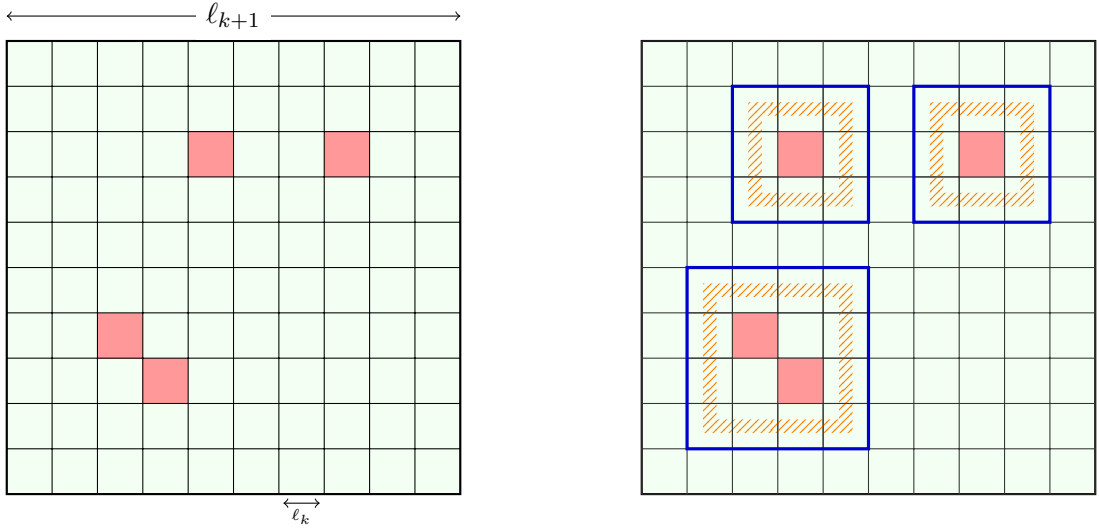

The main part of the proof, after defining these formally, is then showing inductively, that for the definition of Definition~\ref{def:informal-good-block}, the desired items (1)-(3) above hold for all $k$. 

Of course, an inductive proof requires a base case, which in this case asks at a minimum that at time $t=0$, one has $X_{0}^Q(B) = X_{0}^\pi(B)$ with probability at least $1-\epsilon$ for blocks of scale $\ell_0$. The base case is attained from the simple observation that if the bias in the disordered initialization is sufficiently large, then for fixed $\ell_0$, there is a $1-\epsilon$ probability that both initializations are identically $+1$ on $B$. 
This is where the uniformity of Assumption~\ref{assump:uniform-mixing-time} over large $\beta$ becomes absolutely essential. If the constants in the mixing time deteriorate with $\beta \uparrow \infty$, then the $\ell_0$ we need to start with would grow with $\beta$, and therefore in order for a block of side-length $\ell_0$ to be initialized at all-plus we would need to take the initial bias to grow with $\beta$. For the minimal bias $p_0$ to not depend on $\beta$ large (ensuring our initialization is not stochastically comparable to the $\pi^+$), this cannot happen, and all our mixing time and equilibrium estimates have to be uniform over large~$\beta$. 

\subsubsection{Uniform-in-$\beta$ quasipolynomial mixing with plus boundary}
Once one obtains Theorem~\ref{thm:main-general}, the remaining step to get Theorem~\ref{thm:main} for large $\beta$ is establishing that Assumption~\ref{assump:uniform-mixing-time} holds in dimension two. 

\begin{thm}\label{thm:beta-independent-2D-mixing-time}
    Consider the Ising model on $\Lambda_n = [- \frac{n}{2},\frac{n}{2}]^2\cap \mathbb Z^2$ with all-plus boundary conditions. There exist $\beta_0>0$ and $C>0$ such that for all $\beta > \beta_0$ and all $n\ge 1$, 
    \begin{align*}
        \tmix \le  \exp( C (1 \vee \log n)^3)\,.
    \end{align*}
\end{thm}

We in fact prove a bound that for all $n \ge e^{\beta/5}$, one has $\tmix \le \exp( C \beta (\log n)^2)$ for $C$ independent of $\beta$, effectively quantifying the implicit constants in the $\exp(O(\log^2 n))$ in~\cite{LMST}.

Theorem~\ref{thm:beta-independent-2D-mixing-time} is shown by bounding $\beta$ dependencies carefully in~\cite{LMST} (and preceding works which established the equilibrium inputs necessary to those paper including understanding of the surface tension, two-point function, etc in the low-temperature Ising model~\cite{MaTo,DKS}). Indeed, naively, the proofs of sub-exponential mixing time with plus boundary conditions from~\cite{MaTo,LMST} rely on ``guiding" the dynamics to equilibrium by utilizing certain somewhat rare events, whose rarity actually gets worse exponentially as $\beta \to\infty$. We therefore have to correct for this by emulating their proofs, but with the recursive scale-changes being $\beta$-dependent, and with sharp understanding of $\beta$-dependencies in hidden constants in the interface equilibrium estimates of~\cite{LMST}. 

We give two examples of the kind of non-asymptotic, uniform-in-$\beta$ estimates we need to develop. 

\begin{lem}[Special case of Proposition~\ref{prop:vertical-maximum-bound-replacement}]\label{lem:informal-max-bound}
	There exists $C>0$ such that for all $\beta > \beta_0$, all $\ell,h$ the following holds. In an $\ell \times h$ rectangle with minus boundary conditions on the bottom and plus boundary conditions on the other three sides,  the probability the interface reaches height $C\max\{ e^{ - \beta} \sqrt{\ell \log \ell}\,,\,\beta\}$ is at most $\ell^{-10}$. 
\end{lem}

Gaussian upper tail bounds on the height of an Ising interface have been established at low temperatures dating back to the monograph~\cite{DKS} (see also~\cite{GreenbergIoffe,IOVW-invarianceprinciple} for Brownian bridge convergence, and the more refined bounds of~\cite{LMST} that hold at all low temperatures) However, in these prior bounds, the constants in the Gaussian upper tail were either only asymptotic as $n\to \infty$, or $\beta$-dependent. In particular the fact that the variance in the tail bound scaled like $n e^{-2\beta}$ was not seen because a $\beta$-oblivious ``sharp triangle inequality" for the surface tension was used. 

Key to Lemma~\ref{lem:informal-max-bound}, and other equilibrium estimates necessary for Theorem~\ref{assump:uniform-mixing-time} is a sharp understanding of the probability of a $\pm$-interface separating minuses from pluses, connecting some vertices $x,y$ in $(\mathbb Z^2)^*$. By Kramers--Wannier duality, this is equivalent to two-point functions at a dual high temperature $\beta^*<\beta_c$.  
The celebrated Ornstein--Zernike theory, developed for all high temperatures in~\cite{PfisterVelenik,CIV03}, obtains the following sharp asymptotics for these dual two-point functions: for vertex $v$ forming angle $\theta\in [-\frac{\pi}{4},\frac{\pi}{4}]$ to the origin $0$, for all $\beta^*<\beta_c$, 
\begin{align*}
	\langle \sigma_0\sigma_v \rangle_{\beta^*,\mathbb Z^2}  = (1+o_{\|v\|}(1)) \frac{\Phi_{\beta}(\theta)}{\sqrt{\|v\|} }\exp ( - \tau_\beta(\theta) \|v\|)\,,
\end{align*}
where $\theta$ is the angle formed between the line $xy$ and the $x$-axis, $\Phi_\beta, \tau_\beta$ are explicit constants, and $o_{\|v\|}(1)$ means as $\|v\|\to\infty$ at angle $\theta$.  
However, the $o(1)$ in the above was not uniform over $\beta$; in fact if the distance $\|v\|$ is sub-exponential in $\beta$, the behavior of this interface is notably different. Our Lemma~\ref{lem:OZ-asymptotics-beta-indep} bounds two point functions at all finite distances $\|v\|$, up to constants that are uniform over $\beta$: For all $\beta>\beta_0$ (so that $\beta^*<\beta_0^*$) and all $v \in \mathbb Z^2$ forming angle $\theta_v\in [-\frac{\pi}{4},\frac{\pi}{4}]$, 
\begin{align}\label{eq:intro-version-of-sharp-OZ}
     \frac{C^{-1}}{\sqrt{1 + |v_1| e^{-2\beta} + |v_2|} }e^{ - \tau_\beta(\theta_v) \|v\|} \le \langle \sigma_0 \sigma_v\rangle_{\beta^*,\mathbb Z^2} \le \frac{C}{\sqrt{1 + |v_1| e^{-2\beta} + |v_2|} }e^{ - \tau_\beta(\theta_v) \|v\|}\,,
\end{align}
where, importantly, the proportionality constant $C$ does not depend on $\beta$.  

\subsection*{Acknowledgments}
R.G.\ thanks Vladas Sidoravicius for having introduced him to this problem. The research of R.G.\ is supported in part by NSF CAREER grant 2440509 and NSF DMS grant 2246780. A.S. was supported by a Simons Investigator Grant.

\section{The multi-scale framework}

In this section, we focus on the multiscale (recursion in space and time) framework that will be used to couple the chain $X_{t}^{x_0}$ from the random initialization, to the stationary chain $X_t^{\pi^+}$. The section will reduce the proof of Theorem~\ref{thm:main-general} to a series of equilibrium estimates to show the probability a block at scale-$k$ is bad is small. 

\subsection{Time and space scales, and notation for different processes}\label{sec:scales}

The multiscale framework will be indexed by a parameter $k\ge 1$. We define the following sequences:  Let $M$ be a sufficiently large constant (the choice will only depend on $d$ and the constant $C$ in the mixing time bound assumption). Initialize $\ell_0$ as a large enough constant depending only on $M,d$ and the $C$ in Assumption~\ref{assump:uniform-mixing-time}), and $t_0 = 0$. Define, 
\begin{align}
    \text{Spatial scale $\ell_k$}: & \qquad \log \ell_k = (\log (\ell_{k-1}))^M \qquad \text{i.e., } \quad \ell_k = \exp( \ell_0^{M^k})\,; \label{eq:spatial-scale} \\ 
    \text{Temporal scale $t_k$}: & \qquad t_k = \frac{1}{M} \ell_k  \qquad  \text{and} \qquad T_k = \sum_{i=1}^k t_k\,; \label{eq:temporal-scale}\\ 
    \text{Disagreement probability $q_k$}:  & \qquad q_k: = \frac{1}{\ell_{k+3}} = \exp( -(\log\ell_{k})^{M^3})\,. \label{eq:disagreeent-probability}
\end{align}

Given the above definitions, and in particular, the spatial scale, we define collections of blocks at each spatial scale, that partition $\mathbb Z^2$. 
Namely, let $\mathscr{B}_k$ be the set of blocks $B_{v,k} = (v+ [-\frac{\ell_k}{2}, \frac{\ell_k}{2}]^2)\cap \Z^2$ indexed by $v\in \ell_k \mathbb Z^2$. (For readability, we drop floors/ceilings and omit all associated rounding.)

Consider the following coupling of dynamics with different initializations and boundary conditions. Throughout, we will use the notations $(U_t)_t,(V_t)_{t}$ and variants on it for Markov chains used in the analysis, while reserving $(X_t)_t$ for the main Glauber dynamics on $\mathbb Z^d$ we are interested in. We begin by putting all these processes in the same probability space using the grand coupling.

\begin{definition}\label{def:initialization-coupling}
    [Time-zero monotone coupling] Assign each vertex in $\mathbb Z^d$ at time $0$ two $\text{Unif}[0,1]$ random variables $(\mathcal U_{v,0}^\pi)_{v\in \mathbb Z^d}$ and $(\mathcal U^Q_{v,0})_{v\in \mathbb Z^d}$ independently. Start with an arbitrary enumeration $v_1,v_2,...$ of the vertices of $\mathbb Z^d$. 
    
    For any distribution $\pi_{\Lambda}^{\eta}$ where $\Lambda \subset \mathbb Z^d$ finite, and $\eta \in \{\pm 1\}^{\Lambda^c}$ is a boundary condition, draw a sample $\sigma \sim \pi_{\Lambda}^{\eta}$  and $Q\sim \bigotimes_{v\in \Lambda} \text{Rad}(p)$ as follows: iteratively, for $i\ge 1$, 
    \begin{itemize}
        \item let $\sigma_{v_i}$ be $+1$ if $$\mathcal U_{v_i,0}^\pi \le \pi_{\Lambda}^{\eta} ( \sigma_{v_i} =+1 \mid (\sigma_{v_j})_{j<i:v_j\in \Lambda})$$
        and $-1$ else. 
        \item let $Q_{v_i}$ be $+1$ if $\mathcal U_{v_i,0}^Q \le p$ and $-1$ else. . 
    \end{itemize}
    Once all vertices in $\Lambda$ have been processed, terminate. 
\end{definition}

By monotonicity of the Ising model via the FKG inequality, the above coupling encodes monotone relations on the initializations: if two distributions $\pi_{\Lambda}^\eta \preceq \pi_{\Lambda'}^{\eta'}$, then the samples $(\sigma,\sigma')$ from the two under the above coupling will obey this ordering $\sigma \le \sigma'$ pointwise. Moreover,  $\sigma \wedge Q$ (denoting the minimum applied entrywise) will be stochastically below $\sigma' \wedge Q$ since the same $Q$ is used.

\begin{definition}\label{def:dynamics-coupling}
    [Grand coupling of dynamics] For each vertex $v$ in $\mathbb Z^d$, assign an independent intensity-$1$ Poisson process process on $(0,\infty)$, denoted $\mathcal T_v = (t_{v,1}<t_{v,2}<...)$ and to each point $t\in \bigcup_i \{t_{v,i}\}$ assign a $\text{Unif}[0,1]$ random variable $\mathcal U_{v,t}$. For every finite $\Lambda \subset \mathbb Z^d$, boundary conditions $\eta \in \{\pm 1\}^{\Lambda^c}$ and inital configuration $x_0\in \{\pm 1\}^\Lambda$, generate the Glauber dynamics $(X_t)_{t\ge 0} = (X_{\Lambda^\eta,t}^{x_0})_{t\ge 0}$ on $\Lambda$ with boundary conditions $\eta$ initialized from $x_0$ as follows: let $t_1<t_2<...<t_{k(t)}<t$ be the ordered points of $\bigcup_{v \in \Lambda} \mathcal T_v$ at times before $t$, 
    \begin{itemize}
        \item for $s\in [0,t_1)$, let $X_s = x_0$; 
        \item iteratively, for $i = 1,...,k(t)$, for $s\in [t_i,t_{i+1})$, if $t_i\in \mathcal T_v$, let 
        \begin{align*}
            X_s(w) & = X_{t_i^-}(w) \qquad w\ne v\\ 
            X_s(v) & = \begin{cases}+1  & \text{ if }\quad  \mathcal U_{v,t_i}\le \pi_{\Lambda}^{\eta}(\sigma_v = +1 \mid (\sigma_w)_{w\ne v} = (X_{t_i^-}(w))_{w\ne v}) \\ -1 & \text{ else}
            \end{cases}
        \end{align*}
    \end{itemize}
\end{definition}

As is standard from the monotonicity of the Ising model and its dynamics, the couplings are monotone: if $x_0 \preceq x_0'$ and the boundary conditions $\eta \preceq \eta'$ then for all $t\ge 0$, 
    $X_{\Lambda^\eta,t}^{x_0} \le X_{\Lambda^{\eta'},t}^{x_0'}$. In what follows, all Glauber chains we consider are coupled using Definitions~\ref{def:initialization-coupling}--\ref{def:dynamics-coupling} with the same Poisson processes $(\mathcal T_v)_{v\in \Z^2}$ and the same uniform random variables $\mathcal U_{v,0}^\pi,\mathcal U_{v,0}^Q$ and $\mathcal U_{v,t}$. The \emph{randomness on $\Lambda \times [0,t]$} means the intersection of the Poisson processes associated to vertices in $\Lambda$ with $[0,t]$, the uniform random variables used to generate the initialization for $v\in \Lambda$, and the uniform random variables used to generate the Glauber updates on vertices in $\Lambda$ for times in $(0,t]$.

\subsection{The scale-$k$ processes}
We first define domain enlargements of blocks. For a block $B\in \mathscr{B}_k$, let $E_r(B)$ be the set of vertices at $\ell^\infty$-distance at most $r$ from $B$, giving $E_r(B_{v,k}) = v + [ - \frac{\ell_k}{2} - r, \frac{\ell_k}{2} + r]^2$. If $B \in \mathscr{B}_k$ (so that the scale $k$ is understood contextually), then we simplify notation to write $E_{+j}(B) = E_{\ell_k j/10}(B)$. 

We define Markov chains on blocks of scale $k$ which are used to confine the spread of information, and locally equilibrate at scale $k$. 
For a block $B\in \mathscr{B}_k$, we let 
\begin{align*}
    U_t^{B,\pi} &: \text{Stationary Glauber dynamics on }E_{+4}(B)\text{ with $+$ boundary conditions }\\ 
    U_t^{B,\pi \wedge Q} & : \text{Glauber dynamics on }E_{+4}(B)\text{ with $+$ boundary conditions initialized from } x_0^B({E_{+4}(B)}^+)
\end{align*}
where 
\begin{align*}
    x_0^B (E_{+4}(B)) (u) = \begin{cases}
        U_0^{B,\pi}(u) \wedge Q(u) & u\in E_{+1}(B)\\ U_0^{B,\pi}(u)  & u\in E_{+4}(B)\setminus E_{+1}(B)
    \end{cases}\,.
\end{align*}
When we say $U_t^{B,\pi}$ is ``stationary" on domain $E_{+4}(B)$ and with $+1$ boundary, we mean that its initialization is from $\pi_{E_{+4}(B)}^{1}$. In words, $U_{t}^{B, \pi \wedge Q}$ is decreasing the initialization by taking a vertex-wise minimum with the disordered initialization $Q$, in the bulk of $E_{+4}(B)$. 
Observe that we have the ordering 
\begin{align*}
   U_t^{B,\pi \wedge Q} \le  U^{B,\pi}_t \le U_{t}^{B',\pi} \qquad \text{if }B' \subset B\,.
\end{align*}
In particular, if $B^{1} \subset B^2 \subset \cdots $ with $B^i \in \mathscr{B}_i$, then for each $t$, for all $u\in \Z^d$, 
\begin{align}\label{eq:i-to-infty-limit-of-localized-dynamics}
    X_t^\pi(u)= \lim_{i\to\infty} U_t^{B^i,\pi}(u) \qquad \text{and} \qquad X_t^{\pi \wedge Q} = \lim_{i\to\infty} U_t^{B^i, \pi \wedge Q}(u)\,,
\end{align}
where $X_t^\pi$ and $X_t^{\pi \wedge Q}$ are the infinite-volume processes started from the plus measure and the disordered initialization that we are ultimately trying to couple. 
Our focus is therefore on bounding the probability of disagreement of $U_t^{B,\pi}$ and $U_{t}^{B, \pi \wedge Q}$ at time $T_k$. 

\subsection{Auxiliary processes used in the analysis}\label{subsec:V-processes}
The other kinds of processes we need to consider live on domains that are roughly, but not quite, a scale down. Namely, we introduce a final parameter 
\begin{align}\label{eq:localization-scale}
    \text{Localization scale }s_k:  \qquad s_k = \log(\ell_{k+3}) = (\log (\ell_k))^{M^3}\,.
\end{align}

\begin{definition}\label{def:R-domains}
    For $ B\in \mathscr{B}_k$, define $\mathscr{R}_{k-1}( B)$ to be the set of square subsets of $E_{+2}( B)$ of side-length between $\ell_{k-1}$ and $100 s_k \ell_{k-1}$ made from unions of blocks in $\mathscr{B}_{k-1}$. 
\end{definition}

Since $R\in \mathscr{R}_{k-1}( B)$ is effectively at the $(k-1)$-scale (dilated by a polylogarithmic factor but not more), we understand $E_{+j}(R)$ to be its enlargement in all directions by $j\ell_{k-1}/10$. 

These blocks $R\in \mathscr{R}_{k-1}( B)$ will be used to cover disagreements at the $k-1$ scale and locally cure using the mixing time input them before information leaks in from the boundary of $ B$. 

For $R \in \mathscr{R}_{k-1}(B)$ for $B\in \mathscr{B}_k$, let  
\begin{align}\label{eq:V-processes}
    V_t^{R,\pi} &: \text{Stationary Glauber dynamics on $E_{+4}(R)$ with $+$ boundary conditions} \\ 
    V_t^{B\setminus R, \pi} & : \text{Stationary Glauber dynamics on $ E_{+4}(B) \setminus R$ with $+$ boundary condition on $E_{+4}(B)^c$} \nonumber \\
    & \qquad \text{and $-$ boundary condition on $R$} \nonumber 
\end{align}

Finally, we have two processes localized to  $R$ with all-plus and all-minus initializations: 
\begin{align*}
    V_t^{R, +}: \text{Glauber dynamics on $E_{+4}(R)$ with $+$ boundary conditions and $+$ initialization at time $T_{k-1}$} \\ 
    V_t^{R, -}: \text{Glauber dynamics on $E_{+4}(R)$ with $+$ boundary conditions and $-$ initialization at time $T_{k-1}$} 
\end{align*}

To try to sum up the process notations, we use $X_t$ for processes on the full $\Z^d$  domain, $U$ for processes on scale $k$, and $V$ for processes used in the analysis to bridge between scale $k-1$ and $k$.

\subsection{Good events which together ensure coupling}
We work towards constructing a dominating set $\tilde {\mathcal D}$ that confines the disagreement locations where $X_t^{\pi\wedge Q}$ and $X_t^\pi$ differ.  
The dominating process will take value $1$ on $B\in \mathscr{B}_k$ if one of several ``bad" events happen interior to $E_{+4}(B)$ on times $[0,T_k]$ and will take value $0$ if none of the ``bad" events happen (as that will imply $X_{T_k}^{\pi\wedge Q}(B) = X_{T_k}^\pi(B)$).  

For a block $B\in \mathscr{B}_k$, we define the following ``good" events. 

\smallskip
\noindent \emph{Nested stationary processes with $+$-boundary condition agree away from their boundaries.}
\begin{align*}
    \mathsf{StatEquiv}_k(B) & = \bigcap_{B'\in \mathscr{B}_{k+1}: B\subset E_{+3}(B')} \big\{ \forall t\in [0,T_k]: U_t^{B,\pi}(E_{+3}(B)) = U_t^{B',\pi}(E_{+3}(B))\big\}\,. 
\end{align*}

\smallskip
\noindent \emph{Information does not travel atypically fast.}

\begin{definition}\label{def:info-propagating-chain}
    We say there exists an \emph{$(L,T)$-propagating chain in $A$} if there exists a sequence of clock rings $(v_i,t_i)_{i=1}^L$ (i.e., the Poisson process at $v_i$ has a point at time $t_i$) with $v_i \in A$ and $v_i\sim v_{i+1}$ for all $i$, and with  $0 \le t_1\le t_2\le ... \le t_L \le T$. 
\end{definition}

\begin{observation}\label{obs:information-propagation}
    Suppose $\Lambda\subset \mathbb Z^d$ and $\Lambda'$ is its enlargement by $L$, and there is no $(L,T)$ propagating chain in $\Lambda'$. If $X_t, Y_t$ are coupled Glauber chains on some $D \supset \Lambda'$ with $X_0(\Lambda') = Y_0(\Lambda')$, then $X_T(\Lambda) = Y_T(\Lambda)$. 
\end{observation}

We can then define the following event that information does not propagate faster than linearly: 
\begin{align*}
    \mathsf{InfProp}_k(B) = \{\text{no $(\ell_k/10, T_k)$-propagating chain in $E_{+4}(B)$}\}\,.
\end{align*}

\begin{figure}[t]
\centering
\begin{tikzpicture}[scale=1.3,    orangehatch/.style={pattern=north east lines, pattern color=orange}
]

\draw[fill=green!5, thick] (-4,-4) rectangle (4,4);
\node[anchor=south east] at (3.8,-3.8) {\textbf{Large Domain} $E_{+4}(B)$};

\foreach \x in {-3.6,-3.3,...,3.6} {
    \node[scale=0.6, blue] at (\x, 3.9) {$+$}; 
    \node[scale=0.6, blue] at (\x, -3.9) {$+$}; 
    \node[scale=0.6, blue] at (3.9, \x) {$+$}; 
    \node[scale=0.6, blue] at (-3.9, \x) {$+$}; 
}

\node[blue, anchor=south] at (0, 4) {Outer Boundary $\partial E_{+4}(B)$ is All-Plus ($+$)};

\draw[fill = red!5, thick] (-0.8,-0.8) rectangle (0.8,0.8);
\node[fill=white, inner sep=1pt, text=red] at (0,0) {\textbf{Bad} $R$};

\foreach \x in {-0.4,0,0.4} {
    \node[scale=0.6, red] at (\x, 0.65) {$-$}; 
    \node[scale=0.6, red] at (\x, -0.65) {$-$}; 
    \node[scale=0.6, red] at (0.65, \x) {$-$}; 
    \node[scale=0.6, red] at (-0.65, \x) {$-$}; 
}

\draw[dashed, thick, purple] (-2.5,-2.5) rectangle (2.5,2.5);
\node[purple, anchor=north west] at (-2.5, 2.9) {Local chain is on $E_{+4}(R)$};

    \fill[orangehatch, even odd rule] 
                (-1.8, -1.8) rectangle (1.8, 1.8)
                (-1.4, -1.4) rectangle (1.4,1.4);
                
\node[gray, scale=0.8] at (0, 1.95) {Buffer Zone $E_{+2}(R)\setminus E_{+1}(R)$};

\end{tikzpicture}
\caption{The different regions used in the proof to construct sandwiching dynamics that cure a bad region $R$ at scale $k-1$ in time $t_k$.}\label{fig:sandwich-event}
\end{figure}
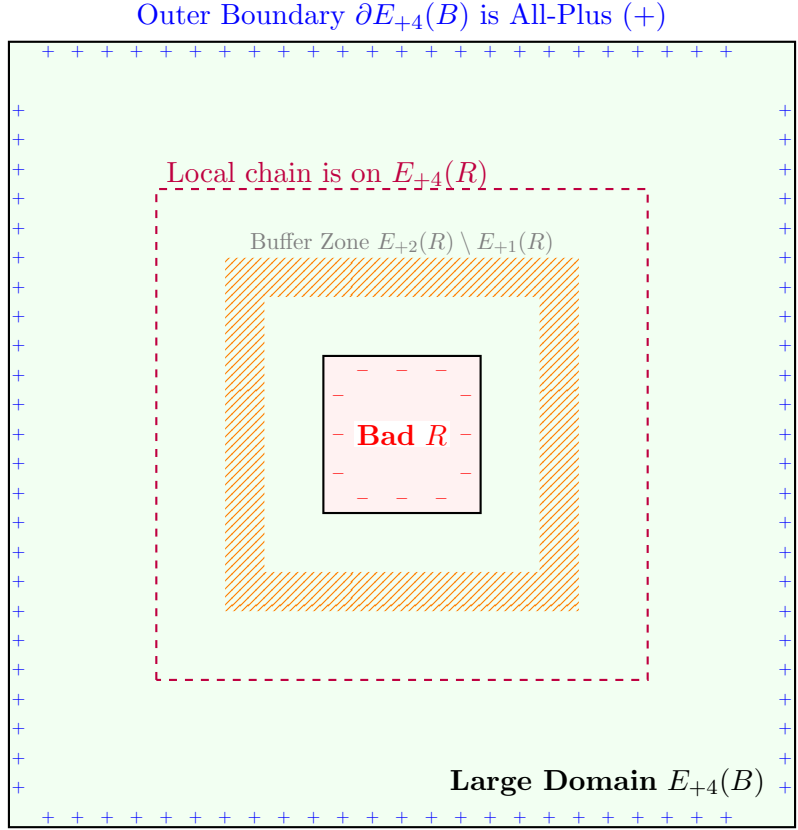

\smallskip
\noindent \emph{Locally coupling}. 
Recall the $V$-processes from~\eqref{eq:V-processes}. For a local region $R\in \mathscr{R}_{k-1}(B)$, define  
\begin{align}\label{eq:sandwich-B-R}
    \mathsf{Sandwich}(B,R) & = \{ V_{T_k}^{R,-} \equiv V_{T_k}^{R,+}\} \cap \\ 
    &  \qquad \qquad  \bigcap_{t\in [T_{k-1},T_k]} \{ V_t^{R, \pi}(E_{+2}(R) \setminus E_{+1}(R)) \equiv V_t^{B\setminus R,\pi}(E_{+2}(R) \setminus E_{+1}(R))\}\,. \nonumber 
\end{align}
This sandwiching event is the principal mechanism for locally curing bad regions at scale $k-1$ and moving up in scale: see Figure~\ref{fig:sandwich-event} for a depiction. 
We then say the local coupling event holds if $\mathsf{Sandwich}(B,R)$ holds for all $R \in \mathscr{R}_{k-1}(B)$ which ensures all discrepancies at the $k-1$-scale interior to $B$ got resolved by the local sandwiching chains: 
\begin{align}\label{eq:loc-coup}
    \mathsf{LocCoup}_k(B) = \bigcap_{R\in \mathscr{R}_{k-1}(B)} \mathsf{Sandwich}(B,R)\,.
\end{align}

\begin{lemma}\label{lem:events-are-measurable}
Both of $\mathsf{InfProp}_k(B)$ and $\mathsf{LocCoup}_k(B)$ are measurable with respect to the randomness on $E_{+4}(B) \times [0,T_k]$. If $B' \in \mathscr{B}_{k-1}$ is such that $B' \subset E_{+3}(B)$, then $\mathsf{StatEquiv}_{k-1}(B')$ is also measurable with respect to the randomness on $E_{+4}(B) \times [0,T_k]$. 
\end{lemma}
\begin{proof}
    The event $\mathsf{InfProp}_k(B)$ is evidently measurable simply with respect to the Poisson processes restricted to $E_{+4}(B)\times [0,T_k]$. 
    
    For each $R\in \mathscr{R}_{k-1}(B)$, since the chains $V_{t}^{R,-}, V_{t}^{R,+}, V_{t}^{R, \pi}$ and $V_{t}^{B\setminus R,\pi}$ from~\eqref{eq:V-processes} have frozen boundary conditions on all of $\mathbb Z^d \setminus E_{+4}(B)$, their initial states are measurable with respect to the initial randomness in $E_{+4}(B)$. Their evolutions are also then measurable with respect to the randomness on $E_{+4}(B) \times [0,T_k]$. As such, $\mathsf{LocCoup}_k(B)$ is also measurable with respect to the randomness on $E_{+4}(B)\times [0,T_k]$. 

    Since $B' \subset E_{+3}(B)$ is such that $E_{+4}(B') \subset E_{+4}(B)$, and the processes in $\mathsf{StatEquiv}_{k-1}(B')$ have boundary conditions outside $E_{+4}(B')$, we get the claim for $\mathsf{StatEquiv}_{k-1}(B')$. 
\end{proof}

\subsection{Dominating the disagreement set}

We now can construct $\widetilde{D}_k: \mathscr{B}_k \to \{0,1\}$ a percolation process on the level-$k$ blocks that will stochastically dominating the disagreement set $\{v: X_{t}^{\pi\wedge Q} (v) \ne X_t^\pi(v)\}$ as follows. 

\begin{definition}\label{def:Dtilde}
    For $B \in \mathscr{B}_0$, let $$\widetilde{D}_0(B) = \mathbf 1\Big\{\bigcup_{v\in B} Q(v) = -1\Big\}\,.$$ Now suppose we have defined $(\widetilde{D}_{k-1}(B'))_{B'\in \mathscr{B}_{k-1}}$; for each $B\in \mathscr{B}_k$, we will describe how to assign $\widetilde{D}_k(B)$.  Define 
    \begin{align}
        \mathsf{Dis}_{k-1}(B) & = \{B' \in \mathscr{B}_{k-1}: B' \subset E_{+2}(B)\,,\,\widetilde D_{k-1}(B') = 1\} \label{eq:Dis-set}\,, \\ 
        \mathsf{Bad_{k-1}}(B)& = \{B' \in \mathscr{B}_{k-1}: B' \subset E_{+3}(B)\,,\, \mathsf{StatEquiv}_{k-1}(B')^c \text{ or } \mathsf{InfProp}_{k-1}(B')^c \text{ hold}\}\,. \label{eq:Bad-set}
    \end{align}
    ($\mathsf{Dis}_{k-1}$ is roughly the set of blocks one scale down that were $1$ in the dominating set, and $\mathsf{Bad}_{k-1}$ are those one scale down that had very atypical events happen for them.) 
    We then set 
    \begin{align*}
        \widetilde{D}_k(B) = 1 - \mathbf 1\{|\mathsf{Dis}_{k-1}(B)| \le s_k\,,\, \mathsf{Bad}_{k-1}(B) =\emptyset\,,\, \mathsf{InfProp}_k(B), \mathsf{LocCoup}_k(B)\}\,.
    \end{align*}
\end{definition}

\begin{cor}\label{cor:measurable}[Corollary of Lemma~\ref{lem:events-are-measurable}]
        The random variable $\widetilde{D}_k(B)$ is measurable with respect to the randomness on $E_{+4}(B) \times [0,T_k]$. 
\end{cor}

\begin{proof}
    Assume inductively that this is true for $k-1$. Then, since for $B' \in \mathscr{B}_{k-1}: B' \subset E_{+2}(B)$, one has $E_{+4}(B') \subset E_{+4}(B)$ and $T_{k-1}\le T_{k}$, we get $\mathsf{Dis}_{k-1}(B)$ is measurable with respect to that randomness. The remaining constituent events of $\{\widetilde{D}_k(B) = 1\}$ were shown to be measurable with respect to the randomness in $E_{+4}(B)$ in Lemma~\ref{lem:events-are-measurable}. 
\end{proof}

The following proposition justifies calling this a dominating set for the disagreement set of blocks, as it says that any block in which there is a disagreement must have $\widetilde D_k(B)=1$. In particular, if $\widetilde{D}_k(B) =0$ then the original processes (disordered initialization and stationary plus phase) are coupled on $B$. Most of the rest of this section will be spent proving this using properties of the couplings between the different processes involved. 

\begin{prop}\label{prop:Dtilde-implies-coupling}
    The event $\{\widetilde{D}_k(B) =0\}$ implies that $U_{T_k}^{B,\pi \wedge Q}(B) =  U_{T_k}^{B,\pi}(B)$. If moreover 
    \begin{align*}
        \mathsf{StatEquiv}_{\ge k}(B) & = \bigcap_{k'\ge k} \bigcap_{B' \in \mathscr{B}_{k'}: B\subset B'} \mathsf{StatEquiv}(B')\,.
    \end{align*}
    holds, then we have that $X_{T_k}^{\pi \wedge Q}(B) = X_{T_k}^{\pi}(B)$. 
\end{prop}

Given Proposition~\ref{prop:Dtilde-implies-coupling}, the rest of the proof of Theorem~\ref{thm:main-general} will go by showing that this dominating set behaves like a very subcritical percolation process whose probability of being $1$ decreases appropriately with the scale. 

\begin{prop}\label{prop:Dtilde-probability-bound}
Under Assumption~\ref{assump:uniform-mixing-time}, there exist constants $M_0,\ell_0(M), \epsilon_0$ (only depending on $C,d$) such that for all $\beta>\beta_0$, all $M\ge M_0$, all $p_0 \ge 1-\epsilon_0$,  all $k$ and $B\in \mathscr{B}_k$, we have $\mathbb P(\widetilde{D}_k(B) =1) \le q_k$ for $q_k$ from~\eqref{eq:disagreeent-probability}.  
\end{prop}

Since $\widetilde D_k(B)=1$ includes the event of $\mathsf{StatEquiv}_{k-1}(B')$  a step of the proof (and a union bound) also gives the following lower bound on the probability of $\mathsf{StatEquiv}_{\ge k}(B)$. 

\begin{cor}\label{cor:statequiv-atleastk-bound}
    There exists $\beta_0(d)<\infty$ such that for every $\beta>\beta_0$, every $k$ and $B \in \mathscr{B}_k$, 
    \begin{align*}
        \mathbb P(\mathsf{StatEquiv}_{\ge k}(B)^c) \le \frac{q_k}{10}\,.
    \end{align*}
\end{cor}

The proof of Proposition~\ref{prop:Dtilde-probability-bound} is deferred to the following section where we provide bounds on the probabilities of the complements of each of the good events defined above. The rest of this section is focused on proving  Proposition~\ref{prop:Dtilde-implies-coupling}. We begin with a series of preliminary lemmas using the good events to locally couple chains from different initializations. The first claim will be used to give us a good cover of $\mathsf{Dis}_{k-1}(B)$ by squares of side-length that is not much larger than $\ell_{k-1}$. Recall from Definition~\ref{def:R-domains}. that for $B\in \mathscr{B}_k$, the set $\mathscr{R}_{k-1}$ is the set of squares internal to $E_{+2}(B)$ consisting of unions of blocks of $\mathscr{B}_{k-1}$ and having side-length at most $100 s_k \ell_{k-1}$. 

\begin{claim}\label{cl:covering}
Fix $B \in \mathscr{B}_k$. If $D\subset \mathscr{B}_{k-1}$ with each element of $D$ being contained in $E_{+2}(B)$, and with $|D|\le s_k$, then there is a cover of $D$ by squares $R_1,...,R_K \in \mathscr{R}_{k-1}(B)$ such that $K\le s_k$ and if $i\ne j$, then the $\mathbb Z^d$-distance between $R_i$ and $R_j$ is at least $10 \ell_{k-1}$.    
\end{claim}
\begin{proof}
    We start by defining a set of squares $R_1,...,R_K \in \mathscr{R}_{k-1}(B)$ covering $\mathsf{Dis}_{k-1}(B)$ from~\eqref{eq:Dis-set}. 
    
    \begin{definition}
        From a set $D = \mathsf{Dis}_{k-1}(B) \subset \mathscr{B}_{k-1}$ having $|\mathsf{Dis}_{k-1}(B)|\le s_k$, we construct $R_1,...,R_K$ as follows. Start with candidate set $\mathcal R = \{B': B'\in \mathsf{Dis}_{k-1}(B)\}$;  repeat the following process
        \begin{itemize}
            \item If the current candidate $\mathcal R$ contains two squares $R, R'$ of side-length $a \ell_{k-1}$ and $b \ell_{k-1}$ within distance $10 \ell_{k-1}$, then replace $R,R'$ in $\mathcal R$ with $R''$ which is a smallest square covering both squares.  
        \end{itemize}
    \end{definition}

    Evidently, in each iteration of the above process, $R''$ will have side-length at most $(a + b+ 10)\ell_{k-1}$. Also, since all the constituent blocks of $D$ are in $E_{+2}(B)$, so will all the $R''$ throughout the process. Thus, the sum of all sidelengths of squares in $\mathcal R$ will increase by at most $10\ell_{k-1}$ while the size of $\mathcal R$ will decrease by $1$. When the process terminates, the sum of all side-lengths of regions in $\mathcal R$ will be at most 
    \begin{align*}
        s_k \ell_{k-1} + 10 \ell_{k-1} (|\mathsf{Dis}_{k-1}(B)|-1) \le 10 \ell_{k-1} s_k\,.
    \end{align*}
    Recalling the definition of $\mathscr{R}_{k-1}(B)$, this ensures that any $R\in \mathcal R$ is in $\mathscr{R}_{k-1}(B)$, and also any two $R_{i},R_j \in \mathcal R$ will be at distance at least $10 \ell_{k-1}$. 
\end{proof}

The next is an observation that if two chains agree on a buffer separating $R_i$ from $R_j$ for all $i\ne j$ and a disagreement on $R_i$ gets cured in a certain amount of time $T$, then disagreements on all of $(R_i)_{i=1}^K$ get cured simultaneously. In what follows, $R_1,...,R_K$ are the output of Claim~\ref{cl:covering} so for $i\ne j$, one has $\mathbb Z^d$-distance at least $3$ between $E_{+2}(R_i)$ and $E_{+2}(R_j)$. 

\begin{observation}\label{obs:local-sandwiching}
    Suppose $(U_t)_{t\in [T_{k-1},T_k]}$ is Glauber dynamics on a domain $E_{+4}(B)$ with $+$ boundary conditions and initial condition $U_{T_{k-1}}$. Let $W_t^{(i)}$ be coupled to $U_t$ by the grand coupling, but with initialization at time $T_{k-1}$ given by 
    \begin{align*}
        W_{T_{k-1}}^{(i)}(u) = \begin{cases}
            -1 & u\in R_i \\ U_{T_{k-1}}(u) & u\notin R_i
        \end{cases}\,.
    \end{align*}
    Define the local sandwiching event  
    \begin{align*}
        \mathsf{Sandwich}(W^{(i)}) = \{W_{T_k}^{(i)} = U_{T_k}\}\cap \bigcap_{t\in [T_{k-1},T_k]} \{W_t^{(i)} (E_{+4}(B) \setminus E_{+2}(R_i)) = U_{t}(E_{+4}(B) \setminus E_{+2}(R_i))\}\,.
    \end{align*}
    Then on the event $\bigcap_{i=1}^K \mathsf{Sandwich}(W^{(i)})$ the coupled dynamics $W^{\min}_t$ initialized from $W^{\min}_{T_{k-1}} = \bigwedge_{i=1}^{K} W_{T_{k-1}}^{(i)}$ achieves $W^{\min}_{T_k} = U_{T_k}$ everywhere on $E_{+4}(B)$. 
\end{observation}

\begin{proof}
    We will show inductively in time, that 
    $$W_{t}^{\min}(E_{+4}(B) \setminus \bigcup_{i\le K} E_{+2}(R_i) ) = U_{t}(E_{+4}(B)\setminus \bigcup_{i\le K}E_{+2}(R_i))
    \qquad t\in [T_{k-1},T_k]$$
    and 
    $$W_{t}^{\min}(E_{+2}(R_i)) = W_t^{(i)} (E_{+2}(R_i)) \qquad t\in [T_{k-1},T_k]\,.$$
    These hold at time $t=T_{k-1}$ by construction of the initializations and the fact that $E_{+2}(R_i) \cap E_{+2}(R_j) = \emptyset$ for $i\ne j$. Suppose it holds at $t^-$ and at time $t$ there is an update at vertex $v$. 
    \begin{itemize}
    \item If $v$ is at graph distance at least $2$ from $\bigcup_{i} R_i$, then all its neighbors are the same in $W^{\min}_{t^-}$ and in $U_{t^-}$ so the update is coupled perfectly. 
    \item If $v$ is at distance one from the boundary of $E_{+2}(R_i)$, then by $\mathsf{Sandwich}(W^{(i)})$, one has  $W_{t^-}^{(i)}(w) = U_{t^-}(w)$ for $w\sim v: w\notin E_{+2}(R_i)$ and by inductive assumption that is also equal to $W_{t^-}^{\min}(w)$; for $w\sim v: w\in E_{+2}(R_i)$. At the same time, one has $W_{t^-}^{\min}(w) = W_{t^-}^{(i)}(w)$ by the second inductive assumption. 
    
    Thus, one perfectly couples $W_{t}^{\min}(v) = W_t^{(i)}(v)$. If $v\in E_{+2}(R_i)$ then that is what we want to show at time $t$. If $v\notin E_{+2}(R_i)$, by $\mathsf{Sandwich}(W^{(i)})$ it is also equal to $U_{t}(v)$ as wanted. 
    \item If $v\in E_{+2}(R_i)$ and all its neighbors are in $E_{+2}(R_i)$, then the second inductive assumption implies its neighbors are the same in $W_{t^-}^{\min}$ and in $W_t^{(i)}$ so we retain $W_{t}^{\min}(v)= W_t^{(i)}(v)$. 
    \end{itemize}
    Having established the above two equalities, adding in the first part of $\mathsf{Sandwich}(W^{(i)})$ for every $i$ says that $W_{T_k}^{(i)}(E_{+2}(R_i)) = U_{T_k}(E_{+2}(R_i))$  so altogether, $W_{T_k}^{\min} = U_{T_k}$ in all of $E_{+4}(B)$. 
\end{proof}

The last lemma we need will show how to go from $\mathsf{LocCoup}_k(B)$ of~\eqref{eq:loc-coup} with the bounding $V$ processes, to the $\mathsf{Sandwich}(W^{(i)})$ event of the above observation for $i=1,...,K$. 

\begin{lem}\label{lem:sandwiching-couplings}
    Fix $B\in \mathscr{B}_k$ and $R_i \in \mathscr{R}_{k-1}(B)$, and history $\mathcal F_{T_{k-1}}$. Let $U_t,W_t^{(i)}$ be as in Observation~\ref{obs:local-sandwiching} with initializations $U_{T_{k-1}} = U_{T_{k-1}}^{B,\pi}$ (and with $-1$ in $R_i$ respectively). If $\mathsf{Sandwich}(B,R_i)$ from~\eqref{eq:sandwich-B-R} holds, then so does $\mathsf{Sandwich}(W^{(i)})$. 
\end{lem}
\begin{proof}
    For ease of notation, drop $i$ sub and superscripts. Recall the bounding $V$ processes from~\eqref{eq:V-processes} out of which $\mathsf{Sandwich}(B,R_i)$ is defined. First observe 
    \begin{align}\label{eq:W-le-U-le-V}
        W_t(E_{+4}(R)) \le U_t(E_{+4}(R)) \le V_t^{R,\pi}(E_{+4}(R)) \qquad t\in [T_{k-1},T_k]\,,
    \end{align}
    where the first inequality is because it has a more minus initialization at time $T_{k-1}$, and the second is because it has closer plus boundary conditions at $E_{+4}(R)$ (causing both $V_{T_{k-1}}^{R,\pi}\ge U_{T_{k-1}}$ and that inequality to persist). At the same time, 
    \begin{align}\label{eq:V-le-W-le-U}
        V_t^{B\setminus R,\pi}(E_{+4}(B)) \le  W_t(E_{+4}(B))\le U_{t}(E_{+4}(B))\,,
    \end{align}
    where the first inequality is because $V_{T_{k-1}}^{B\setminus R,\pi} \le W_{T_{k-1}}$ by monotonicity, and the fact that it keeps the minuses on $R$ frozen. On the event $\mathsf{Sandwich}(B,R)$ we have for all $t\in [T_{k-1},T_k]$ that 
    \begin{align}
        V_{t}^{R,\pi}(E_{+2}(R)\setminus E_{+1}(R)) = V_{t}^{B\setminus R,\pi}(E_{+2}(R) \setminus E_{+1}(R))\,,
    \end{align}
    so by sandwiching, in fact, for all $t\in [T_{k-1},T_k]$
    \begin{align}\label{eq:W-U-equal-on-annulus}
        V_{t}^{B\setminus R,\pi}(v) = W_t(v) = U_t(v) = V_t^{R,\pi}(v) \qquad \forall v\in E_{+2}(R)\setminus E_{+1}(R)\,.
    \end{align}
    This in particular implies 
    \begin{align*}
        \{\forall t\in [T_{k-1},T_k] : W_t(E_{+4}(B) \setminus E_{+1}(R)) = U_t(E_{+4}(B)\setminus E_{+1}(R))\}\,,
    \end{align*}
    since for all times between $T_{k-1},T_k$, any update outside $E_{+2}(R)$ will have the same neighbors under $W_{t^-}$ and $U_{t^-}$. This is exactly the second constituent event of $\mathsf{Sandwich}(W^{(i)})$. 
    
    It remains to show that $\{W_{T_k} = U_{T_k}\}$ on $E_{+4}(B)$. By monotonicity, for $t\in [T_{k-1},T_k]$, 
$$V_t^{R,-}(E_{+4}(R)) \le V_{t}^{R,\pi}(E_{+4}(R)) \le V_t^{R,+}(E_{+4}(R))\,.$$ Therefore, on $\mathsf{Sandwich}(B,R)$, by~\eqref{eq:W-U-equal-on-annulus}, for $t\in [T_{k-1},T_k]$, 
\begin{align*}
    V_t^{R,-}(E_{+2}(R)\setminus E_{+1}(R)) \le V_t^{R,\pi}(E_{+2}(R)\setminus E_{+1}(R)) = W_t(E_{+2}(R) \setminus E_{+1}(R))\,.
\end{align*}
Since also at time $T_{k-1}$, one has $V_{T_{k-1}}^{R,-}(E_{+2}(R))\le W_{T_{k-1}}(E_{+2}(R))$ the monotonicity of the coupling implies that 
\begin{align}\label{eq:V-le-W-all-time}
    V_{t}^{R,-}(E_{+2}(R)) \le W_t(E_{+2}(R)) \qquad \forall t\in [T_{k-1},T_k]\,.
\end{align}
This allows us to sandwich $W$ and $U$ by the local-scale Markov chains $V_t^{R,-},V_t^{R,+}$ which couple in the mixing time of $R$ (assumed in the first part of $\mathsf{Sandwich}(B,R)$). Namely, we therefore get 
\begin{align*}
    V_{t}^{R,-}(E_{+2}(R)) \le W_{T_k}(E_{+2}(R)) \le U_{T_k}(E_{+2}(R))\le V_{T_k}^{R,\pi}(E_{+2}(R)) \le V_{T_k}^{R,+}(E_{+2}(R))\,,
\end{align*}
where the first inequality was~\eqref{eq:V-le-W-all-time}, the second and third~\eqref{eq:W-le-U-le-V} and the last monotonicity of initialization. The first part of $\mathsf{Sandwich}(B,R)$ thus ensures the left and right sides equal each other, so that indeed $W_{T_k}(E_{+2}(R)) = U_{T_k}(E_{+2}(R))$. Since they agreed on $E_{+4}(B) \setminus E_{+1}(R)$ for all $t\in [T_{k-1},T_k]$ by~\eqref{eq:W-U-equal-on-annulus}, we deduce that $W_{T_k} = U_{T_k}$ on all of $E_{+4}(B)$. 
\end{proof}

We now combine all the ingredients to establish the main Proposition~\ref{prop:Dtilde-implies-coupling}. 

\begin{proof}[\textbf{\emph{Proof of Proposition~\ref{prop:Dtilde-implies-coupling}}}]
        We prove inductively over $k$ that for all $B\in \mathscr{B}_k$, we have $\{\widetilde D_k(B)=0\}$ implies $U_{T_{k}}^{B,\pi \wedge Q}(B) = U_{T_k}^{B,\pi}(B)$. The base case holds because $T_0 = 0$ and $\widetilde D_0$ is $1$ anywhere $Q= -1$. Now assume it holds for $k-1$, and show it holds for $k$.

        Fix $B\in \mathscr{B}_k$. Introduce a new process $\widetilde{U}^{B,\pi}_t$ as the variant of $U_{t}^{B, \pi}$ that at time $T_{k-1}$ sets all vertices in $\mathsf{Dis}_{k-1}(B)$ to minus $-1$, and otherwise evolves according to the coupled Glauber dynamics. Evidently, 
        \begin{align}
        \widetilde{U}_{t}^{B,\pi} \le U_{t}^{B,\pi} \qquad t\in [T_{k-1},T_k]\,.
        \end{align}
        If $\widetilde D_{k}(B) = 0$, then $|\mathsf{Dis}_{k-1}(B)|\le s_k$, so we can use Claim~\ref{cl:covering} to construct a cover $R_1,...,R_K\in \mathscr{R}_{k-1}(B)$ with $K\le s_k$ and distance at least $10\ell_{k-1}$ between $R_i, R_j$. We can then define the process $W_{t}^{\min}$ which is initialized at time $T_{k-1}$ from $U_{T_{k-1}}^{B,\pi}$ but with minuses on all $(R_i)_{i=1}^K$. 
        
        Since $\widetilde D_{k}(B) = 0$, $\mathsf{LocCoup}_k(B)$ also holds, which means for all such $R_i$, we have $\mathsf{Sandwich}(B,R_i)$. Lemma~\ref{lem:sandwiching-couplings} implies that in fact we get $\mathsf{Sandwich}(W^{(i)})$ for all $i=1,...,K$. Observation~\ref{obs:local-sandwiching} then implies $$W_{T_k}^{\min}(E_{+4}(B)) = U_{T_k}^{B,\pi}(E_{+4}(B)) \qquad \text{and} \qquad W_t^{\min}(E_{+4}(B)\setminus E_{+2.5}(B)) = U_t^{B,\pi}(E_{+4}(B)\setminus E_{+2.5}(B))\,,$$
        for all $t\in [T_{k-1},T_k]$ (as every $E_{+2}(R_i) \subset E_{+2.5}(B)$). Since the tilde-process $\widetilde U_t^{B,\pi}$ is sandwiched between these processes everywhere, it implies 
        \begin{align}\label{eq:Utilde-equals-not-tilde}
            \widetilde U_{T_k}^{B,\pi}(E_{+4}(B)) = U_{T_k}^{B,\pi}(E_{+4}(B))\,.
        \end{align}

        Now for each $B' \in \mathscr{B}_{k-1}: B' \subset E_{+3}(B)$, since $\mathsf{Bad}_{k-1}(B) = \emptyset$, we have that $\mathsf{StatEquiv}_{k-1}(B')$ holds, and therefore 
        \begin{align*}
            U_{T_{k-1}}^{B',\pi}(B') = U_{T_{k-1}}^{B,\pi}(B')\,.
        \end{align*}
        At the same time, 
        \begin{align}\label{eq:inductive-step-that-uses-Q}
            U_{0}^{B',\pi \wedge Q}(E_{+1}(B')) = U_0^{B,\pi \wedge Q}(E_{+1}(B'))\,,
        \end{align}
        and therefore $\mathsf{InfProp}_{k-1}(B')$ holding implies by Observation~\ref{obs:information-propagation}, 
        \begin{align}\label{eq:inductive-step-that-uses-Q-2}
            U_{T_{k-1}}^{B',\pi \wedge Q}(B') = U_{T_{k-1}}^{B,\pi \wedge Q}(B')\,.
        \end{align}
        By the inductive hypothesis, the above imply that for every $B' \in \mathscr{B}_{k-1}$ such that $B' \subset E_{+3}(B)$ having $\widetilde{D}_{k-1}(B')=0$, one has 
        \begin{align*}
            U_{T_{k-1}}^{B,\pi \wedge Q}(B') = U_{T_{k-1}}^{B,\pi}(B')\,.
        \end{align*}
        Since everywhere $\widetilde D_{k-1}(B')= 1$, we covered by minuses at time $T_{k-1}$ in $\widetilde U$, we therefore have 
        \begin{align*}
            \widetilde U_{T_{k-1}}^{B,\pi} (E_{+2}(B)) \le U_{T_{k-1}}^{B,\pi \wedge Q}(E_{+2}(B)) \le U_{T_{k-1}}^{B,\pi}(E_{+2}(B))\,. 
        \end{align*}
        At the same time, by $\mathsf{InfProp}_k(B)$, no information propagates from $E_{+2}(B)^c$ to $E_{+1}(B)$ in time $t\in [T_{k-1},T_k]$ so this ordering is retained in $E_{+1}(B)$ for that period of time, and 
        \begin{align*}
            \widetilde U_{T_{k}}^{B,\pi} (E_{+1}(B)) \le U_{T_{k}}^{B,\pi \wedge Q}(E_{+1}(B)) \le U_{T_{k}}^{B,\pi}(E_{+1}(B))\,. 
        \end{align*}
        Combined with~\eqref{eq:Utilde-equals-not-tilde}, we conclude that $\widetilde D_k(B) = 0$ implies the desired 
        \begin{align}\label{eq:U-couple-on-B}
            U_{T_k}^{B,\pi \wedge Q} (E_{+1}(B)) = U_{T_k}^{B,\pi}(E_{+1}(B))\,.
        \end{align}

        We now establish the moreover statement. On the event $\mathsf{StatEquiv}_{\ge k}(B)$, if $(B_{\ell})_{\ell \ge k+1}$ are a sequence in $\mathscr{B}_\ell$ with $B \subset B_{\ell}$ for all $\ell$, then 
        \begin{align*}
            X_0^\pi (E_{+3}(B)) &  = \lim_{\ell \to\infty} U_0^{B_\ell,\pi}(E_{+3}(B)) = U_0^{B,\pi}(E_{+3}(B))\,. \\ 
            X_0^{\pi \wedge Q}(E_{+3}(B))  & = \lim_{\ell \to\infty} U_0^{B_\ell, \pi \wedge Q}(E_{+3}(B)) = U_0^{B,\pi \wedge Q}(E_{+3}(B))\,.
        \end{align*}
        Using $\mathsf{InfProp}_k(B)$, which holds because $\widetilde D_k(B) =0$, the exterior of $E_{+3}(B)$ does not affect the configuration in $E_{+1}(B)$ at time $t\le T_k$ in any Glauber dynamics chains, so 
        \begin{align*}
            X_{T_k}^{\pi \wedge Q}(E_{+1}(B)) = U_{T_k}^{B,\pi \wedge Q}(E_{+1}(B)) = U_{T_k}^{B,\pi}(E_{+1}(B)) = X_{T_k}^\pi(E_{+1}(B))
        \end{align*}
        where the middle equality used~\eqref{eq:U-couple-on-B}. 
\end{proof}

\subsection{Proof of main theorem assuming Proposition~\ref{prop:Dtilde-probability-bound}}
Given Proposition~\ref{prop:Dtilde-implies-coupling}, if we assume the probability bound Proposition~\ref{prop:Dtilde-probability-bound} (which will be proved in the following section), we can conclude Theorem~\ref{thm:main-general}.

\begin{proof}[\textbf{\emph{Proof of Theorem~\ref{thm:main-general}: infinite volume}}]
    For each $k$, and $B \in \mathscr{B}_k$, we have by Proposition~\ref{prop:Dtilde-probability-bound} (and its Corollary~\ref{cor:statequiv-atleastk-bound}) that 
    \begin{align*}
        \mathbb P(\widetilde D_k(B) = 1) + \mathbb P(\mathsf{StatEquiv}_{\ge k}(B)^c) \le q_k + \sum_{r \ge k} q_r \le \frac{2}{\ell_{k+3}}\,.
    \end{align*}
    Together with a union bound over $k \ge k_0$ where $k_0$ is large enough that $\Lambda \subset B$ for $B \in \mathscr{B}_{k_0}$, we have by Proposition~\ref{prop:Dtilde-implies-coupling} that 
    \begin{align*}
        \mathbb P(X_{t}^{\pi \wedge Q}(\Lambda) \ne X_{t}^\pi(\Lambda)) \le O( e^{ - (\log (M t))^{3M}})\,,
    \end{align*}
    so long as $t$ is bigger than $T_{k_0}$. Since $T_{k_0}$ is comparable to  $\diam(\Lambda)$, the above bound applies for $t = \Omega( \diam(\Lambda))$. 
    Since this is under the grand coupling, we have  
    \begin{align*}
        X_{t}^{\pi \wedge Q}(\Lambda) \le X_t^{Q}(\Lambda) \le X_t^+(\Lambda)\,,
    \end{align*}
    and it follows from dynamically easier arguments (see e.g.,~\cite[Corollary 1.5]{GhSi22}---most of the work in that paper was having the result down to criticality in $d\ge 3$, which is not as relevant to us) that as long as $\beta>\beta_c$ under Assumption~\ref{assump:uniform-mixing-time}, 
    \begin{align*}
        \mathbb P(X_t^+(\Lambda) \ne X_{t}^\pi(\Lambda)) \le \, |\Lambda| \exp( - e^{\sqrt{\log t}/C})\,.
    \end{align*}
    Combining these two, by Markov's inequality, we have that except with probability $e^{ - \Omega( \log t)^{3M})}$, $x_0 \sim Q$ is such that
    \begin{align*}
        \mathbb P(X_t^{x_0} (\Lambda) \ne X_t^\pi(\Lambda) ) \le \mathbb P( X_t^{\pi \wedge x_0}(\Lambda)\ne X_t^+(\Lambda))  \le e^{ - \Omega((\log t)^{3M})}\,,
    \end{align*}
    for $t \ge \Omega(\diam(\Lambda))$. Taking $\Lambda = \Lambda_R$ and noting that the probability of bad $x_0 \sim Q$ goes to $0$ as $R\to \infty$ because $t \ge \Omega(R)$, we conclude. 
\end{proof}

\subsection{The finite domain $(\mathbb Z/n\mathbb Z)^d$ case} 

In this subsection, instead of infinite-volume, let $X_t^{\pi \wedge Q}, X_t^{\pi}$ denote the Glauber dynamics on finite volume $\mathbb T_n^d$, with $\pi$ representing $\pi^+$, the stationary distribution conditioned on positive magnetization. Here, we have to modify the above slightly for two reasons, there is a maximal spatial scale before we ``wrap around" the torus, and the process $X_{t}^{\pi}$ is not stationary.

\begin{proof}[\textbf{\emph{Proof of Theorem~\ref{thm:main-general}: finite torus}}]
        Fix a vertex $v$ in $\mathbb T_n^d = (\mathbb Z/n\mathbb Z)^d$ as the origin (by transitivity) and embed the graph in $\mathbb Z^d$ so that as long as $\ell_k < n/10$, if $B\in \mathscr{B}_r$ for $r\le k$, centered at $v$, $E_{+4}(B)$ is not distorted by this embedding. Our aim is to show that if $\pi = \pi_{\mathbb T_n^d}(\cdot \mid \Omega_+)$, where $\Omega_+$ is the set of configurations with non-negative magnetization, then under the grand coupling, 
        \begin{align}\label{eq:want-to-show-torus}
            \mathbb P(X_{T}^{\pi \wedge Q} (v) \ne X_T^{\pi}(v)) \le o\Big(\frac{1}{n^{d+10}}\Big)\,,
        \end{align}
        for $T\in [n^{o(1)},e^{ O(n^{d-1})}]$. 
        
        In order to show this, take $k$ such that $\ell_k = \exp( (\log n)^{2/3M})$. This scale is chosen such that $\ell_k$ is sub-polynomial in $n$, while $q_k = \frac{1}{\ell_{k+3}} = e^{ - (\log n)^2}$ decays faster than any polynomial in $n$. Then for $T = T_k = O(\exp((\log n)^{2/3M}))$, we have by the first part of Proposition~\ref{prop:Dtilde-implies-coupling} that 
        \begin{align*}
           \mathbb P( U_{T_k}^{B,\pi \wedge Q} (B)  \ne U_{T_k}^{B,\pi}(B)) \le n^{-10d}\,.
        \end{align*} 
        Define an extra event 
        \begin{align*}
            \mathsf{StatEquiv}^{\mathbb T}_k(B) = \{\forall t\in [0,T_k]: X_t^{\pi}(E_{+3}(B)) = U_{t}^{B,\pi}(E_{+3}(B))\}\,. 
        \end{align*}
        The proof that it has probability at least $1-q_k/10$ is nearly identical to the proof of Corollary~\ref{cor:statequiv-atleastk-bound} for $\mathsf{StatEquiv}_k(B)$ having high probability,  we defer it to Lemma~\ref{lem:Torus-statequiv}. 

        On the intersection of $\widetilde {D}_k(B) = 0$ with $\mathsf{StatEquiv}^{\mathbb T}_k(B)$, we also have  $$X_{T_k}^{\pi}(B)  = U_{T_k}^{B,\pi}(B) \qquad \text{and} \qquad X_{T_k}^{\pi \wedge Q}(B) = U_{T_k}^{B,\pi \wedge Q}(B)$$ 
        (first equality being by the torus stationary equivalence, and the second because $X_{0}^{\pi \wedge Q}( E_{+3}(B)) = U_{0}^{B,\pi \wedge Q}(E_{+3}(B))$ by torus stationary equivalence, combined with $\mathsf{InfProp}_k(B)$ ensuring that the equality persists on $B$ until time $T_k$. We therefore deduce that 
        \begin{align*}
            \mathbb P(X_{T}^{\pi}(B) \ne X_{T}^{\pi \wedge Q}(B))  \le q_k +  n^{-10d} \le n^{-9d}\,.
        \end{align*}
        Again, sandwiching $X_{t}^{\pi \wedge Q} \le X_{t}^{Q} \le X_t^+$, and using the dynamically simpler proof of convergence from the $+$-initialization (see Proposition 3.3 of~\cite{GhSi22}), under Assumption~\ref{assump:uniform-mixing-time}, 
        \begin{align*}
            \mathbb P(X_{T}^+(B)\ne X_T^{\pi}(B)) \le  \sum_{v\in B} \Big(\mathbb P(X_{T}^+(v) = +1) - \mathbb P(X_{T}^\pi(v) = +1)\Big) \le C \ell_k^d e^{ - (e^{ \sqrt{\log T}})/C} 
        \end{align*}
        which since $T \ge \Omega(\exp((\log n)^{2/3M}))$,  gives a right-hand side of $C n^d e^{ - e^{ (\log n)^{1/3M}}}$ which is smaller than any polynomial in $n$. Therefore, by the sandwiching, we also get 
        \begin{align*}
            \mathbb P(X_{T}^Q(B) \ne X_{T}^\pi(B)) \le  n^{ - 8d}\,.
        \end{align*}
        
        Finally, by Lemma~\ref{lem:torus-bottleneck} (which is just the well-known existence of a low-temperature bottleneck between $\Omega_+$ and $\Omega_-$), one has $\|X_{T}^{\pi} - \pi(\cdot \mid \Omega_+)\|_{\tv} = e^{ - \Omega(n^{d-1})}$ for all $T = \exp( o(n^{d-1}))$. Combining that with the above display, and observing that $n^{-8d} = o(n^{-d -10})$ when $d\ge 2$, we conclude~\eqref{eq:want-to-show-torus} 
        for all $T \in [T_k,e^{ o(n^{d-1})}]$ where $T_k = e^{(\log n)^{2/3M}} = n^{o(1)}$. 
\end{proof}

\section{Bounding the probabilities of bad blocks}
Our aim in this section is to prove Proposition~\ref{prop:Dtilde-probability-bound}, showing the probability of a bad event on a block at scale $k$, i.e., $\widetilde{D}_k(B) =1$, is at most $q_k$. 

\subsection{Linear speed of information propagation}
Here, we give a simple claim about the speed at which disagreements can possibly propagate, showing that with exponentially high probability, it takes order $L$ linear time for a disagreement to propagate a distance $L$. This is a standard bound that has appeared many places including~\cite{Martinelli-notes,DSVW}. 

\begin{lemma}\label{lem:information-propagation}
    There exists $M_0(d),\ell_0(d)$ such that as long as $M$ in~\eqref{eq:spatial-scale} is larger than $M_0$, then for every $k$ and every $B\in \mathscr{B}_k$, the probability 
    \begin{align*}
        \mathbb P(\mathsf{InfProp}_k(B)^c )\le  e^{ - \ell_k/10}\,.
    \end{align*}
\end{lemma}

\begin{proof}
    In order for $\mathsf{InfProp}_k(B)$ to fail, there must exist a sequence of vertices $v_1,v_2,...$ of length $\ell_k/10$ in $E_{+4}(B)$ with $v_{i} \sim v_{i+1}$ for all $i$, and associated clock rings on those vertices $t_1,t_2,...$ such that $t_1 <t_2<...$. We can obtain by a union bound that  
    \begin{align*}
       \mathbb P(\mathsf{InfProp}_k(B)^c) & \le  |E_{+4}(B)|  (2d)^{\ell_k/10} \binom{2T_k |E_{+4}(B)|}{\ell_k/10}  \frac{1}{|E_{+4}(B)|^{\ell_k/10}} + (e/4)^{T_k |E_{+4}(B)|} \,,
    \end{align*}
    where the second term comes from the probability of more than $2T_k |E_{+4}(B)|$ many clock rings in time $T_k |E_{+4}(B)|$ by Poisson Chernoff bound, and the first term is a union bound over $v_1$, a union bound over paths of length $\ell_k/10$ started from $v$, a union bound over which of the at most $2T_k |E_{+4}(B)|$ clock rings occur at those vertices, and the probability that the first of them occured at $v_1$, the next at $v_2$, and so on. 

    Then, we can bound this by 
    \begin{align*}
          \mathbb P(\mathsf{InfProp}_k(B)^c) & \le (4\ell_k)^d 4^{\ell_k/10} \exp\Big( \frac{\ell_k}{10} \log (\frac{20}{M} |E_{+4}(B)|)\Big) \exp\Big( - \frac{\ell_k}{10} \log (|E_{+4}(B)|)\Big) + e^{ - \ell_k^{d+1}/M}\,,
    \end{align*}
    which as long as $M$ is large enough (e.g., $M_0 = 100$) and $\ell_0$ is larger than a large universal constant ($10$ or so should already suffice) the first term is at most $e^{ - \ell_k/10}$. 
\end{proof}

\subsection{Rarity of disagreement regions at a scale below}
We next use show that if the probability of a bad block at the lower scale is bounded as in Proposition~\ref{prop:Dtilde-probability-bound}, then the number of bad blocks in $\mathscr{B}_{k-1}$ internal to $B$ are indeed at most $s_{k}$ from~\eqref{eq:localization-scale}. 

\begin{lem}\label{lem:rarity-of-disagreement-regions}
    There exists $M_0(d)$ such that for all $M\ge M_0$ the following holds. Suppose that $\mathbb P(\widetilde{D}_{k-1}(B') =1)\le q_{k-1}$ for all $B' \in \mathscr{B}_{k-1}$. Then for every $B \in \mathscr{B}_k$, one has 
    \begin{align*}
        \mathbb P(|\mathsf{Dis}_{k-1}(B)|>s_k ) \le q_k/10\,.
    \end{align*}
\end{lem}

\begin{proof}
    Corollary~\ref{cor:measurable} said that $\widetilde{D}_{k-1}(B')$ is measurable with respect to the randomness on $E_{+4}(B')$. The set $E_{+4}(B')$ is disjoint from $E_{+4}(B'')$ for all $B'' \in \mathscr{B}_{k-1}(B)$ that is not adjacent to $B'$ in the graph of $\mathscr{B}_{k-1}$. In order for $|\mathsf{Dis}_{k-1}(B)| \ge s_k$, it must have a subset of size at least $s_k/2d$ such that no two blocks in it are $\mathscr{B}_{k-1}$-adjacent. There are at most $\binom{(2\ell_k/\ell_{k-1})^d}{s_k/2d}$ many such possible subsets, and the probability of one of them having $\widetilde{D}_{k-1}(B')$ for all $B'$ in it, is at most $q_{k-1}^{s_k/2d}$. Therefore, plugging in for the quantities $\ell_k, q_k$ and $s_k$, 
    \begin{align*}
        \mathbb P(|\mathsf{Dis}_{k}(B)| \ge s_k ) \le \binom{2^d \ell_k^d}{s_k/2d} e^{ - s_k(\log \ell_{k})^{M^2}/2d} \le \exp\Big( \frac{(\log \ell_k)^{M^3}}{2d} \Big(2^d d \log \ell_k\ - (2d)^{-1}(\log \ell_k)^{M^2}\Big)\Big)\,,
    \end{align*}
    which so long as $\ell_k$ is larger than an absolute constant (even $2$ is enough) and $M$ is larger than a dimension-dependent constant $M \ge M_0$, satisfies that it is at most $\frac{1}{10}\exp( - (\log \ell_k)^{M^3})$. 
\end{proof}

\subsection{Equilibrium coupling estimates}

In this subsection, we will establish various equilibrium estimates showing that boxes with $+$ boundary conditions at different distances away, as well as an annulus with $-$ boundary conditions on its interior and $+$ boundary conditions on its exterior, can be coupled to the plus phase distribution except with probability decaying exponentially in the distance to the boundary. Since it is important that the spatial scale $\ell_k$ and bad-block probability $q_k$ are uniform in $\beta$, we are careful that all the bounds in this subsection are only improving as $\beta$ gets large. 

In what follows, we use the shorthand $\cdot \restriction_{A}$ as the marginal on a set $A$. The first of such bounds is a standard consequence of the Peierls bound; we provide a proof for completeness.

\begin{lem}\label{lem:stationary-plus-bc-stationary-mixing}
    Consider the Ising model with $+$ boundary conditions on a centered box of side-length $r$, denoted $\Lambda_r \subset \mathbb Z^d$. There exists $C(d)>0, \beta_0(d)<\infty$, such that for all $\beta>\beta_0$, one has  
    \begin{align*}
        \|\pi_{\Lambda_r}^+(\cdot \restriction_{\Lambda_{r- \ell}}) - \pi_{\mathbb Z^d}^+(\cdot \restriction_{\Lambda_{r- \ell}})\|_{\tv} \le C r^{d-1} e^{ - \beta \ell}\,.
    \end{align*}
\end{lem}

\begin{proof}
It is sufficient to consider coupling $\pi_{\Lambda_r}^+$ and $\pi_{\Lambda_R}^+$ for $r \le R$, as then taking $R \to\infty$ gives the above because $\lim_{R\to\infty} \pi_{\Lambda_R}^+ = \pi_{\mathbb Z^d}^+$. 

Suppose $\sigma \sim \pi_{\Lambda_R}^+$ and $\sigma'\sim \pi_{\Lambda_r}^+$ are coupled via the monotone coupling (the coupling of Definition~\ref{def:initialization-coupling}) such that $\sigma \le \sigma'$. Let $\mathcal E_{r,\ell}$ be the event that in $\sigma$, there is a minus path from $\Lambda_{r}^c$ to $\Lambda_{r-\ell}$. By exponential tails on connected minus regions (using e.g., a Peierls argument), the probability $$\pi_{\Lambda_R}^+(\mathcal E_{r,\ell}) \le C r^{d-1} e^{ - \beta \ell}\,.$$ 

Now expose, under a monotone coupling, the set of all minus connected components of $\sigma$ incident to $\Lambda_{r}^c$, revealing in particular, their outer boundary to be entirely plus in $\sigma$ and therefore also in $\sigma'$. This exposed set of pluses, call it $\gamma$ forms the boundary of the unrevealed set of vertices in $\Lambda_{r}$, and is measurable with respect to itself and its exterior. Therefore, the domain Markov property ensures that both $\sigma$ and $\sigma'$ have the same distribution on the un-revealed set of vertices in $\Lambda_r$ whose boundary is $\gamma$. Therefore, that interior can be coupled using the identity coupling. 

On the event $\mathcal E_{r,\ell}^c$, the revealed set of vertices does not reach $\Lambda_{r-\ell}$ and therefore the above-described coupling leaves $\sigma(\Lambda_{r- \ell}) = \sigma'(\Lambda_{r- \ell})$ except with probability $\pi_{\Lambda_R}^+(\mathcal E_{r,\ell})$.  
\end{proof}

Let $\pi^{\pm}_{A_{r,R}}$ denote the Ising distribution on the annulus $A_{r,R} = \Lambda_R\setminus \Lambda_r$ with $\pm$ boundary conditions, by which we mean $+$ on $\Lambda_R^c$ and $-$ in $\Lambda_r$. 

Our aim is to show that already at a sub-linear distance $r^{0.9}$ away from $\Lambda_r$, the measure under $\pi^\pm_{A_{r,R}}$ looks like the plus measure. Towards that goal, we first show that it suffices to work with annuli whose inner and outer box side-lengths are comparable to one another.

\begin{lem}\label{lem:annulus-coupling-estimate-linear}
Fix $r/100 \le m \le  (R-r)/2$. There exist $C(d), \beta_0(d)<\infty$ such that the following holds for all $\beta>\beta_0$:
    \begin{align*}
        \|\pi_{A_{r,R}}^\pm ( \cdot \restriction_{\Lambda_{r+m}}) - \pi_{A_{r,r+2m}}^\pm (\cdot \restriction_{\Lambda_{r+m}})\|_{\tv}  \le Ce^{ - \beta m}\,.
    \end{align*}
\end{lem}

\begin{proof}
    As a first step, we will couple $\pi_{A_{r,R}}^\pm$ to $\pi_{\Lambda_{R}}^+$ on the complement of $\Lambda_{r+m}$. 
    In the configuration $\sigma$ on an annulus $A_{r, R}$ with $\pm$ boundary conditions, let $\cI(\sigma)$ be the Ising interface surrounding the minus boundary conditions on $\Lambda_r$. (Formally, this can be taken to be the dual-plaquette boundary of the minus connected component of $\Lambda_r$.) 
    Couple $\sigma, \sigma'$ under a  monotone coupling for $\sigma \sim \pi_{A_{r,R}}^\pm$ and $\sigma' \sim \pi_{\Lambda_{R}}^+$ so that $\sigma \le \sigma'$ deterministically.

    Let $\mathcal E_{m}$ be the event that $\cI(\sigma) \cap \Lambda_{r+m}^c \ne \emptyset$ for $\sigma \sim \pi_{A_{r,R}}^\pm$. We bound the probability of $\mathcal E_m$ by a Peierls argument. Consider any $\sigma$ with interface $\cI(\sigma) = I$, and consider the operation $\Phi$ that replaces $I$ by a simpler contour given by the boundary $\partial \Lambda_r$. More precisely, take the set of disagreeing dual plaquettes (those separating differing spins in $\sigma$), remove from it all of the plaquettes in $I$, and perform the XOR operation with the dual plaquettes along $I_0 = \partial \Lambda_r$.

    Since $\cI(\sigma) = I$ had more than $2d r^{d-1} + 2(d-1)m$ many edges, and $I_0$ only has $2d r^{d-1}$ many edges, this operation decreases the energy of the associated configuration by at least $|I| - 2d r^{d-1} := K\ge 2m$. The enumeration over the number of pre-images whose energy change was $L\ge K$ is bounded by the number of connected sets of dual plaquettes confining $\Lambda_r$ in its interior, of which there are $C^L$ for a lattice dependent, $\beta$-independent, constant $C$. Therefore, by the usual Peierls bound, this implies that 
    \begin{align*}
        \pi^{\pm}_{A_{r,R}}( \mathcal E_m) \le Ce^{ - 2(d-1)(\beta - C)m}\,.
    \end{align*}

    Expose the connected component of minuses of $B_r$ under $\sigma \sim \pi^{\pm}_{A_{r,R}}$, note that its outer boundary is all-plus (and by monotonicity of the coupling, also is under a coupled sample from $\pi_{\Lambda_R}^+$). The identity coupling on its exterior, by the Markov property, leads to $\sigma(\Lambda_{r+m}^c) = \sigma'(\Lambda_{r+m}^c)$ on $\mathcal E_{m}^c$. This implies  
    \begin{align}\label{eq:annulus-close-to-box-bound}
        \|\pi_{A_{r,R}}^\pm (\cdot \restriction_{\Lambda_{r+m}^c})- \pi_{\Lambda_R}^+(\cdot \restriction_{\Lambda_{r+m}^c} ) \|_{\tv} \le   \pi_{A_{r,R}}^\pm (\mathcal E_m) \le Ce^{ - 2(d-1)(\beta - C) m}\,.
    \end{align}
    Now consider $\sigma'' \sim \pi_{A_{r,r+2m}}^{\pm}$, and note that by the same coupling arguments, the total-variation distance 
    \begin{align*}
        \|\pi_{A_{r,R}}^\pm (\cdot \restriction_{\Lambda_{r+m}}) - \pi_{A_{r,r+2m}}^\pm(\cdot \restriction_{\Lambda_{r+m}}) \|_{\tv} \le \pi_{A_{r,R}}^\pm(\mathcal E'_{m\leftrightarrow2m})\,,
    \end{align*}
    where we use $\mathcal E'_{m\leftrightarrow 2m}$ to denote a minus path from $\Lambda_{r+m}$ to $\Lambda_{r+2m}^c$. This last probability, is, by~\eqref{eq:annulus-close-to-box-bound},  bounded by its probability under $\pi_{\Lambda_R}^+$ which is at most $Ce^{ -  2(\beta -C) m}$, by a Peierls bound. A triangle inequality then gives the claimed bound for $\beta$ large. 
\end{proof}

The next lemma says that in the annuli with $-1$ boundary conditions on the inner box, we can couple them $o(r^{0.51})$ away from the $-$'s in the center to a plus  boundary condition measure. This lemma uses the good understanding of the typical height of a Dobrushin Ising interface above a hard floor, at sufficiently low temperatures.    

\begin{lem}\label{lem:coupling-annulus-away-from-minus}
   There exists $C(d)>0$ such that for every $\beta>\beta_0(d)$, $\ell \le r/100$, and $R\ge r$,  
    \begin{align*}
        \| \pi^{\pm}_{A_{r,R}}(\cdot \restriction_{\Lambda_{r+\ell}^c}) - \pi^+_{\Lambda_R}(\cdot \restriction_{\Lambda_{r+\ell}^c})\|_{\tv} \le \begin{cases}
            C r^C e^{- \ell^2/Cr}+ C e^{ - \beta r} & d=2 \\ C r^{d-1} e^{ - \beta \ell} & d\ge 3
        \end{cases}\,.
    \end{align*}
\end{lem}

\begin{proof}
    Let $\mathcal E_\ell$ be the event, as in the previous proof, that the interface $\mathcal I(\sigma)$ for $\sigma \sim \pi^\pm_{A_{r,R}}$ does not intersect $\Lambda_{r+\ell}^c$. As in that proof, we have 
    \begin{align*}
        \|\pi_{A_{r,R}}^\pm (\cdot \restriction_{\Lambda_{r+\ell}^c} ) - \pi_{\Lambda_R}^+(\cdot \restriction_{\Lambda_{r+\ell}^c})\|_{\tv} \le \pi_{A_{r,R}}^\pm (\mathcal E_\ell)\,.
    \end{align*}
    The event $\mathcal E_{\ell}$ is measurable with respect to the configuration on $\Lambda_{r+\ell}$ and therefore, by Lemma~\ref{lem:annulus-coupling-estimate-linear} it is sufficient to bound $\pi_{A_{r,R}}^\pm(\mathcal E_\ell)$ on the right-hand side under $R = 2r$, up to an error of $Ce^{ - \beta r/2}$.

    For this, it is sufficient to union bound over the probability that the interface $\mathcal I(\sigma)$ intersects any of the $2d$ planes $\{-r-\ell\} \times \mathbb Z^{d-1}$, $\{r+\ell\} \times \mathbb Z^{d-1}$, $\mathbb Z \times \{-r-\ell\} \times \mathbb Z^{d-2}$, etc. By rotational symmetry, these are all the same, so it suffices to consider one of them, let's say the north one $\mathbb Z^{d-1} \times \{r+ \ell\}$. Namely, if $\mathcal E_{\ell}^{\textsc{n}}$ is the event that $\mathcal I(\sigma)$ intersects $\mathbb Z^{d-1} \times \{r+\ell\}$, we have 
    \begin{align*}
        \pi_{A_{r,R}}^\pm (\mathcal E_\ell) \le 2d  \pi_{A_{r,R}}^\pm(\mathcal E_\ell^{\textsc{n}})\,.
    \end{align*}
    By monotonicity, we only increase the probability on the right-hand side if we increase the minus boundary to be all vertices in $\Lambda_R$ below height $r$, leaving us with a rectangular prism $[-r, r]^{d-1}\times [r, 3r]$ with $\pm$ boundary conditions that are $-$ on the southern face, and $+$ on all the other faces. On this domain, we are then asking the probability on a box with $-$ boundary conditions on its bottom, and $+$ boundary conditions on the other sides, that the interface reaches a height of $\ell/2$. 

    In dimension $d=2$, this maximal height oscillation of the interface is known to have a Gaussian tail, with the probability being at most $C_\beta \exp( - \kappa_\beta \ell^2 /r)$ for a sharp triangle inequality constant $\kappa_\beta$ and a $\beta$-dependent constant $C_\beta$ per e.g.,~\cite[Theorem 5.3]{LMST}. However, it is important to have a $\beta$-uniform version of such a statement. This is obtained as a step of our work in Section~\ref{sec:uniform-mixing-time} to get the uniformly quasi-polynomial mixing time, specifically in Lemma~\ref{lem:maximal-vertical-oscillations-with-floor} where it is shown (taking $N$ there to be our $r$ and $h$ there to be our $\ell$)  that for a $C$ independent of $\beta$, one has that if $r \ge e^{\beta/5}$, 
    \begin{align*}
        \pi_{A_{r,R}}^\pm (\mathcal E_\ell^{\textsc{n}}) \le C r^C e^{ - \kappa_\beta \ell^2/2r}
    \end{align*}
    which since $\kappa_\beta \ge 1/3$ for all large $\beta$ by~\eqref{eq:tau''-lower-bound}, is in particular at most $C r^C e^{ - \ell^2/Cr}$. On the other hand, if $r \le e^{\beta/5}$ then the probability of having a vertical oscillation of height $\ell$ can be bounded by a Peierls argument by $e^{-3\beta \ell}$ as follows. In an $r\times h$ box with $-$ boundary conditions on its bottom side and plus boundary conditions on its other sides, the probability of the interface reaching height $l$ given it reaches height $l-1$ is bounded by the probability of reaching height $l$ if its first $l-1$ heights are filled by the all-minus configuration, which is the probability of the interface deviating from the ground state flat interface in an $r\times h'$ box with $\pm$ boundary conditions. This latter probability is bounded, via a union bound, by $e^{ - (4-C)\beta}r \le e^{-3\beta}$ for $\beta>\beta_0$. The claimed bound of $e^{-3\beta\ell}$ follows by applying this iteratively.

    In dimension $d\ge 3$, this maximal height oscillation has an exponential tail beyond the $\log r$ scale due to rigidity of the interface at sufficiently low temperatures~\cite{Dobrushin72a}. However, this exact implication needs some extra work to deal with the interaction of the interface with the nearby floor with minus boundary conditions: see the bound of~\cite[Corollary 2.12]{chen2024logarithmicdelocalizationlowtemperature} and the following Remark~\ref{rem:on-the-beta-dependence-in-higher-dim} for a comment on the uniformity in large $\beta$ of that bound. 
\end{proof}

    \begin{remark}\label{rem:on-the-beta-dependence-in-higher-dim}
            Let us justify that the argument of Section 2.2 of~\cite{chen2024logarithmicdelocalizationlowtemperature} leads to $\beta$-independent prefactors. The first step therein is getting a bound with a ``soft floor", i.e., when the interface is only conditioned to be above height $0$, but there are no boundary pinnings at height $-1$. This is Lemma~2.8 therein, which only uses Theorem 2.5 therein, and which in turn has all its $\beta$-dependencies originating in the cluster-expansion based bound of~\cite{Dobrushin72a}. Therefore the constants involved therein only improve with $\beta \to\infty$. In order to translate this to a bound in the hard floor setup, the authors use a stochastic ordering saying the interface with soft boundary conditions is above the one with hard boundary conditions (their Proposition~2.9) to get \cite[Corollary 2.12]{chen2024logarithmicdelocalizationlowtemperature} together with its implicit tail bound. Evidently, in the stochastic comparison between the two distributions, there is no extra $\beta$-dependencies introduced. 
    \end{remark}

We conclude by using the above to bound the probability of not having the $\mathsf{StatEquiv}_k(B)$ event. 

\begin{cor}\label{cor:statequiv-bound}
    There exists $\beta_0(d)<\infty$ such that for every $\beta>\beta_0$, for every $k$, every $B\in \mathscr{B}_k$, 
    \begin{align*}
        \mathbb P(\mathsf{StatEquiv}_k(B)^c )\le \frac{1}{(2\ell_{k+1})^d} \frac{q_{k+1}}{10}\,.
    \end{align*}
\end{cor}

\begin{proof}
    We begin with a union bound, 
    \begin{align*}
        \mathbb P(\mathsf{StatEquiv}_k(B)^c) \le (2\ell_{k+1})^d \max_{B' \in \mathscr{B}_{k+1} : B\subset E_{+3}(B')} \mathbb P(\exists t\in [0,T_k]: U_t^{B,\pi}(E_{+3}(B)) \ne U_{t}^{B',\pi}(E_{+3}(B)))\,.
    \end{align*}
    Now introduce the event $\mathcal B(B)$ as the bad event that more than $2T_k |E_{+3}(B)|$ many clock rings occur in $E_{+3}(B)$ in times $[0,T_k]$. Notice that the status of the event whose probability is being taken on the right cannot change without a clock ring in $E_{+3}(B)$. 
    
    By Poisson tail bounds, the probability $\mathbb P(\mathcal B(B))$ is at most $\exp( - T_k |E_{+3}(B)|/C)$ for a universal constant $C$. Thus, up to that error, we can intersect the probabilities on the right-hand side with the event $\mathcal B(B)^c$. Next condition on the clock ring sequence in $E_{+3}(B)\times [0,T_k]$ on the event $\mathcal B(B)^c$. For any sequence of at most $2 T_k |E_{+3}(B)|$ clock rings in $[0,T_k]$, denoted $\mathcal T_{B \times [0,T_k]}$, we have 
    \begin{align*}
        \mathbb P(\exists t\in [0,T_k]: & U_t^{B,\pi}(E_{+3}(B)) \ne U_{t}^{B',\pi}(E_{+3}(B)), \mathcal B(B)^c) \\
        & \le 2T_k|E_{+3}(B)| \cdot \sup_{t\in [0,T_k]} \mathbb P(U_t^{B,\pi}(E_{+3}(B)) \ne U_{t}^{B',\pi}(E_{+3}(B)) \mid \mathcal T_{B\times [0,T_k]})\,.
    \end{align*}
    Since these chains are initialized at stationarity, even conditional on the clock ring times $\mathcal T_{B \times [0,T_k]}$ their distributions are $\pi_{E_{+4}(B)}^+$ and $\pi_{E_{+4}(B')}^+$ respectively. Furthermore, they are monotonically ordered. Therefore, the disagreement probability above is at most 
    \begin{align}\label{eq:disagreement-probability-above-bd}
        2T_k |E_{+3}(B)| \|\pi_{E_{+4}(B)}^+(\cdot \restriction_{E_{+3}(B)}) - \pi_{E_{+4}(B')}^+(\cdot \restriction_{E_{+3}(B)})\|_{\tv}\,.
    \end{align}
    By Lemma~\ref{lem:stationary-plus-bc-stationary-mixing}, the right-hand side is at most $2C T_k |E_{+3}(B)|^2e^{ -  \beta \ell_{k}/10}$. In total, we get 
    \begin{align*}
        \mathbb P(\mathsf{StatEquiv}_k(B)^c) \le e^{ - T_k |E_{+3}(B)|/C} + 2C T_k |E_{+3}(B)|^2 e^{ - \beta \ell_k/10}\,.
    \end{align*}
    Plugging in for $T_k$ of order $\ell_k$ and $|E_{+3}(B)| \le 2\ell_k$, the right-hand side is exponentially small in $\beta \ell_k/C$ for some universal constant $C$. Since $\ell_{k+1}$ is only quasi-polynomially large in $\ell_k$ and $q_{k+1}$ is only quasi-polynomially small in $\ell_k$, as long as $\ell_k$ is at least a sufficiently large (only depending on $C, d, M$) constant, the claimed inequality holds. 
\end{proof}

\begin{proof}[\textbf{\emph{Proof of Corollary~\ref{cor:statequiv-atleastk-bound}}}]
    By definition of $\mathsf{StatEquiv}_{\ge k}(B)$ with a union bound, we have 
    \begin{align*}
        \mathbb P(\mathsf{StatEquiv}_{\ge k}(B)^c) \le \sum_{k'\ge k} \sum_{B'\in \mathscr{B}_{k'}: B\subset B'} \mathbb P(\mathsf{StatEquiv}_{k'}(B')^c)\,.
    \end{align*}
    Bounding the number of summands, and applying Corollary~\ref{cor:statequiv-bound} on each summand, we get 
    \begin{align*}
        \mathbb P(\mathsf{StatEquiv}_{\ge k}(B)^c) \le \sum_{k' \ge k} (2\ell_{k'})^d \frac{1}{(2\ell_{k'+1})^d} \frac{q_{k'+1}}{10} \le  C q_{k+1}\,,
    \end{align*}
    for a universal constant $C$, since the sequence $(q_{k'})_{k'}$ is summable, decaying faster than any polynomial. This concludes the proof so long as $\ell_0$ is a large enough constant (only depending on $C,d$). 
\end{proof}

\begin{cor}\label{cor:statequiv-annulus-bound}
For any $B \in \mathscr{B}_k$ and any $R \in \mathscr{R}_{k-1}(B)$, 
    \begin{align*}
        \mathbb P\big(\exists t\in [T_{k-1},T_k] : V_t^{R,\pi}(E_{+2}(R) \setminus E_{+1}(R)) \ne V_t^{B\setminus R,\pi}(E_{+2}(R)\setminus E_{+1}(R)\big)\le \frac{1}{(2\ell_{k})^{d+1}}  \frac{q_{k+1}}{10}\,.
    \end{align*}
\end{cor}
\begin{proof}
    Similar to the proof of Corollary~\ref{cor:statequiv-bound}, introduce the event $\mathcal B(B)$ that there are more than $2T_k|E_{+2}(R)|$ many clock rings in $E_{+2}(R) \times [0,T_k]$. The probability of $\mathcal B(B)^c$, by Poisson tail bounds, is at most $e^{ - T_k |E_{+2}(R)|/C}$ for a universal constant $C$. Conditional on any clock ring sequence, the status of the event in the probability above only can change at one of the clock ring times. Moreover, even conditional on the clock ring sequence, the law of the two processes involved are exactly $\pi_{E_{+4}(R)}^+$ and $\pi_{E_{+4}(B)\setminus R}^{\pm}$. Therefore, following the proof of Corollary~\ref{cor:statequiv-bound}, we can bound the above by 
    \begin{align*}
        e^{ - T_k|E_{+2}(R)|/C} + 2T_k |E_{+2}(R)|\|\pi_{E_{+4}(R)}^+(\cdot \restriction_{E_{+2}(R)\setminus E_{+1}(R)}) - \pi_{E_{+4}(B) \setminus R}^\pm(\cdot \restriction_{E_{+2}(R)\setminus E_{+1}(R)})\|_{\tv}\,.
    \end{align*}
    Note here that $R$ has side-length at most $100 s_k \ell_{k-1}$, but its enlargement is only by $4\ell_{k-1}/10$. 
    By a triangle inequality
\begin{align*}
    \|\pi_{E_{+4}(R)}^+(\cdot \restriction_{E_{+2}(R)\setminus E_{+1}(R)})   - \pi_{E_{+4}(B) \setminus R}^\pm & (\cdot \restriction_{E_{+2}(R)\setminus E_{+1}(R)})\|_{\tv} \\
    & \le \|\pi_{E_{+4}(R)}^+ (\cdot \restriction_{E_{+2}(R)})  - \pi_{E_{+4}(B)}^+ (\cdot \restriction_{E_{+2}(R)})\|_\tv \\
    & \quad  + \|\pi_{E_{+4}(B)}^+ (\cdot \restriction_{E_{+1}(R)^c} ) - \pi_{E_{+4}(B) \setminus R}^\pm(\cdot \restriction_{E_{+1}(R)^c})\|_\tv\,.
\end{align*}
The first distance above is bounded by Lemma~\ref{lem:stationary-plus-bc-stationary-mixing} by $C(100 s_k \ell_{k-1})^{d-1}e^{ - \beta \ell_{k-1}}$. The second distance is bounded above by Lemma~\ref{lem:coupling-annulus-away-from-minus} by the following sum, for a $C(d)$:  
$$C(s_k \ell_{k-1})^{C} e^{ -  \ell_{k-1}^2/C s_k \ell_{k-1}} +  C ( s_k \ell_{k-1})^{d-1}e^{ - \beta \ell_{k-1}}\,.$$
 Since the localization scale $s_k$ was chosen to be polylogarithmic in $\ell_k$, as long as $\ell_0$ is larger than an absolute constant (only depending on $M, C,d$), all the above terms are at most $e^{ - \ell_{k-1}^{0.9}}$, and in turn since $q_{k+1}$ is at least quasi-polynomially small in $\ell_k$, the claimed inequality holds.     
\end{proof}

\subsection{Equilibrium estimates for the conditional finite-domain distribution}

Let $\pi_{\mathbb T}^+$ be the plus phase measure on the torus $\mathbb T = (\mathbb Z/n\mathbb Z)^d$, meaning, $\pi_{\mathbb T}^+ = \pi_{\mathbb T} ( \cdot \mid \sum_v \sigma_v \ge 0)$. With minor modifications of the proof of Lemma~\ref{lem:stationary-plus-bc-stationary-mixing} (the additional step being first ruling out topologically non-trivial interfaces separating plus and minus spins before implementing the Peierls argument)  one arrives at the following.

\begin{lem}[See Theorem 1.4 of~\cite{GhSi22}; though easier Peierls-style arguments suffice at low enough temperatures]\label{lem:Torus-statequiv}
For every $d\ge 2$, there exists $C(d),c(\beta,d)>0$ (going to infinity as $\beta \uparrow \infty$) such that for all $\beta>\beta_c(d)$ the following holds. Suppose $\Lambda_r \subset \mathbb T_n^d$ is a box of side-length $r$, and let $\Lambda_{r/2}$ be the concentric box of side-length $r/2$. 
    $$\|\pi_{\mathbb T}^+( \cdot \restriction_{\Lambda_{r/2}}) - \pi_{\Lambda_r}^+(\cdot \restriction_{\Lambda_{r/2}})\|_\tv \le C e^{ -  c r}\,.$$ 
\end{lem}

\begin{lem}[See~\cite{Pisztora96,Bodineau05}, though easier Peierls-style arguments suffice at low enough temperatures]\label{lem:torus-bottleneck}
    For every $d\ge 2$ and $\beta>\beta_c(d)$, there exists $C(d)$ and $c(\beta,d)$ (going to inifinity as $\beta\uparrow \infty$) such that 
    \begin{align*}
        \pi_{\mathbb T^d}\Big( |\sum_{v} \sigma_v|\le 2 \Big) \le C e^{ - cn^{d-1}}\,.
    \end{align*}
\end{lem}

\subsection{Concluding Proposition~\ref{prop:Dtilde-probability-bound}}

\begin{proof}[\textbf{\emph{Proof of Proposition~\ref{prop:Dtilde-probability-bound}}}]
    We prove this bound inductively. For the base case, $k=0$, by Definition~\ref{def:Dtilde}, for every $B' \in \mathscr{B}_0$ that $\widetilde D_0(B')$ is $1$ if some vertex $v$ in $B'$ (of which there are $\ell_0^d$) has $Q(v) = -1$, so its probability is at most $q_0$ so as long as $1-p_0 \le \frac{1}{\ell_0^d \ell_3}$ (a $\beta$-independent threshold because $\ell_0$ is large independently of $\beta$).
    
    Now assume the bound of Proposition~\ref{prop:Dtilde-probability-bound} holds for $k-1 \ge 0$ and show it holds for $k$. We start with the union bound 
    \begin{align}\label{eq:decomposing-Dtilde-probability}
           \mathbb P(\widetilde{D}_k(B) =1) \le \mathbb P(|\mathsf{Dis}_{k-1}(B)| > s_k) + \mathbb P(\mathsf{Bad}_{k-1}(B) \ne \emptyset) + \mathbb P(\mathsf{InfProp}_k(B)^c) + \mathbb P(\mathsf{LocCoup}_k(B)^c)\,.
    \end{align}
    By the inductive assumption and Lemma~\ref{lem:rarity-of-disagreement-regions}, the first term above is bounded by $q_k/10$. For the second term in~\eqref{eq:decomposing-Dtilde-probability}, by a union bound over $B'\in \mathscr{B}_{k-1}$ such that $B' \subset E_{+3}(B)$, we have 
    \begin{align*}
        \mathbb P(|\mathsf{Bad}_{k-1}(B)| \ne \emptyset) \le (2\ell_k)^d \max_{B'\in \mathscr{B}_{k-1}} \Big(\mathbb P(\mathsf{StatEquiv}_{k-1}(B')^c) + \mathbb P(\mathsf{InfProp}_{k-1}(B')^c)\Big)\,.
    \end{align*}
    By Lemma~\ref{lem:information-propagation}, $\mathbb P(\mathsf{InfProp}_{k-1}(B')^c) \le e^{ - \ell_{k-1}/10}$ which is smaller than $q_k/(10 (2\ell_k)^d)$ so long as $\ell_{k-1}$ is at least a large enough constant (as a function of $M$). For the other term, above, we use Corollary~\ref{cor:statequiv-bound}, to get that term two of~\eqref{eq:decomposing-Dtilde-probability} is bounded by $q_{k}/5$. 

    Next, on term three of~\eqref{eq:decomposing-Dtilde-probability}, by Lemma~\ref{lem:information-propagation}, $\mathbb P(\mathsf{InfProp}_k(B)^c) \le e^{ - \ell_k/10}$, which is smaller than $q_k/10 = \frac{1}{10} e^{ - (\log \ell_k)^{M^3}}$ so long as $\ell_k$ is at least a large enough constant (as a function of $M$). 

    Finally, for the fourth term of~\eqref{eq:decomposing-Dtilde-probability}, by a union bound, 
    \begin{align*}
        \mathbb P(\mathsf{LocCoup}_k(B)^c) \le \sum_{R\in \mathscr{R}_{k-1}(B)} \mathbb P(\mathsf{Sandwich}(B,R)^c)\,.
    \end{align*}
    There are at most $\ell_k^d \cdot (100 s_k \ell_{k-1}) \le \ell_{k}^{d+1}$ many choices for the $R$. For each one, bounding the probability of $\mathsf{Sandwich}(B,R)^c$ by a union bound and Corollary~\ref{cor:statequiv-annulus-bound}, we get   
    \begin{align*}
        \mathbb P(\mathsf{LocCoup}_k(B)^c) \le \ell_{k}^{d+1} \cdot \Big( \max_{R\in \mathscr{R}_{k-1}(B)} \mathbb P(V_{T_k}^{R,+}\ne  V_{T_k}^{R,-}) + \frac{1}{(2\ell_{k})^{d+1}} \frac{q_{k+1}}{10}\Big)\,.
    \end{align*}
    The processes $V_{T_k}^{R,-}$ and $V_{T_k}^{R,+}$ are Glauber dynamics on $E_{+4}(R)$ with plus boundary conditions, run for time $t_k = \ell_k/10$. By monotonicity of the coupling, we can bound 
    \begin{align*}
        \mathbb P(V_{T_k}^{R,+}\ne  V_{T_k}^{R,-}) \le  (200 s_k \ell_{k-1})^d \max_{\iota \in \{+,-\}} \|  \mathbb P(V_{T_k}^{R,\iota} \in \cdot) - \pi_{E_{+4}(R)}^+\|_{\tv}\,.
    \end{align*}
    By sub-multiplicativity of total-variation distance to stationarity after the mixing time, as long as $t_k \ge \tmix(E_{+4}(R), +1)$ (the mixing time on $E_{+4}(R)$ with $+1$ boundary conditions), by Assumption~\ref{assump:uniform-mixing-time}, this is at most 
    \begin{align*}
        \mathbb P(V_{T_k}^{R,+}\ne  V_{T_k}^{R,-}) \le  (200 s_k \ell_{k-1})^d e^{ - t_k / C\exp((\log 100 s_k \ell_{k-1})^C)}\,,
    \end{align*}
    for a $\beta$-independent constant $C$. Using $200 s_k \le \ell_{k-1}$, say, and then that $\log \ell_{k-1} = (\log \ell_k)^{1/M}$, this is at most 
    \begin{align*}
        \ell_{k-1}^{2d} \exp \Big( - \frac{\ell_k}{C M e^{2 (\log \ell_k)^{C/M}}}\Big)\,.
    \end{align*}
    As long as $M \ge M_0(C,d)$ and in turn $\ell_0$ is a sufficiently large constant, this decays as a stretched exponential in $\ell_k$ and therefore is bounded by $\frac{1}{(2\ell_k)^{d+1}} \frac{q_k}{10}$ as well. Plugging back in to the bound on $\mathbb P(\mathsf{LocCoup}_k(B)^c)$, we get that that is also at most $q_k/5$. 
    In total, we have shown that each of the four terms in~\eqref{eq:decomposing-Dtilde-probability} are at most $q_k/5$, concluding the proof. 
\end{proof}

\subsection{Proof of Theorem~\ref{thm:main}}

Given the above, we have fully established Theorem~\ref{thm:main-general}. To deduce Theorem~\ref{thm:main}, it suffices to handle the case of $\beta \in (\beta_c,\beta_0)$ by taking $p_0$ sufficiently large and using a significantly simpler argument as the initialization stochastically dominates the target stationary distribution. 

\begin{proof}[\textbf{\emph{Proof of Theorem~\ref{thm:main}}}]
    Let $\beta_0$ be the larger of the constants from Theorem~\ref{thm:main-general} and Theorem~\ref{thm:beta-independent-2D-mixing-time}. By Theorem~\ref{thm:beta-independent-2D-mixing-time}, Assumption~\ref{assump:uniform-mixing-time} holds in $d=2$ for that $\beta_0$, and therefore by Theorem~\ref{thm:main-general} the claimed result of Theorem~\ref{thm:main} holds for the $p_{0}(2)$ of Theorem~\ref{thm:main-general}, for all $\beta>\beta_0$. 
    
    It now suffices to show that the claims of Theorem~\ref{thm:main} hold for a different $p_0'<1$, for all $\beta \in (\beta_c,\beta_0]$ (then concluding by taking $p_0 \gets\max\{p_0(2),p_0'\}$). By the Ising single-site marginals  (e.g., as in~\eqref{eq:Glauber-update-rule}), setting $p_0'$ to be the maximal probability of a site being plus conditional on its neighbors, there exists a $p_0' = \frac{1}{1+e^{-8\beta_0}}<1$ such that $\pi_{\mathbb Z^2}^+ \preceq \bigotimes_{\Z^d} \text{Rad}(p)$ for all $\beta \le \beta_0$ and all $p>p_0'$. 
    
    \smallskip 
    \noindent 
    \textbf{Infinite volume}. We begin with the infinite-volume case. By the domination above, for any $v$, 
    \begin{align*}
        \pi_{\mathbb Z^2}^+(\sigma_v=+1) \le \mathbb P_{\otimes \text{Rad}(p)}(X_t(v)  =+1) \le \mathbb P_{+}(X_t(v) = +1)\,.
    \end{align*}
    Let $X_t^+$ denote the chain from the all-plus initialization, and $X_t^Q$ be the one from the Rademacher initialization. 
    Under the grand monotone coupling, we have the bound 
    \begin{align*}
        \mathbb P \big( X_t^\pi(\Lambda_r) \ne X_t^Q(\Lambda_r)\big)  & \le \E[X_t^{\pi}(\Lambda_r) \ne X_t^1(\Lambda_r) ]  \le r^d \mathbb E[X_t^\pi (v) \ne X_t^1 (v)]\,.
    \end{align*}
    Then by e.g.,~the proof of Proposition~3.3 of~\cite{GhSi22} (using spatial mixing within a phase when $\beta>\beta_c$ and monotonicity of dynamics to move to a dynamics on a box of radius $r$ for time $t$ with all-plus boundary conditions), for any $r,t$ this right-hand side is at most  $C_\beta e^{- r/C_\beta} +  r^d e^{ - t/\tmix(\Lambda_r^+)}$. Choosing $r = (\log t)^{100}$ or $t = e^{ r^{1/100}}$ (so that by~\cite{LMST}, $t/\tmix(\Lambda_r^+) \ge C_\beta^{-1} e^{ r^{1/100}- C_\beta(\log r)^2} = e^{\Omega(r^{1/100})}$), then 
    \begin{align*}
        \mathbb P \big( X_t^\pi(\Lambda_r) \ne X_t^Q(\Lambda_r)\big)  \le e^{ - \Omega((\log t)^{100})} + r^d e^{ - e^{\Omega( r^{1/100})}} \le r^d e^{ - \Omega( (\log t)^{100})}\,.
    \end{align*}
    Evidently, since $t \ge r$ then the $r^d$ prefactor can be absorbed into the constant in the exponential and this is at most $e^{ - \Omega((\log t)^{100})}$.     
    By Markov's inequality therefore, if $t \ge r$, we have for all $s$, 
    \begin{align*}
        \P(x_0 : \mathbb P(X_t^{x_0}(\Lambda_r) \ne X_t^\pi(\Lambda_r)) \ge s)  &\le \frac{\E_{x_0 \sim Q}[\P(X_t^{x_0}(\Lambda_r) \ne X_t^{\pi}(\Lambda_r))]}{r} \\
        & \le s^{-1}e^{ - \Omega((\log t)^{100})}\,.
    \end{align*}
    By taking $s = e^{ - \Omega((\log t)^{100}})$ in such a way that the right-hand above is itself $e^{ - \Omega((\log t)^{100})}$ (just taking half the constant in the exponent), we conclude. 

    \smallskip
    \noindent \textbf{Finite volume}. The proof is essentially identical to the infinite-volume one with a need to do the stochastic domination step with respect to a restricted chain that rejects updates that take it to negative magnetizations. 
This chain is identical to the unrestricted one until the hitting time of zero-magnetization, which occurs on $e^{\Omega(n^{d-1})}$ timescales, which is the source of the upper bound on the times $t$ for which it applies.  
    
    Indeed, by Eq.~(3.4) of~\cite{GhSi22} and the bounds following it in in the two subsequent sentences, for all $t \le e^{ O(n^{d-1})}$, with the choice of $g_n(t) = e^{ \Omega((\log t)^{1/2})}$ due to the mixing time result of~\cite{LMST}, 
    \begin{align*}
    	 \|\mathbb P(X_t^+(v)\in \cdot) - \pi^+_{\mathbb T}(\sigma_v \in \cdot)\|_{\tv} \le t e^{ - \Omega(n^{d-1})} + n^d e^{ - e^{\Omega((\log t)^{1/2})}} 
    \end{align*}
    which is at most $n^{-10}$ for $t  \in [n^{o(1)}, e^{ O(n^{d-1})}]$, for a suitably chosen $n^{o(1)}$ sequence. 
	\end{proof}

\subsection{Extension to other biased initializations}
In this section, we show how to modify the above argument to allow $Q$ to instead be drawn from $\pi_{\Z^2,\beta',h}$ with a large enough external field $h$ and any $\beta'$, or from a different low-temperature plus phase distribution $\pi^+_{\Z^2,\beta'}$ for $\beta'>\beta_0$. The idea will be  instead of using a single $Q$ process, to also consider a sequence of $Q$ processes, each of which are stationary samples $\pi_{E_{+4}(B),\beta',h}$ for each $B$. 

\begin{proof}[\textbf{\emph{Proof of Theorem~\ref{cor:mixing-from-other-temp}}}]
For each $B \in \mathscr{B}_k$, let $Q_B$ be a stationary sample from $\pi_{E_{+4}(B),\beta',h}^+$, i.e., stationary on $E_{+4}(B)$ with plus boundary conditions and external field $h>h_0$. The random variables $(Q_B)_{B}$ are all put into the same probability space by an independent use of the coupling of stationary samples from Definition~\ref{def:initialization-coupling}.

Observe that one has the inequalities $Q_{B'}\le Q_B$ if $B \subset B'$, and for $B^1 \subset B^2 \subset \cdots$ with $B^i \in \mathscr{B}_i$, then $\lim_{i\to\infty} Q_{B^i}(u)$ exists and is drawn from $\pi_{\mathbb Z^2,\beta',h}$. In particular, one still has for fixed $t$ that~\eqref{eq:i-to-infty-limit-of-localized-dynamics} still holds in the sense that 
\begin{align*}
    X_t^{\pi \wedge Q}(u) = \lim_{i\to\infty} U_t^{B^i, \pi \wedge Q_{B^i}}(u)\,.
\end{align*}
We now describe the (natural and minor) modifications to the proof that would give the same result of Theorem~\ref{thm:main-general} when the initial process $Q$ is $\pi_{\Z^2,\beta',h}$. Whenever we are at scale $B\in \mathscr{B}_k$, and the process $Q(u)$ is invoked, it is evaluated using $Q_B(u)$. 

Define the analogue of $\mathsf{StatEquiv}$ for the $Q$-process: 
\begin{align*}
    \mathsf{StatEquiv}^Q_{k}(B) = \bigcap_{\substack{B'\in \mathscr{B}_{k+1} \\ B \subset E_{+3}(B') }} \{Q_B(E_{+3}(B)) = Q_{B'}(E_{+3}(B))\}\,. 
\end{align*}
Continuing along the lines of the proof, note that in Lemma~\ref{lem:events-are-measurable}, we could add that if $B'\in \mathscr{B}_{k-1}$ such that $B'\subset E_{+3}(B)$, then $\mathsf{StatEquiv}_{k-1}^Q(B')$ is measurable with respect to the randomness of $E_{+4}(B)\times \{0\}$. 

Turning now to Definition~\ref{def:Dtilde}, the adjustment to the definition would be that in~\eqref{eq:Bad-set}, there would be an extra union with $\mathsf{StatEquiv}_{k-1}^Q(B')^c$. In words that is to say that a block is also called bad if the local version of the $\pi_{\beta',h}$ initialization does not agree with the global version.  

Corollary~\ref{cor:measurable} holds because of the addition of $\mathsf{StatEquiv}_{k-1}^Q(B')$'s measurability to Lemma~\ref{lem:events-are-measurable}.  

In Proposition~\ref{prop:Dtilde-implies-coupling}, the statement becomes $U_{T_k}^{B,\pi \wedge Q_B}(B)$, and $\mathsf{StatEquiv}_{\ge k}(B)$ is further intersected with 
\begin{align*}
    \mathsf{StatEquiv}_{\ge k}^{Q}(B) = \bigcap_{k'\ge k} \bigcap_{B' \in \mathscr{B}_{k'}: B \subset B'} \mathsf{StatEquiv}^Q_{k'}(B')\,.
\end{align*}
Towards the proof of Proposition~\ref{prop:Dtilde-implies-coupling}, 
the preliminary steps of Claim~\ref{cl:covering}--Lemma~\ref{lem:sandwiching-couplings} are all unchanged as they do not involve the $Q$ process. 
In its proof, the steps to consider are those that involve the $Q$ process. For the base case, it becomes: for every $B \in \mathscr{B}_0$, $U_{0}^{B,\pi \wedge Q_B}(B) = U_0^{B,\pi}(B)$ by construction $\widetilde  D_0(B)= \mathbf 1\{\bigcup_{v\in B} Q_B(v) = -1\}$. 

For the inductive step, when going to show~\eqref{eq:inductive-step-that-uses-Q}--\eqref{eq:inductive-step-that-uses-Q-2}, the first equality holds because we have added the event $\mathsf{StatEquiv}_{k-1}^Q(B')$ into $\mathsf{Bad}_{k-1}(B) = \emptyset$, and the implication still follows from the information propagation bound. 

For the ``moreover" statement, recall that under $\mathsf{StatEquiv}_{\ge k}^Q(B)$, if $(B_{\ell})_{\ell \ge k+1}$ are a sequence in $\mathscr{B}_{\ell}$ with $B \subset B_\ell$, then 
\begin{align*}
    X_0^{\pi \wedge Q} (E_{+3}(B)) = \lim_{\ell \to \infty} U_0^{B_\ell, \pi\wedge Q_{B_\ell}}(E_{+3}(B))= U_{0}^{B, \pi \wedge Q_{B}}(E_{+3}(B))\,.
\end{align*}
From there, the conclusion follows identically. 

The final things to generalize are Proposition~\ref{prop:Dtilde-probability-bound} and Corollary~\ref{cor:statequiv-atleastk-bound}. For the base case of Proposition~\ref{prop:Dtilde-probability-bound}, we just note that as long as $h \ge h_0$ is sufficiently large (depending on $d$ but not depending on $\beta$), for any $\beta'>0$, one can have that the probability of some minus site in $B'\in \mathscr{B}_0$ by a union bound is at most $\frac{1}{\ell_3}$ by the fact that the magnetization of $\pi_{\mathbb Z^2,\beta',h}$ goes to $1$ as $h\to\infty$. 

For both of these, all that is required to add in is to absorb a further bound on the probability of $\mathsf{StatEquiv}_{k-1}^Q(B)$ analogous to Corollary~\ref{cor:statequiv-bound}. But this reduces to the stationary estimate analogous to~\eqref{eq:disagreement-probability-above-bd} for Ising models with a large external field. The exponential decay of such models is classical and holds at large $h$ for every $\beta'>0$. 
Thus, the $Q$-process analogue for~\eqref{eq:disagreement-probability-above-bd} is at most $|E_{+3}(B)|e^{ - c\beta \ell_k/10} \le q_k/10$ as there. Since there was room for an extra $q_k/10$ in the proof of Proposition~\ref{prop:Dtilde-probability-bound} the same inductive proof goes through.  

The proof of Theorem~\ref{thm:main-general} then goes mutatis mutandis, yielding the desired generalization. 
\end{proof}

\begin{remark}\label{rem:other-low-temp-initialization}
    The case of initializations from other low-temperature plus phase measures, i.e., $\pi_{\Z^2,\beta'}^+$ proceeds in the same way. Indeed, the only places properties of the initial distribution were used above were its translation invariance, its exponential decay of correlations, and its having a magnetization sufficiently close to $1$. All of these are achieved by $\pi^+_{\Z^2,\beta'}$ for $\beta'>\beta_0$. 
\end{remark}

\appendix

\section{Uniform quasi-polynomial mixing with plus boundary in 2D}\label{sec:uniform-mixing-time}

In this section, we show that the 2D Ising model Glauber dynamics satisfies Assumption~\ref{assump:uniform-mixing-time}. Namely, our aim in this section is to establish Theorem~\ref{thm:beta-independent-2D-mixing-time}.

A quasipolynomial bound on the mixing time was established in~\cite{LMST} for the 2D Ising Glauber dynamics with $+$ boundary conditions; however the dependency of the constant in the bound on $\beta$ was not tracked, and in fact, relies subtly on the $\beta$-dependencies of various equilibrium estimates on Gaussian tail behavior of the interface. In the below we track the chain of dependencies, give some more precise $\beta$-dependent bounds on interface fluctuations, and show a quasi-polynomial mixing time bound with plus boundary that does not deteriorate as $\beta \uparrow \infty$.

\subsection{We only need to consider $n$ at least exponential in $\beta$ }
We first show that in order to establish Theorem~\ref{thm:beta-independent-2D-mixing-time}, it is sufficient to give a quantitative bound on the large-$\beta$ dependence of the constants in the bound of~\cite{LMST} only when $n\ge e^{\Omega(\beta)}$. 

\begin{lem}\label{lem:small-n-mixing-time}
    There exists $C>0$ and $\beta_0$ such that if $\beta >\beta_0$, and $n\le n_0:= e^{\beta/3}$ then the mixing time on $\Lambda_n$ with $+$ boundary conditions is at most $C n^{2} (\log n)^C$. 
\end{lem}

\begin{proof}
    By~\cite[Theorem 1.3]{FoScSi02} (see also~\cite{Lacoin-Lifshitz-any-dimension} for the analogous result in $\mathbb Z^d$), the ``$\beta = \infty$" time to absorption in the all-$+$ configuration is at most $C n^2 (\log n)^C$ except with probability $1/20$, for some $C(d)$. (The $\beta = \infty$ Markov chain assigns each vertex a rate-1 Poisson clock, then when a vertex's clock rings, updates by taking the majority spin of its neighbors, and if it is a tie, then flipping a fair coin to decide its new spin. This is the $\beta \to\infty$ limit of~\eqref{eq:Glauber-update-rule}.)
    
    We couple the low-temperature dynamics with the zero-temperature one. A single update of the two agrees except with probability 
    \begin{align*}
        \max_{(\sigma_w)_{w\sim v}} \|\pi_\beta(\sigma_v\in \cdot \mid (\sigma_w)_{w\sim v}) - \lim_{\beta \uparrow \infty} \pi_{\beta}(\sigma_v \in \cdot \mid (\sigma_w)_{w\sim v})\|_{\tv}\le  1- \min\Big\{ \frac{1}{1 + e^{-2\beta}}, \frac{1}{1+e^{-4\beta}}\Big\} \le e^{-2\beta}\,.
    \end{align*}
    For a universal large constant $C'$, with probability $1-(1/20)$ in continuous time $C \cdot  n^2 (\log n)^C$, at most $C \cdot C'\cdot  n^4 (\log n)^C$ clock rings occur. Thus, so long as \begin{align}\label{eq:zero-temp-positive-coupling-no-of-steps}e^{ - 2\beta} \cdot C\cdot C' \cdot  n^4 (\log n)^C \le 1/20\end{align} then by a union bound, one has that the positive temperature and zero-temperature chains agree on their entire trajectory for time $C \cdot C' \cdot n^2 (\log n)^C$ except with probability $1/20$. In particular,  assuming~\eqref{eq:zero-temp-positive-coupling-no-of-steps}, 
    one has for the finite-$\beta$ Glauber dynamics that 
    \begin{align*}
        \max_{x_0} \mathbb P_{x_0}( X_t \not \equiv +1) \le \frac{1}{10}\,,\qquad \text{for $t= C n^2 (\log n)^C$} \,.
    \end{align*}
    To argue that this bounds the mixing time, we also show that the stationary distribution puts most of its mass on the all-plus configuration as follows: 
    \begin{align*}
        \|\pi - \delta_{+1}\|_{\tv} \le 2 \pi( \sigma \not \equiv +1) \le 2 n^2 \sum_{k\ge 4} 4^k e^{ - \beta k}\le 2n^2 \frac{1}{1-e^{- 4(\beta - \log 4)}} e^{- 4 (\beta - \log 4)} \le 4 n^2 e^{ - 2\beta}\,,
    \end{align*}
    by a Peierls bound sending all configurations to the all-plus one, where the last inequality used $\beta \ge \log 4$. We therefore obtain that 
    \begin{align*}
        e^{ - 2\beta} \cdot C \cdot C'\cdot n^4 (\log n)^C \vee 4n^2 e^{-2\beta} \le \frac{1}{20} \implies \tmix(\Lambda_n^+ ) \le C n^2 (\log n)^C\,.
    \end{align*}
    The relations on $\beta, n$ hold if $n\le n_0:= e^{\beta/3}$ and $\beta$ is a big enough constant. 
\end{proof}

This reduces the task of proving Theorem~\ref{thm:beta-independent-2D-mixing-time} into proving that for all $n\ge e^{\beta/3}$ that for some $C$ and all $\beta>\beta_0$, the mixing time $\tmix(\Lambda_n^+)\le \exp(C\beta (\log n)^2)$. As described in the proof sketch, this entails modifying various of the key equilibrium estimates that were inputs into~\cite{LMST}, and then explaining how with those new equilibrium estimates the proofs can be adapted to have this kind of tame $\beta$-dependency. 

In what follows, our notation will change from subsection to subsection to mirror the notation of the respective paper being followed (and refined in terms of $\beta$-dependencies) therein. $C$ will be used to denote a constant that does not depend on $\beta$ (and can change from line to line).

\subsection{Large-$\beta$ behavior of limiting surface tension}
Many of the key estimates on Ising interfaces used for the mixing time with plus boundary conditions are expressed in terms of the surface tension function, whose exact expression will actually be used for us. In this section, we describe this and other infinite-volume quantities that arise naturally in the study of low-temperature Ising interfaces, and study their large $\beta$ asymptotics. 

In this subsection, for an angle $\theta \in [- \frac{\pi}{4}, \frac{\pi}{4}]$, define the Ising model with $\pm$-boundary conditions at angle $\theta$ as the one on $\Lambda_{N,N'} = ([0,N]\times [-N'/2,N'/2])\cap \mathbb Z^2$ with boundary conditions $(\pm,\theta)$ that are $+$ on all sites above the line $L_\theta(x) = (x,x\tan\theta)$, and minus on all sites at or below $L_\theta(x)$. Let $Z_{\pm,\theta,N,N'}$ denote the corresponding partition function, while $Z_{+,N,N'}$ is the partition function on the same domain but with all-plus  boundary conditions. 

Define the surface tension at angle $\theta\in (-\frac{\pi}{2},\frac{\pi}{2})$ as the limit 
\begin{align}\label{eq:surface-tension-def}
    \tau_\beta(\theta) := \lim_{N\to\infty} \lim_{N'\to\infty} -\frac{1}{\|L_\theta(N)\|} \log \frac{{Z}_{\pm,\theta,N,N'}}{Z_{+,N,N'}}\,.
\end{align}
which is to say that it is the exponential rate for the unlikeliness of seeing an interface at angle $\theta$ in a box of width $N$. By e.g., Lemma 3.1.1 of~\cite{VelenikThesis}, the limit $\tau$ exists.

\begin{lem}\label{lem:uniform-in-beta-bounds-on-surface-tension-and-stiffness}
    There exists $\beta_0$ such that for all $\beta>\beta_0$, 
    \begin{enumerate}
        \item Scaling of the surface tension in $\beta$: For all $\theta\in [-\frac{\pi}{4}, \frac{\pi}{4}]$  one has 
        \begin{align*}
            \tau_{\beta}(\theta) \ge (1.9)\beta\,.
        \end{align*}
        \item Uniform positive stiffness: for all $\theta\in [-\frac{\pi}{4}, \frac{\pi}{4}]$, one has
        \begin{align*}
            \tau_\beta''(\theta) + \tau_\beta(\theta)> 1/3\,.
        \end{align*}
    \end{enumerate}
\end{lem}

\begin{proof}
We use exact formulas for the two-dimensional Ising model (see e.g.,~\cite{McCoyWu}), surface tension $\tau_\beta(\theta)$ to derive these uniform-in-$\beta$ bounds. 
We copy below the exact computation of Eqs.\ (24a)--(24e) of~\cite{AkutsuAkutsu}, with the mapping $\beta \mapsto 1, K \mapsto \beta, \gamma \mapsto \tau$ of the surface tension at fixed $\beta>\beta_c$ (see also~\cite{AbrahamReed1977,RottmanWortis} which they cite for the exact calculation of the surface tension):
\begin{align*}
    \tau_\beta (\theta)&  = \eta_1 \cos \theta + \eta_2 \sin \theta \\ 
    \eta_1 & = \sinh^{-1}(\alpha_\theta \cos \theta) \qquad \text{and} \qquad \eta_2 = \sinh^{-1}(\alpha_\theta \sin \theta) \\ 
    \alpha_\theta&  = M(1-(2/M)^2)^{1/2}(1+(\sin^2(2 \theta)+ (2/M)^2 \cos^2(2\theta))^{1/2})^{-1/2} \\ 
    M &= \cosh^2(2\beta)/\sinh(2\beta)\,.
    \end{align*}
    The first bound to check is the behavior of the surface tension itself as $\beta$ grows. Towards that, notice first that for large $\beta$, $M  = \cosh(2\beta) \coth(2\beta)$ is asymptotic to $\frac{1}{2}(1+o_\beta(1)) e^{ 2\beta}$. In turn that leads to behavior of $\alpha_\theta$ as $\beta \to\infty$ that is 
    \begin{align}\label{eq:alpha-theta-bound}\alpha_\theta = \frac{1}{2}(1+o_\beta(1)) e^{2\beta}(1+|\sin(2\theta)|)^{-1/2}\end{align}
    In turn, for $\theta \in (-\frac{\pi}{2},\frac{\pi}{2})$, the asymptotics of $\sinh^{-1}(\alpha_\theta \cos\theta)$ as $\beta \to\infty$ are 
    \begin{align*}
        \eta_1 = (1+o_\beta(1))\log (2 \alpha_\theta \cos \theta) = 2\beta (1+o_\beta(1))\,,
    \end{align*}
    where the $o_\beta(1)$ is uniform over compacts of $\theta$. 
    For $\theta \in [0, \pi/2)$, the second quantity $\eta_2$ satisfies a lower bound of 
    \begin{align*}
        \eta_2 \ge \begin{cases}\frac{1}{2} e^{2\beta} \theta (1+o_\beta(1)) & \theta \lesssim e^{-2\beta} \\ 
        2\beta  (1+o_\beta(1)) & \theta \gtrsim e^{-2\beta}\end{cases}
    \end{align*}
    In total, we get the asymptotics for every $\theta\in [0,\pi/2)$ that 
    \begin{align}\label{eq:tau-asymptotics}
        \tau_\beta(\theta) = 2\beta (\cos \theta  + \sin\theta) (1+o_\beta(1))\,.
    \end{align}
    In particular, it satisfies the desired lower bound, and that holds for all $\theta \in (-\pi/2,\pi/2)$ by symmetry of $\tau_\beta$. 
        
    We now move to the uniform positive stiffness. Differentiating $\tau$ twice, and subsequently using that $\cosh(\sinh^{-1}(x)) = \sqrt{1+x^2}$, we get that the stiffness $\tau''_\beta(\theta) + \tau_\beta(\theta)$ is given by 
    \begin{align}\label{eq:stiffness-asymptotics}
        \tau(\theta) + \tau''(\theta) = \frac{\alpha_\theta}{\sin^2\theta (1+\alpha^2_{\theta}\cos^2 \theta )^{1/2} + \cos^2\theta (1+ \alpha^2_{\theta}\sin^2 \theta )^{1/2}}\,.
    \end{align}
    To see that this is uniformly bounded by a constant independent of both $\theta$ and of $\beta$,  we can use the upper bound on the denominator by $\sqrt{a + b} \le \sqrt{a} + \sqrt{b}$, to get 
    \begin{align*}
        \tau(\theta) + \tau''(\theta) \ge \frac{a_\theta}{ 1 + \alpha_\theta \sin^2 \theta |\cos \theta|  + \alpha_\theta \cos^2 \theta |\sin\theta|} 
    \end{align*}
    Bounding the sine and cosine terms all by $1$, we get $\tau(\theta) + \tau''(\theta) \ge \alpha_\theta/(1+2 \alpha_\theta)$ which for $\alpha_\theta\ge 1$ (which it will be for all $\beta$ large) is at least $1/3$ say. 
\end{proof}

The first thing to deduce is that the uniformity of the positive stiffness bound (item 2 in Lemma~\ref{lem:uniform-in-beta-bounds-on-surface-tension-and-stiffness}) implies a uniform \emph{sharp triangle inequality}. Indeed, the proof of this implication is a statement in convex geometry, and therefore the change in constant between one to the other is fully $\beta$-independent, as seen in Lemma~2.1 of~\cite{IoffeLD} as well as Proposition 2.1 of~\cite{PfisterVelenik}. By the latter, with item 2 of Lemma~\ref{lem:uniform-in-beta-bounds-on-surface-tension-and-stiffness}, we have for all $\beta>\beta_0$ and all vectors $x,y$, that 
    \begin{align}\label{eq:sharp-triangle-ineq}
         \tau (x) +  \tau(y) -  \tau(x+y) \ge \kappa_\beta (\|x\| + \|y\| - \|x+y\|)\,, \qquad \text{for a } \quad \kappa_\beta \ge \frac{1}{3}\,.
    \end{align}
    where for vector $x$, we are using $\tau(x) := \|x\|\tau_\beta(\frac{x}{\|x\|})$. 

    For the specific case where $x+y$ forms a horizontal base of a triangle, when $\beta$ is large, deviations from the straight line should be exponentially in $\beta$ unlikely, whereas the positive stiffness above just will imply they are uniformly unlikely. Indeed, if the angle is 45 degrees, then as $\beta\to\infty$ the deviations do not become more unlikely, because the law converges to the uniform distribution over up-right paths, or equivalently simple random walk. Conversely, it is easy to see that $\tau_\beta''(0) \asymp e^{2\beta}$ for $\beta$ large. The following will be used to show that deviations from flat interfaces are not more unlikely as $\beta\to\infty$ than this amount.   

    \begin{lem}\label{lem:tau-double-prime-near-zero-angle}
        For all $\beta>\beta_0$, 
        \begin{align}\label{eq:tau''-upper-bound}
            \sup_{\theta \in [-\pi/4,\pi/4]} |\tau''_\beta(\theta)| \le e^{ 2\beta } \,.
        \end{align}
        In the other direction, we have for sufficiently small $c$ (independent of $\beta$), for all $\beta>\beta_0$, 
        \begin{align}\label{eq:tau''-lower-bound}
            \inf_{\theta \in [-ce^{ - 2\beta},ce^{- 2\beta}]} \tau_{\beta}''(\theta) \ge \frac{1}{8} e^{2\beta}
        \end{align}
    \end{lem}
    \begin{proof}
        For the upper bound~\eqref{eq:tau''-upper-bound}, recall the large $\beta$ asymptotics from~\eqref{eq:tau-asymptotics}--\eqref{eq:stiffness-asymptotics}    \begin{align*}
        \tau''(\theta) = \frac{\alpha_\theta}{\sin^2\theta (1+\alpha^2_{\theta}\cos^2 \theta )^{1/2} + \cos^2\theta (1+ \alpha^2_{\theta}\sin^2 \theta )^{1/2}} - 2\beta (\cos\theta + \sin\theta)(1+o_\beta(1))
    \end{align*}
    We can trivially lower bound the denominator of the first term by $1$, and upper bound the second term by $2\sqrt{2}\beta(1+o_\beta(1))$, to get for large $\beta$ that 
    \begin{align*}
        \tau''(\theta) \le \alpha_\theta + 4\beta\le \frac{1}{2}(1+o_\beta(1))e^{2\beta} 
    \end{align*}
    yielding the claimed bound. 

       For the lower bound~\eqref{eq:tau''-lower-bound}, we will use a quartic Taylor expansion of the function~$\tau$:
        \begin{align}\label{eq:tau''-expansion}
            \tau''(\theta) \ge \tau''(0) - \frac{1}{2}\sup_{\xi \in [-\theta,\theta]} | \tau^{(4)}(\xi) | \theta^2
        \end{align}
        By its exact expression, 
        \begin{align*}
            \tau''(0) = \alpha_0 - \tau(0) = \frac{1}{2}(1+o_\beta(1)) e^{2\beta}\,.
        \end{align*}
        We now investigate the behavior of the fourth derivative. Let 
        \begin{align*}
            Q(\theta) &= \sin^2(2\theta) + (2/M)^2 \cos^2(2\theta)\,, \\ 
            Q'(\theta) &= 4\sin(2\theta)\cos(2\theta) (1-(2/M)^2)\,, \\ 
            Q''(\theta) & = 8 \cos (4\theta) (1-(2/M)^2) \,.
        \end{align*}
        Then, in terms of these derivatives of $Q$, we have 
        \begin{align*}
            \alpha_\theta & =  \frac{M}{(1-(2/M)^2)^{1/2}} (1+\sqrt{Q})^{-1/2}\\ 
            \alpha'_\theta & = - \frac{1}{4} \frac{M}{(1-(2/M)^2)^{1/2}} \frac{Q'}{\sqrt{Q}(1+\sqrt{Q})^{3/2}} \\ 
            \alpha''_\theta & = \frac{M}{(1-(2/M)^2)^{1/2}} \Big(\frac{Q'^2(2+5\sqrt{Q})}{16 Q^{3/2}(1+\sqrt{Q})^{5/2}} -  \frac{Q''}{4\sqrt{Q}(1+\sqrt{Q})^{3/2}}\Big)
        \end{align*}
        In particular, as long as $0\le \theta \le e^{-2\beta}$ so that $\theta M\le 1$, since $Q(\theta)  = (1+o_{\beta,\theta}(1))( 4\theta^2 + (2/M)^2)$, the first term in $\alpha''_\theta$ satisfies
    \[
    \frac{M}{(1-(2/M)^2)^{1/2}} \frac{Q'^2(2+5\sqrt{Q})}{16 Q^{3/2}(1+\sqrt{Q})^{5/2}}\asymp \frac{M \theta^2}{(\theta^2 + (2/M)^2)^{3/2}}\,,
    \]
where the implicit constants in $\asymp$ are uniform over $\beta > \beta_0$ and $|\theta| \le e^{-2\beta}$. The second term in $\alpha''_\theta$ is $
    \asymp M/(\theta^2 + (2/M)^2)^{1/2}$. 

    Now towards $\tau^{(4)}$, define 
    \begin{align*}
        D(\theta) = \sin^2(\theta) (1+ \alpha_\theta^2 \cos^2 \theta)^{1/2} + \cos^2 \theta(1+\alpha_\theta^2\sin^2 \theta)^{1/2} \asymp 1+ \theta M + \theta^2 M
    \end{align*}
    If $\theta \in [0,e^{-2\beta}]$ then $D\asymp 1$. 
    With that notation, note that (dropping $\theta$ arguments for readability), 
    \begin{align*}
        \tau^{(4)}(\theta) = \frac{d^2}{d\theta^2}  \frac{\alpha}{D} - \tau'' =   \frac{\alpha'' D - \alpha D''}{D^2} - 2\frac{(\alpha' D - \alpha D')D'}{D^3}\,.
    \end{align*}
    Consider the asymptotics of each of these terms: firstly, 
    \begin{align*}
        \alpha' \asymp M \frac{\theta}{(\theta + (2/M))}
    \end{align*}
   If $\theta \in [0,e^{-2\beta}]$ then $\alpha'\asymp M^2 \theta$. Next,
    \begin{align*}
         \alpha'' \asymp \frac{M\theta^2}{(|\theta| + (2/M))^{3}}  + \frac{M}{(|\theta| + (2/M))} 
    \end{align*}    
    If $\theta \in [0,e^{-2\beta}]$ then $\alpha''\asymp M^4 \theta^2 + M^2\asymp M^2$. Next,  
    \begin{align*}
        D' & \asymp  \theta \alpha + \theta^2 \frac{\alpha \alpha' - \alpha^2 \theta}{(1+\alpha)} + \theta (1+\theta \alpha) + \frac{ \alpha \alpha' \theta^2 + \alpha^2 \theta}{(1+\alpha \theta)} \\&  \lesssim \theta e^{2\beta} + \theta^2 e^{2\beta}\frac{e^{2\beta}\theta}{(|\theta| + e^{-2\beta})} +  \theta^2 e^{2\beta}  + \theta e^{4\beta} \lesssim e^{4\beta}\theta + \theta^2 e^{2\beta}\frac{e^{2\beta}\theta}{(|\theta| + e^{-2\beta})}
    \end{align*}
    where constants in $\lesssim$ are uniform in $\beta>\beta_0$ and $\theta \le e^{-2\beta}$. 
    If $\theta \in [0,e^{-2\beta}]$ then $D'\lesssim M^2 \theta + M^3 \theta^3\lesssim M^2\theta$. 
    Finally, explicitly calculating $D''$ and doing the same asymptotics gives 
    \begin{align*}
        D'' & \lesssim (1+ \theta^2)(1+M +M \theta) + \theta \Big(\frac{M \alpha' + M^2 \theta}{M} + \frac{M \alpha' \theta^2 + M^2 \theta}{1 + M \theta}\Big) \\ & \qquad + \theta^2 \frac{\alpha'^2 + M \alpha'' + M \alpha'\theta + M^2}{M} + \theta^2 \frac{(M \alpha' + M^2 \theta)^2}{M^3} \\
        &  \qquad + \frac{\alpha'^2 \theta^2 + M \alpha '' \theta^2 + M \alpha' \theta + M^2}{1+ M\theta} + \frac{(M \alpha'\theta^2 + M^2 \theta)^2}{(1+ M \theta)^3}
    \end{align*}
    Working through this term by term and using the above asymptotics, we find that if $\theta \in [0,e^{-2\beta}]$ then $D'' \lesssim M^2$
    where the constant hidden in the $\lesssim$ is universal.  
    
    We thus get for constant $c$ that for some $C(c)$ independent of $\beta$, we have  
    \begin{align*}
        \sup_{\xi \in [-ce^{ - 2\beta},ce^{-2\beta}]}| \tau^{(4)}(\xi)| \le C(c) \cdot M^3 \le C e^{6\beta}\,.
    \end{align*}
    When multiplied by $\theta^2$ for $|\theta|\le \delta e^{- 2\beta}$ for $\delta$ universal only depending on $C(c)$, this is less than $\frac{1}8 e^{2\beta}$ say and in~\eqref{eq:tau''-expansion} is at most $(1/2)\tau''(0)$.  
    \end{proof}

By the standard random-line representation of the Ising model and the Kramers--Wannier duality, the pre-limiting expression for the low-temperature surface tension in~\eqref{eq:surface-tension-def} is exactly given by a two-point correlation function of the Ising model at the dual (high) temperature $\beta^*$ defined by $\sinh(2\beta) \sinh(2\beta^*) =1$. Namely, for all $\beta>\beta_c$, 
\begin{align}\label{eq:interface-two-pt-function-duality}
    \frac{{Z}_{\pm,\theta,N,N'}}{Z_{+,N,N'}} = \langle \sigma_{(0,0)} \sigma_{(x,x\tan\theta)}\rangle_{\beta^*,\Lambda_{N,N'}^*}
\end{align}
where the right-hand side $\langle \cdot \rangle_{\beta^*,\Lambda^*}$ denotes the Gibbs expectation on the planar dual of $\Lambda$. We will move back and forth between these representations, and refer the reader to e.g.,~\cite{PfisterVelenik} for details. For ease of notation, with $\beta$ understood from context, we often drop $\beta^*$ from the two-point function. 

Our next aim is to give finite-$N$ lower and upper bounds on two-point functions appearing in~\eqref{eq:interface-two-pt-function-duality}, bounding the corrections to $e^{ - \tau(\theta) \|(x,x\tan \theta)\|}$. Since point-to-plane connectivities are sub-additive, $e^{ - \tau(\theta_v)\|v\|}$ always provides a (non-asymptotic) upper bound on $\langle \sigma_0 \sigma_v\rangle $ always: 
\begin{align*}
    \langle \sigma_0 \sigma_v\rangle_{\beta^*,(\Z^2)^*}  \le \exp( - \tau_\beta (\theta_v) \|v\|)
\end{align*}
where $\theta_v$ denotes the angle the vector $v$ makes. 
However, getting a finite-$N$ lower bound without a coefficient that deteriorates with $\beta\to \infty$ is formally more challenging to establish. Indeed, the argument surrounding the lower bound is that by random walk heuristics, there is a $1/\sqrt{\|v\|}$ chance that an interface following angle $\theta_v$ ends up passing exactly through the vertex $v$. A naive approach of just having a variance bound, and then forcing the path into the vertex $v$ would cost an amount that deteriorates exponentially quickly as $\beta \to \infty$.

Thus, we must follow the derivation of the Ornstein--Zernike asymptotics of~\cite{DKS,VelenikThesis,CIV03} more carefully, and in steps where there is a constant $c(\beta)$, replace it with a constant that is uniform in $\beta$. That is the subject of the following subsection.

\subsection{Uniform-in-$\beta$ two-point function estimates}

Our main aim in this subsection is to establish the following asymptotics on two-point functions (which by the random line representation of the Ising model are dual to interface probabilities at low temperature, and in fact our proof will leverage this duality). In the statement of the following lemma, let $S_n$ be the strip $\{1,...,n\} \times \mathbb Z$, and let the vertex $0^* = (\frac{1}{2},\frac{1}{2}) \in S_n^*$.

\begin{lem}\label{lem:OZ-asymptotics-beta-indep}
     There exists $C,\beta_0$ such that the following holds for all $n \ge 1$ and all $\beta>\beta_0$. For all $v^* = (n,n \tan \varphi) + (\frac{1}{2},\frac{1}{2}) \in (\Z^2)^*$ for $\varphi \in [-\frac{\pi}{4}, \frac{\pi}{4}]$, one has 
    \begin{align*}
   \frac{C^{-1}}{\sqrt{1+ne^{-2\beta}+n\tan \varphi}} e^{ - \tau_\beta(\varphi) \|v\|} \le \langle \sigma_{0^*} \sigma_{v^*} \rangle_{\beta^*, S_n^*}\le \frac{C}{\sqrt {1+ne^{-2\beta}+n\tan \varphi}} e^{-\tau_\beta (\varphi ) \|v\|}\,.
    \end{align*}
\end{lem}

    The above is capturing the transition between two regimes, where if the interface is (approximately) at slope zero, the variance for the effective random walk is of order $ne^{-2\beta}$, whereas if it is at strictly positive slope, there is residual variance even as $\beta \uparrow \infty$.

\begin{remark}\label{rem:applicability-of-OZ-lemma}
    Next, note that by the GKS inequality, the upper bound of Lemma~\ref{lem:OZ-asymptotics-beta-indep} also holds for all domains $\Lambda^* \subset S_n^*$. It also holds on all bigger domains $\Lambda^* \supset S_n^*$, because by cluster expansion enlarging the domain only costs a $|1\pm \epsilon_\beta| \le 2$ factor.  Thus, Lemma~\ref{lem:OZ-asymptotics-beta-indep} gives the uniform-in-$\beta$ and $n$ analogues of both Lemma 2.1 and Eq.~(2.3) of~\cite{LMST} which are used extensively there for the ultimate mixing time goal. In particular, if $n\ge e^{ \beta/10}$, 
    \begin{align*} 
           \frac{C^{-1} e^{-\beta}}{\sqrt{n}} e^{ - \tau_\beta(\varphi) \|v\|} \le \langle \sigma_{0^*} \sigma_{v^*} \rangle_{\beta^*, S_n^*}\le C\Big(\frac{ e^{\beta}}{\sqrt {n}}\vee 1\Big) e^{-\tau_\beta (\varphi ) \|v\|}\,.
    \end{align*}
\end{remark}

Though we stated the above lemma in the dual, high-temperature form, we will prove it using cluster expansion and in the low-temperature form~\eqref{eq:interface-two-pt-function-duality}. We follow the proof of~\cite{DKS}, and largely borrow the Chapter 4 of~\cite{DKS} for ease of comparison. Let us recall some of this notation. After rewriting the law of Ising interfaces, with some cluster expansion and viewing it as a random walk with ``decorations", they unpin the right endpoint, so that they are considering a directed polymer model starting at $0 \in \mathbb Z^2$, and of length $N$. In order to direct the endpoint in the direction $\varphi$, they introduce a tilt of the measure on polymers for the slope, parametrized by $H$. Then 
\begin{itemize}
    \item $P_{N,H}$ is the law of this Ising interface polymer $S_{N,H}$ for $x$-axis distance $N$, with tilt $H$;
    \item The quantity $h$ is the right endpoint of the walk, i.e., $h = S_{N,H}(N)\cdot e_2$. 
    \item The expected value of $h$ is denoted $M_{N,H}$, and the variance of $h$ is denoted $D_{N,H}$. 
\end{itemize}  

The proof of Lemma~\ref{lem:OZ-asymptotics-beta-indep} goes by showing local central limit theorem behavior, with the correct finite-$N$ variance, for this tilted polymer. 
We split the consideration into two regimes, one where $N\gg e^{2\beta}$, and the other where $N\ll e^{3\beta}$. The following lemma is the analogue of Proposition 4.10 of~\cite{DKS} for $N\ge e^{2^+ \beta}$ with its $\beta$-dependencies quantified.

\begin{lem}\label{l:lclt.fourier}
    There exists $\beta_0>0$ such that for all $\beta>\beta_0$ and $N \ge e^{\frac{11}{5}\beta}$ and $|H|\leq 2-\beta^{-1}\delta$, 
    \begin{align*}
        \sup_{k\in \mathbb Z} \Big| D_{N,H}^{1/2}  P_{N,H}(h(S_{N,H}) = k)  - e^{ - \frac{1}{2D_{N,H}} (k - M_{N,H})^2}\Big| \le C(\delta) e^{-\beta/100} \,, 
    \end{align*}
    for a constant $C(\delta)$ only depending on $\delta$. 
\end{lem}

\begin{proof}
    We follow Proposition 4.10 of~\cite{DKS}, describing the quantitative bounds on each of the terms that were only implicit in that proof. The left-hand side of (4.10.3) of~\cite{DKS} which (by their (4.10.21)--(4.10.22)) is given by 
    \begin{align*}
        \sup_{q\in \Z}  \Big| (D_{N,H})^{1/2} P_{N,H}(h=q) - \frac{1}{\sqrt{2\pi}} e^{ - \frac{1}{2 D_{N,H}}(q-M_{N,H})^2}\Big| \le \frac{1}{2\pi} (J_1 + J_2 + J_3 + J_4 ) \,,
    \end{align*}
    where for constants $A=e^{\beta/100}$ and $\alpha$, we define 
    \begin{align*}
        J_1 & = \int_{-A}^A |\hat \chi_{N,H}(t) - \exp( - t^2/2)| dt \,,  & \qquad   J_2 &  = \int _{|t|>A} \exp( - t^2/2) dt\,, \\ 
        J_3 & =  \int_{A\le |t|\le \alpha (D_{N,H})^{1/2}} |\hat \chi_{N,H}(t)|dt\,,  &  J_4 & = \int_{\alpha(D_{N,H})^{1/2} \le |t|\le \pi (D_{N,H})^{1/2}} |\hat \chi_{N,H}(t)| dt\,. 
    \end{align*}
    where $\hat \chi_{N,H}(t)$ is the normalized characteristic function 
    \begin{align*}
        \hat \chi_{N,H}(t) = \chi_{N,H}(t (D_{N,H})^{-1/2})  \exp( - it M_{N,H}(D_{N,H})^{-1/2}) \qquad \text{where} \qquad \chi_{N,H} = \sum_{q\in \Z} e^{ itq} P_{N,H}(q)\,.
    \end{align*}
     Note that as in (4.10.6)--(4.10.19) of~\cite{DKS}, for a remainder term $R_{N,H}(t)$, one has 
     \begin{align}\label{eq:log-characteristic-function-expansion}
         \log \hat \chi_{N,H}(t) &  = - \frac{t^2}{2} + \frac{t^3}{6} (D_{N,H})^{-3/2} R_{N,H}(t) \\ 
         |R_{N,H}(t)| &\le  C(\delta) \sup_{(t,H)\in  G(\delta)} |\log \chi_{N,H}(t)| \\ & \le C(\delta) \sup_{|H|< 2-\beta^{-1}\delta/2} N \Big(\tilde C(\delta) e^{ - \beta ( 2- |H|)} + 2e^{- 4(\beta - \beta_0)}\Big)  \nonumber\,,
     \end{align}
     where $G(\delta) = \{(t,H): H\in \mathbb R\,,\, |H| < 2-\beta^{-1} \delta/2\,,\, t\in \mathbb C\,,\,|\text{Im}(t)|\le \frac{\delta}{3}\}$. 

    Let us now bound each of $J_1,...,J_4$ for $\beta$ sufficiently large, and all $N \ge e^{\frac{11}{5}\beta}$. 
    By (4.9.12) of~\cite{DKS}, 
    \[
        D_{N,H} \ge \frac{1}{8} N \exp( - \beta(2-|H|)) \ge \frac{1}{8} e^{\beta},
    \]
    and, therefore,
    \begin{align*} \sup_{\substack{t\in [-A,A],\\ |H|<2-\beta^{-1} \delta/2}} \frac{|t^3R_{N,H}(t)|}{6D_{N,H}^{3/2}} &\leq
    \sup_{|H|<2-\beta^{-1} \delta/2} \frac{C(\delta)  \Big(\tilde C(\delta) e^{ - \beta ( 2- |H|)} + 2e^{- 4(\beta - \beta_0)}\Big)}{6\cdot 8^{-3/2}\exp( - \frac32\beta(2-|H|))} \frac{e^{3\beta/100}N}{N^{3/2}}\\
    &\leq C'(\delta)\frac{e^{53\beta/100}}{N^{1/2}} \le \frac{1}{2}
    \end{align*}
    provided $\beta_0$ is large enough.   Since $|e^{x}-1|\le 2|x|$ provided $|x|\le 1/2$, combining the above equation with \eqref{eq:log-characteristic-function-expansion} gives

    \[
    |\hat \chi_{N,H}(t) - \exp( - t^2/2)|\leq 2e^{-\frac{t^2}{2}}\Big|\frac{t^3}{6} (D_{N,H})^{-3/2} R_{N,H}(t)\Big|\,,
    \]
    and so
    \begin{align*}
        J_1\leq \int_{-A}^A 2e^{-\frac{t^2}{2}}C'(\delta)\frac{e^{53\beta/100}}{N^{1/2}}  dt \leq C(\delta)e^{-\beta/4}\,.
    \end{align*}
    Next, observe that by standard Gaussian tail bounds, $$J_2 \le C e^{ - A^2/2}.$$  For the third term, by (4.10.23)--(4.10.25) of~\cite{DKS}, we have 
    \begin{align*}
        J_3 \le \int_{A\le |t|\le \alpha (D_{N,H}^{1/2})}e^{-t^2/4} dt \le C e^{ - A^2/4}
    \end{align*}
    so long as $\alpha$ is such that $C'(\delta) \alpha /6 \le 1/4$, i.e., as long as $\alpha$ is sufficiently small (not depending on $\beta,N$). 
Finally, for term $J_4$, equation (4.10.29) of~\cite{DKS}, we get 
    \begin{align*}
        J_4\leq\int_{\alpha (D_{N,H})^{1/2}\le |t|\le \pi (D_{N,H})^{1/2}} \Big( 1- \frac{\alpha^2}{4e^{2\beta}}\Big)^N dt \le \pi (D_{N,H})^{1/2} \Big( 1- \frac{\alpha^2}{4e^{2\beta}}\Big)^N
    \end{align*}
    By Proposition 4.9 and (4.9.11) in~\cite{DKS}, we get 
    \begin{align*}
        N^{-1} D_{N,H} \le \frac{1}{2} \frac{ e^{|H|\beta } + e^{ - |H|\beta}}{e^{2\beta} + e^{ - 2\beta}} + e^{ - 4(\beta - \beta_0)} + N^{-1} e^{ - 4(\beta - \beta_0)} \le \frac{1}{2} + 2e^{ - 4(\beta - \beta_0)} \le 1\,,
    \end{align*}
    for $\beta$ large (regardless of $N$). Therefore, so long as $N\ge e^{\frac{11}{5}\beta}$, 
    \begin{align*}
        J_4 \le \pi e^{ - \frac{\alpha^2}{4} Ne^{-\frac1{10}\beta}}\,,
    \end{align*}
    which for any fixed $\alpha$ (since $\alpha$ was only small depending on $\delta$, not $\beta,N$), is smaller than $e^{ - \beta}$ if $N\ge e^{\frac{11}{5}\beta}$, for universal large $\beta_0$. 
\end{proof}

On the other hand, for small $N$ relative to $\beta$, we compare to the polymer which only takes ``tame" increments, meaning the interface has no overhangs, and therefore is a true random walk with tilted, geometrically decaying, increments. The notation for the tame polymer has an $\infty$ superscript compared to the general polymer measure. 

\begin{lem}\label{l:wild.coupling}
    There exists $\beta_0'$ such that for all $\beta>\beta_0'$ and $N \le e^{\frac{11}{5}\beta}$, for some absolute constant $C$, $P_{N,H}$ has total-variation distance at most $CNe^{-4\beta}\leq \frac1{100}N^{-\frac{2}{3}}$ to $P_{N,H}^\infty$. Here, $P_{N,H}^\infty$ is the law of a sum of independent random variables $S_{N,H}^\infty = \sum_{i=1}^N X_i^\infty$ with $X_i^\infty$ having law proportional to $e^{ - 2\beta |k|+\beta Hk}$ for $k\in \Z$.
\end{lem}

\begin{proof}  
    The probability under $P_{N,H}$ of the interface not being tame is bounded in terms of corresponding partition functions as 
    \begin{align*}
       \frac{\Xi(N,H) - \Xi(N,H)^\infty}{\Xi(N,H)} =  1- \frac{\Xi(N,H)^\infty}{\Xi(N,H)}  = 1-(\Xi(N,H)/\Xi(N,H)^\infty)^{-1}= 1-e^{-\log (\Xi(N,H)/\Xi(N,H)^\infty)}\,.
    \end{align*}
    Writing the partition function ratio as $\hat \Xi(N,H) =\Xi(N,H)/\Xi(N,H)^\infty$ and using the bound of (4.8.1) of~\cite{DKS} that $|\log \hat \Xi(N,H)|\le N e^{ - 4(\beta - \beta_0)}$, we deduce that as long as $N e^{ - 4(\beta - \beta_0)} \le \frac{1}{2}$, then 
    \begin{align*}
        P_{N,H}(\text{not tame}) =  \frac{\Xi(N,H) - \Xi(N,H)^\infty}{\Xi(N,H)} \le 2N e^{ - 4(\beta - \beta_0)}\,.
    \end{align*}
    Since $2Ne^{- 4(\beta - \beta_0)} \le \frac{1}{2}$ if $N \le e^{11\beta/5}$ for large enough $\beta$ (not depending on $N$) and since
$\|\mu - \mu(\cdot \mid A)\|_{tv} \le 2\mu(A^c)$, we get that the total variation distance is at most $4N e^{ - 4(\beta - \beta_0)}$.  Since $N \le e^{\frac{11}{5}\beta}$ we have that
    \[
    4N e^{ - 4(\beta - \beta_0)}\leq 4e^{4\beta_0}N^{-\frac23}e^{-\beta/3}\leq \frac1{100} N^{-\frac23}.
    \]
    for large enough $\beta$.
\end{proof}

Since for $N\le e^{11\beta/5}$ the law is within total-variation $o(N^{-1/2})$ of a random walk, we can use more classical local limit theorems to treat that case. As in~\cite{DKS}, we use $D_{H}^\infty$ for $\lim_N \frac{1}{N} D_{N,H}^\infty$, and similarly define $M_H^\infty, M_H,D_H$. 

\begin{lem}\label{l:log.concave}
    There exists an absolute constant $C>0$ such that with $X_i^\infty$ i.i.d.\ random variables with law proportional to $e^{ - 2\beta |k|+\beta Hk}$ for $k\in \Z$ and $S_{N,H}^\infty = \sum_{i=1}^N X_i^\infty$, if $\E[S_{N,H}^\infty]=k\in\Z$ then
    \[
        \frac{C^{-1}}{\sqrt{1+ND_H^\infty}}\leq P_{N,H}^\infty(h(S_{N,H}^\infty) = k) \leq \frac{C}{\sqrt{1+ND_H^\infty}}
    \]
\end{lem}
\begin{proof}
Melbourne and Palafox-Castillo showed that \cite[Corollary 2.8]{MPC:23} for integer valued log-concave  distributions $Y$,
\[
\max\Big\{\P\big[Y=\lfloor\E[Y]\rfloor\big],\P\big[Y=\lceil\E[Y]\rceil\big]\Big\}\geq e^{-1}\max_y\P[Y=y]
\]
while Bobkov, Marsiglietti and Melbourne~\cite[Theorem 1.1]{BMM:22} showed that
\[
\frac{1}{\sqrt{1+12\var(Y)}} \leq \max_y\P[Y=y] \leq \frac{2}{\sqrt{1+4\var(Y)}} 
\]
Since $X_i^\infty$ is log-concave and log-concavity is preserved under convolution, $S_{N,H}^\infty$ is also log-concave and $\var(h(S_{N,H}^\infty))=ND_H^\infty$.  Hence we have that
\[
\frac{e^{-1}}{\sqrt{1+12ND_H^\infty}} \leq P_{N,H}^\infty (h(S_{N,H}^\infty)=k)\leq\max_y\P[Y=y] \leq \frac{2}{\sqrt{1+4ND_H^\infty}} \,.
\]
This implies the claim for a $\beta$-independent $C$ large enough.
\end{proof} 

The following estimate provides the necessary estimates for the variances. 

\begin{lem}\label{l:var.equiv}
There exists absolute constants $C,\beta_0>0$ such that $|H|\leq 2-\frac1{10\beta}$ and $\beta>\beta_0$ then
\[
\frac12(e^{-2\beta}+M_H^\infty) \leq D_{H}^\infty \leq C(e^{-2\beta}+M_H^\infty).
\]
The same bound holds for $D_{H}$ and $\frac1{N}D_{N,H}$ up to changing $\frac{1}{2}$ and $C$ to $\frac{1}{4}$ and $2C$.
\end{lem}
\begin{proof}
Without loss of generality assume that $H\geq 0$. By explicit computation, the random variables $X_{i}^\infty$ have the distribution of the difference of independent $\text{Geom}(1-a)$ and $\text{Geom}(1-b)$, where
\[
a=e^{-\beta(2-H)},\qquad b=e^{-\beta(2+H)}\,.
\]
Therefore, we can compute the mean and variance as 
\[
M_H^\infty=\frac{a}{1-a}-\frac{b}{1-b},
\qquad
D_{H}^\infty=\frac{a}{(1-a)^2}+\frac{b}{(1-b)^2}.
\]
Since $\beta(2-H)\ge \frac{1}{10}$, we have $e^{-2\beta}\le a\le e^{-1/10}$
and $0<b=e^{-\beta(2+H)}\le e^{-2\beta}\leq\frac12$. 
Let $c_0:=1-e^{-1/10}>0$ so $1-a\ge c_0$.  
Hence for the lower bound,
\[
D_{H}^\infty\geq \frac{a}{(1-a)^2}\geq \frac12\Big(a + (\frac{a}{1-a}-\frac{b}{1-b})\Big) = \frac12(e^{-2\beta}+M_H^\infty).
\]
For the upper bound, since \(1-a\ge c_0\),
\[
D_{H}^\infty
= \frac1{1-a}M_H^\infty +\frac{b}{(1-a)(1-b)} +\frac{b}{(1-b)^2}
\leq \frac{4}{c_0}(e^{-2\beta}+M_H^\infty).
\]
for large enough $\beta$ since $1-b\geq \frac12$ and $b\leq e^{-2\beta}$. By Proposition~4.9 of~\cite{DKS}  we have that $|D_{H}^\infty-D_{H}| \le e^{-4(\beta - \beta_0)}$ and $|D_{H}^\infty-\frac1{N}D_{N,H}|\leq C_1 e^{-4(\beta- \beta_0)}$ which completes the proof.
\end{proof}

\begin{prop}\label{prop:local-CLT-DKS}
There exist absolute constants $C,\beta_0>0$ such that for all $N\geq 1, -N\leq k \leq N$ and $\beta>\beta_0$, if $NM_{N,H}=k$ then
\begin{equation}\label{eq:lclt1}
\frac{C^{-1}}{\sqrt{1+ND_H}}\leq P_{N,H}(h(S) = k) \leq \frac{C}{\sqrt{1+ND_H}} 
\end{equation}
and
\begin{equation}\label{eq:lclt2}
\frac{C^{-1}}{\sqrt{1+ND_H}}e^{N(F(H)-\beta H k)}\leq P_{N}(h(S) = k) \leq \frac{C}{\sqrt{1+ND_H}} e^{N(F(H)-\beta H k)}
\end{equation}
where $F(H) = \lim_{M} \frac{1}{M} \log \Xi(M,H)$.  
\end{prop}

\begin{proof}
Without loss of generality assume, $k\geq 0$ and so $H\geq 0$ since $M_{N,0}=0$ and $M_{N,H}$ is increasing in $H$. We first show that the $H$ that achieves $N M_{N,H} = k$ is suitably bounded away from $2$. 
Note that for $H=2-(10\beta)^{-1}$, if $\beta\geq 2$ then with $a=e^{-1/10}, b=e^{-4\beta+1/10}$
\[
M_H^\infty = \frac{a}{1-a}-\frac{b}{1-b}\geq 9.
\]
By Proposition~4.9 of~\cite{DKS}, $|M_{N,H}-NM_{H}^\infty|\leq C_1e^{-4\beta}$ and so by monotonicity in $H$, if $NM_{N,H}=k\leq N$ then certainly $H\leq 2-(10\beta)^{-1}$.
By Lemma~\ref{l:lclt.fourier}, we have that when $N\geq e^{\frac{11}{5}\beta}$
\begin{equation}
\frac{C^{-1}}{\sqrt{D_{N,H}}}\leq P_{N,H}(h(S) = k) \leq \frac{C}{\sqrt{D_{N,H}}}.
\end{equation}
which implies \eqref{eq:lclt1} since by Lemma~\ref{l:var.equiv}, $D_{N,H}\geq \frac12e^{\frac15 \beta}\geq 1$ and by Proposition~4.9 of~\cite{DKS}, $\frac1{N}D_{N,H}$ is the same as $D_H$ up to a universal multiplicative constant.  

Now consider the case when $N\leq e^{\frac{11}{5}\beta}$. Combining Lemmas~\ref{l:wild.coupling} and~\ref{l:log.concave} we have that 
\begin{equation}
\frac{C_1^{-1}}{\sqrt{1 +ND_{H}^\infty}}-\frac{1}{100} N^{-2/3}\leq P_{N,H}(h(S) = k) \leq \frac{C_1}{\sqrt{1 +ND_{H}^\infty}}+\frac{1}{100} N^{-2/3}\,.
\end{equation}
By Lemma~\ref{l:var.equiv}, the condition $Ne^{-\frac{11}{5}\beta}\leq 1$, and large enough $\beta$ we have that
\[
\frac{C_1^{-1}}{\sqrt{1 +ND_{H}^\infty}} \geq C_3N^{-1/2}  \geq \frac{1}{50} N^{-2/3}\,.
\]
Hence
\begin{equation}
\frac{(2C_1)^{-1}}{\sqrt{1 +ND_{H}^\infty}}\leq P_{N,H}(h(S) = k) \leq \frac{2C_1}{\sqrt{1 +ND_{H}^\infty}}
\end{equation}
which implies the other case of \eqref{eq:lclt1} since $D_{H}^\infty$ is the same as $D_H$ up to a multiplicative constant (by Lemma~\ref{l:var.equiv}).  Finally~\eqref{eq:lclt2} follows from \eqref{eq:lclt1} by noting that
\[
 P_{N,H}(h(S) = k) = \frac{P_{N}(h(S) = k)e^{\beta H k}}{\Xi(N,H)}\,,
\]
and by (4.8.6) of~\cite{DKS},
\[
|NF(H)-\log \Xi(N,H)|\leq C_4 e^{-4\beta}. \qedhere
\]
\end{proof}

The following corollary  gives us exactly Lemma~\ref{lem:OZ-asymptotics-beta-indep} after changing notation back from~\cite{DKS}.  

\begin{cor}
    There exist absolute constants $C, \beta_0>0$ such that for all $N \ge 1$, $k \in \{-N,...,N\}$ and $\beta>\beta_0$, if $H$ is such that $N M_{N,H} = k = N \tan \varphi_{\mathbf{n}}$, 
    \begin{align*}
       \frac{C^{-1}}{\sqrt{1 + N e^{-2\beta} + k}}  e^{ - \tau_{\beta}(\varphi_{\mathbf{n}}) \|(N,k)\| }\le \Xi(N,\mathbf{n}) \le \frac{C}{\sqrt{1 + Ne^{ - 2\beta}  + k}} e^{ - \tau_\beta(\varphi_{\mathbf{n}})\|(N,k)\|}
    \end{align*}
\end{cor}

\begin{proof}
By Proposition~4.12 of~\cite{DKS}, $\tau_\beta(\mathbf{n})$ for $\mathbf{n}\in \mathbb S^1$ and can be associated with its angle, is  
\begin{align*}
    \tau(\mathbf{n}) = ( - \beta^{-1} F(H_{\mathbf{n}})  + H_{\mathbf{n}} \tan \varphi_n) \cos\varphi_n
\end{align*}
where $H_{\mathbf{n}}$ is the $H$ such that $N M_{N,H} = \tan \varphi_{\mathbf{n}}$. 
Then by the previous Proposition~\ref{prop:local-CLT-DKS}, we have 
\begin{align*}
    |\log \Xi(N,\mathbf{n})  - \log \Xi(N,H_{\mathbf{n}})  + \beta H_{\mathbf{n}} h(N,\mathbf{n}) + \frac{1}{2} \log (1+N D_H)| \le C
\end{align*}
for a universal constant $C$, for all $N$. By the inequality for $\log \Xi(N,H)$ to $N F(H)$, the bound of $|H|\le 2$ and $h(N,\mathbf{n}) = N\tan \varphi_{\mathbf{n}} = k$, this implies 
\begin{align*}
    |\log \Xi(N,\mathbf{n}) -  N F(H_{\mathbf{n}}) + \beta H_{\mathbf{n}}\tan\varphi_{\mathbf{n}} + \frac{1}{2}\log (1+ N D_H)| \le C + e^{ - 4(\beta -\beta_0)} \le C'\,.
\end{align*}
At this point, we observe that $d(N,\mathbf{n}) = N/\cos \varphi_{\mathbf{n}}$, and thus this implies 
\begin{align*}
    |\log \Xi(N,\mathbf{n}) + \beta d(N,\mathbf{n}) \tau(\mathbf n) + \frac{1}{2} \log (1+ D_H)| \le C'\,, 
\end{align*}
or in other words, together with Lemma~\ref{l:var.equiv}, we have the claimed bound. 
\end{proof}

\subsection{Uniform bound on vertical oscillations of interface} 
Having established the uniform Ornstein--Zernike asymptotics for low-temperature interfaces in Lemma~\ref{lem:OZ-asymptotics-beta-indep}, in
the next few subsections, we prove analogues of the key equilbrium estimates of~\cite{LMST}. In the analogue statements we prove, the large $\beta$-dependencies are explicit and under control. These subsections will largely follow the notation of~\cite{LMST}, especially its Sections 4--5, for ease of side-by-side comparison.

A slit-strip geometry was important to the interface estimates provided there. We will use $S$ to denote the infinite strip $\{1,...,\ell\}\times \Z$. 
We generally denote by $\eta =\mp$ the boundary conditions, which are $+$ on the boundary vertices on the upper half-space and $-$ on the boundary vertices in the lower half-space (including height zero). 
 Let $\bar S = \bar S(a,b)$ be $S$ setminus the two slits $\{1,...,a\}\times \{0,1\}$ and $\{b,...,\ell\}\times \{0,1\}$. Note that the $\mp$ boundary conditions induce minuses on the top of the slit and pluses on the bottom of the slit along with the boundary conditions on the verticals. 
Each configuration $\sigma$, with these boundary conditions, then induces, on the planar dual $S^*$, a unique open contour connecting $\{a-\frac{1}{2},\frac{1}{2}\},\{b + \frac{1}{2},\frac{1}{2}\}$.
 This open contour is called the \emph{interface} in $\bar S^*$ and will be denoted $\lambda(\sigma)$. For an open contour $\lambda$, $\partial \lambda$ denotes the two vertices of odd degree in $\lambda$. Finally, let $H_i^* = \{\frac{1}{2},...,\ell+ \frac{1}{2},\frac{1}{2}\} \times \{\frac{1}{2},i+\frac{1}{2}\}$ be the vertices at level $i$ in $S^*$. (For more background on the exact definition of these contours/interfaces, and the south-east and south-west splitting rules, we refer the reader to Section 2.4 of~\cite{LMST}.)

 Let us also recall the dual random-line representation of the high-temperature two-point functions $\langle \sigma_{x^*}\sigma_{y^*}\rangle_{\beta^*,\Lambda^*}$. For a finite subgraph $\Lambda^* \subset (\mathbb Z^2)^*$, and an compatible family of contours $\underline{\lambda}$, the weights $q_{\Lambda^*}(\underline{\lambda})$ at $\beta^*<\beta_c$ are defined as in (2.6) of~\cite{LMST}. Then, for $A \subset \Lambda^*$, as described in~(2.7) of~\cite{LMST}, we have the following representation of multi-point functions: 
\begin{align*}
    \pi_{\Lambda^*,\beta^*}\Big[\prod_{x\in A} \sigma_{x} \Big] = \sum_{\underline{\lambda}: \partial \underline{\lambda}= A} q_{\Lambda^*}(\underline{\lambda})\,,
\end{align*}

\begin{lemma}[Replacement for Theorem 5.3 of~\cite{LMST}]\label{lem:maximal-fluctuation-bound}
    There exists $C$ such that for all $\beta>\beta_0$, as long as $\ell \ge b-a\ge e^{\beta/5}$, one has for all $h$ that  
    \begin{align*}
        \pi_{\bar S}^\eta( \sigma: \lambda(\sigma) \text{ reaches } H^*_{h}) \le C (b-a)^{C}\exp( - \kappa_\beta (\tfrac{h^2}{b-a-1}\wedge h )) \qquad \text{for all }h
    \end{align*}
\end{lemma}
\begin{proof}
    The bound is vacuous if $h\le \sqrt{b-a}$ so we can  assume $h\ge e^{\beta/10}$. 
    Since we (uniform-in-$\beta$) polynomial factors, we will use a union bound as in the remark following Theorem 5.3 of~\cite{LMST} rather than their more delicate multi-scale analysis. Let $u$, $v$ denote the endpoints of the unique open countour in $\bar S_n^*$ (the truncation at a large height $\pm n$ of $\bar S$), and let $\text{ht}(\lambda,x) = \max\{y: (x,y)\in \lambda\}$. 

    For any vertex $w$, as in (5.4) of~\cite{LMST} we get 
    \begin{align*}
        \pi_{\bar S_n}^\eta ( \sigma: w\in \lambda(\sigma))  = \Big( \sum_{\lambda: \partial\lambda= \{u,v\}\,,\,w\in \lambda} q_{\bar S_n^*}(\lambda)\Big)/\Big( \sum_{\lambda: \partial\lambda = \{u,v\}} q_{\bar S_n^*}(\lambda)\Big) 
    \end{align*}
    Following the steps of (5.4)--(5.5), replacing their input of Ornstein--Zernike asymptotics with our Lemma~\ref{lem:OZ-asymptotics-beta-indep} (see also Remark~\ref{rem:applicability-of-OZ-lemma}, applicable because the distance of $w$ to $u$ or $v$ is at least $e^{\beta/10}$), the numerator is at most 
    \begin{align*}
       \langle \sigma_u \sigma_w \rangle_{\bar S_n^*}  \langle \sigma_w \sigma_v\rangle_{\bar S_n^*} \le  \frac{C^2 e^{2\beta}}{\sqrt{|u-w||v-w|}} \exp( - \tau(u-w)  - \tau(v-w))
    \end{align*}
    To lower bound the denominator, it is lower bounded via Lemma~\ref{lem:OZ-asymptotics-beta-indep} by 
    \begin{align*}
        \langle \sigma_u \sigma_v \rangle _{\bar S_n^*} \ge \frac{C^{-1} e^{ - \beta}}{\sqrt{|u-v|}} \exp( - \tau(u-v))\,.
    \end{align*}
    Following the logic in~\cite{LMST} from (5.6)--(5.8), gives 
    \begin{align*}
        \pi_{\bar S}^\eta(w\in \lambda(\sigma)) \le \frac{C^3 e^{12\beta}\sqrt{|u-v|}}{\sqrt{|u-w| |v-w|}} \exp\big( - \kappa_\beta (|w-v| + |u-w| -|u-v|)\big)\,.
    \end{align*}
    We separately handle $w$ that have $x$-coordinate less $a$ or bigger than $b$ versus those between $a$ and $b$. Firstly, summing this over $w$ with $x$-coordinate less than $a$, using that for such $w$, $|w-v| \ge |u-v|$, and using from~\eqref{eq:sharp-triangle-ineq} that $\kappa_\beta \ge 1/3$, we get 
    \begin{align*}
        \sum_{w: w_1\le a+\frac{1}{2}}\pi_{\bar S}^\eta (w\in \lambda(\sigma)) \le C e^{ - \kappa_\beta h}
    \end{align*}
    for a $\beta$-independent constant $C$. The contribution from $w$ with $w_1 \ge b-\frac{1}{2}$ is similarly bounded. 
    
    Next for $w$ between $a$ and $b$, following the logic between (5.8) and (5.9) in~\cite{LMST}, we have 
    \begin{align*}
        \pi_{\bar S}^{\eta}( w\in \lambda(\sigma)) \le \frac{C^3 e^{3 \beta}\sqrt{|u-v|}}{\sqrt{|u-w||v-w|}} \exp\big( - \tfrac{6}{5} \kappa_\beta \big( \tfrac{h^2}{|u-v|}\wedge h)\big)
    \end{align*}
    Next, doing a union bound over $w$ whose $x$-coordinate is between $a$ and $b$, we conclude 
    \begin{align*}
        \pi_{\bar S}^\eta(\lambda(\sigma) \text{ reaches }H_h^*) \le \sum_{w\in H_h^*} \pi_{\bar S}^\eta(w\in \lambda(\sigma)) \le C^3 e^{3\beta} |u-v|^{3/2} \exp(- \tfrac{6}{5} \kappa_\beta(\tfrac{h^2}{|u-v|}\wedge h))\,.
    \end{align*}
    After using that $|u-v|\ge e^{ \beta/5}$ to absorb the $e^{3\beta}$ factor, this gives the desired bound. 
\end{proof}

\subsection*{Uniform polynomial lower bound on probability of being above a floor} 
In this section, we give a uniformly (in $\beta$) polynomial lower bound on the probability of the interface connecting $u$ to $v$ being entirely non-negative. 

\begin{lem}[Replacement for Theorem 5.1 of~\cite{LMST}]\label{lem:thm:5.1-replacement}
There exists a $C$ (independent of $\beta$) such that if we consider the infinite strip  $S$ of width $\ell$ with boundary conditions $\eta$, and let $\lambda = \lambda(\sigma)$ be its interface, the following holds. For $i \in \Z$, let $H_i^* = \{ \frac{1}{2},...,\ell + \frac{1}{2}\}\times \{i+\frac{1}{2}\}$. For every $\ell \ge e^{ \beta/5}$,  
\begin{align*}
    \pi_{S}^\eta (\lambda(\sigma) \text{ stays above $H_{-1}^*$}) \ge e^{ - C \beta}{\ell^{-1.1}}\,.
\end{align*}
\end{lem}

\begin{remark}
        Note that $1.1$ could be any constant bigger than $1$. Moreover, Lemma~\ref{lem:thm:5.1-replacement} implies for all $\ell \ge 1$, that the probability of the interface being non-negative is at least $\ell^{ - C}$ for a possibly different $\beta$-independent $C$. (For $\ell \le e^{\beta/5}$ a trivial union bound of the event of the interface deviating from the ground state gives a constant probability of this event.)
\end{remark}

\begin{proof}
    We begin by claiming that in Lemma 5.4 of~\cite{LMST}, the constant $C^\star$ can be taken to be $\beta$-independent. Namely, there exists $\bar C^*$ such that for all $\beta>\beta_0$, all $\ell \ge e^{ \beta/10}$, and all $h$, 
    \begin{align}\label{eq:Lemma-5.4-equivalent}
        \pi_{\bar S}^\eta( \lambda \text{ hits $H^*_{-h + C^* \log h}$ before $H^*_{h - C^* \log h}$ or $(b-\tfrac{1}{2},\tfrac{1}{2})$}) \le \frac{1}{2}  + \frac{\bar C^*}{h}\,.
    \end{align}
    To show that, for Claim 5.5 of~\cite{LMST}, we replace it with the following bound: for all $c$, there exists $\bar C_1^*$ such that for all $\beta>\beta_0$ and for $m\ge e^{\beta/10}$, as long as $|u-v| = b-a \ge e^{\beta/5}$,  
    \begin{align}\label{eq:clm:5.5-replacement}
        \pi_{\bar S}^{\eta}( \text{gn}(\gamma, [a-m,a+m], c\log m) \ge \tfrac{\bar C_1^*}{\beta}\log m) \le m^{-10}\,.
    \end{align}
    where the gain $\text{gn}(A,I,m)$ is as defined in~\cite[(5.21)]{LMST}. 
    Indeed, in that proof, the first step that is asymptotic or has a hidden $\beta$ dependence is the display after (5.22). There, we use the upper bound from Lemma~\ref{lem:OZ-asymptotics-beta-indep}, where if the horizontal distance is at least $e^{\beta/10}$, we use the bound with the division by $\sqrt{n}$, and if less than $e^{\beta/10}$, the bound by $1$.  Either $b-a\ge 10m$, in which case $|z'-v| \ge |u-v| - 2m \ge \frac{1}{2} |u-v|$, or $b-a \le 10m$ in which case at least one of the horizontal distances among $|u-z|,|z'-v|$ is at least $e^{\beta/10}$ and within a factor of two of $|u-v|$. Regardless, we then get the following bound for $\beta$-independent $C$, to refine their display after (5.22):  
    \begin{align*}
        \sum_{\lambda: \partial\lambda = \{u,v\}\,,z,z'\in \lambda} q_{\bar S^*} (\lambda) \le \frac{C e^{ \beta} \exp( - (\tau(u-v) + \tau(z-z') + \tau(v-z'))}{ \sqrt{ |u-v|}} 
    \end{align*}
    Since $\tau(\theta) \ge \tau(0)$ we have the numerator is at most 
    \begin{align*}
        \sum_{x,x': |x-x'|\le c\log m, x\in [a-m,a+m]} \sum_{y} C e^{\beta} e^{ - \tau(0) |u-v|} e^{-c'\tau(0) |y|/m} \sum_{y': |y'-y|\ge C_1^\star \log m} e^{ - c'\tau(0) |y'-y|}
    \end{align*}
    where $c'$ is a universal (geometric) constant, so long as $C_1^* \ge c$. Then using the lower bound $\tau_\beta(0) \ge 3\beta$ from Lemma~\ref{lem:uniform-in-beta-bounds-on-surface-tension-and-stiffness},  we see that as long as $c \gtrsim 1/\beta$, for $C_1^\star  = \bar C_1^*/\beta$ for $\bar C_1^*$ sufficiently large, this is at most $$\frac{m^{-100}}{\sqrt{|u-v|}}  C e^{\beta} e^{ - \tau(0) |u-v|}\,.$$
    Dividing by the lower bound on the denominator of $\langle\sigma_u\sigma_v\rangle$ from Lemma~\ref{lem:OZ-asymptotics-beta-indep},  for any $\bar c$, for $\bar C_1^\star$ large enough (independent of $\beta$),  we get that~\eqref{eq:clm:5.5-replacement} holds for $m\ge e^{\beta/10}$.

    For Claim 5.6 of~\cite{LMST}, it can be written in a way to only improve with $\beta$ as follows: for all $m\ge 1$, 
    \begin{align}\label{eq:clm:5.6-equivalent}
        \pi_{\bar S}^\eta ( \text{every conn.\ comp.\ of $\gamma_{\SW}\!\setminus\! \gamma_{\SE}$ intersecting $([a-m,a+m]\times \mathbb Z)\cap \gamma_{\SE}$ }&\text{has diam.\ $> \tfrac{\bar C_2^*}{\beta} \log m$}) \nonumber\\ 
        & \le m^{-9} \,.
    \end{align}
    To see~\eqref{eq:clm:5.6-equivalent}, we follow the steps of the proof of Claim 5.6 in~\cite{LMST}. When they apply their Theorem~5.3, we apply our Lemma~\ref{lem:maximal-fluctuation-bound} to see that $\gamma_{\SE} \cup \gamma_{\SW} \subset \Lambda^*$ where $\Lambda = \{1,...,\ell\} \times \{-\ell,...,\ell\}$ except with probability $C\ell^C e^{ - \kappa_\beta \ell}\le e^{  \ell/4}$ for all large $\beta$ (because $\kappa_\beta$ is lower bounded by~\eqref{eq:sharp-triangle-ineq} and $\ell\ge e^{\beta/10}$ by assumption). The next step applies Claim 5.5, for which we apply our~\eqref{eq:clm:5.5-replacement} to confine their set $\mathcal I= ([a-m,a+m]\times \Z)\cap \gamma_{\SE}$ to $\frac{\bar C_1^*}{\beta} m \log m \le m^2$ many vertices, except with probability $m^{-10}$. (Note that if $m\le e^{\beta/10}$, then simply apply~\eqref{eq:clm:5.5-replacement} whence the gain is at most $\bar C_1^*$, still with probability $1-m^{-10}$.)
    Finally, their last inequality which is from their Lemma 2.6 has no $\beta$ dependence, except in $\tau_{\beta}(0)$, which by Lemma~\ref{lem:uniform-in-beta-bounds-on-surface-tension-and-stiffness} is at least $3\beta$. Putting those all together exactly as done in the proof of Claim 5.6 in~\cite{LMST} yields~\eqref{eq:clm:5.6-equivalent}.

    Lemma 5.7 of~\cite{LMST} gets replaced by the following bound: For any $w$, define the rectangle $\mathcal R = \{ a- \frac{1}{2}-w,...,a-\frac{1}{2}+w\}\times \{\frac{1}{2}-h,...,h+\frac{1}{2}\}$. Let $\mathcal B$ be the event that $\gamma_{\SE}$ (resp., $\gamma_{\SW}$) exits horizontally from $\mathcal R$ before exiting vertically: we claim that for all $w$,
    \begin{align}\label{eq:lem:5.7-equivalent}
        \pi_{\bar S}^\eta(\mathcal B) \le C e^{3\beta} \exp( -  w/ \bar C_3^\star e^{12\beta} h^2)
    \end{align}
    In the proof, the steps are unchanged with their $C_\beta$ applied from their Lemma 2.1 being $C e^{\beta}$ per our Lemma~\ref{lem:OZ-asymptotics-beta-indep}: then as long as $c(\beta) \ge (3 e \cdot e^{\beta})^2$ and $h\ge 1$, each of their horizontal distances $x_i - x_{i-1}$ are at least $e^{2\beta}$ and the upper bound of Lemma~\ref{lem:OZ-asymptotics-beta-indep} is applicable to give 
    \begin{align*}
        \sum_{\lambda: \partial \lambda = \{u,v\}\,,\, \xi\text{-admissible}} q_{\bar S^*}(\lambda) \le e^{\beta} (2/3)^M \frac{1}{\sqrt{\frac{1}{2} |b-a|}} e^{ - \tau_\beta(u-v)}\,.
    \end{align*}
     Then dividing by the lower bound on $\sum_{\lambda:\partial \lambda= \{u,v\}} q_{\bar S^*}(\lambda)$, using our Lemma~\ref{lem:OZ-asymptotics-beta-indep} in place of their asymptotic bound, using that $|b-a| \ge e^{\beta/10}$, and using that $M \ge w/4ch^2$, we get~\eqref{eq:lem:5.7-equivalent} for some $C, \bar C_3^*$ independent of $\beta$.  

    We now conclude the proof of our analogue to their Lemma 5.4 using the above ingredients. 
    If $|a-b|\le e^{\beta}$, then by a union bound, with high probability the interface will just be the straight horizontal line, and therefore~\eqref{eq:Lemma-5.4-equivalent} holds. Suppose now that $|a-b|\ge e^{\beta}$. 
    
    Using~\eqref{eq:lem:5.7-equivalent} in place of their Lemma 5.7, we get that the interface doesn't travel horizontally by $w=e^{3 \beta} h^4$ without moving vertically by $h$, except with probability $e^{ - h^2/ \bar C_3^*}$. Thus, we consider horizontal gains on distances of size $m=e^{3 \beta} h^4$, whence in~\eqref{eq:clm:5.5-replacement}--\eqref{eq:clm:5.6-equivalent}, the bounds on the gains and discrepancies between $\gamma_{\SW} \oplus \gamma_{\SE}$ will be of size $\bar c \log h$ for a uniform constant $\bar c$ (after the logarithm applied to $m$, the $\beta$ factor cancels with the $\frac{1}{\beta}$ in those equations). Thus, we conclude that for a $\bar C$ independent of $\beta$, the contour $\gamma_{\SE}$ hits $H^*_{-h + \bar C \log h}$  before hitting $H^*_{h- \bar C \log h}$ or $v$ with probability at most $\frac{1}{2} + h^{-8}$.

    At last, we can verify that with~\eqref{eq:Lemma-5.4-equivalent} in place of their Lemma 5.4, we can execute their proof of Theorem 5.1, lower bound (our Lemma~\ref{lem:thm:5.1-replacement}). Let $w_0$ be a large constant independent of $\beta$, and let $w_i = 2w_{i-1} - 2 \bar C\log w_{i-1}$, for $\beta$-independent $\bar C$ from our above. Then $w_j \ge \bar c2^j$ for a $\beta$-independent $\bar c$. The events $A, B$ are defined as there. 

    The forcing for $A_0$ is done as there, having a probability at least $(\frac{1}{2} e^{ - 8\beta})^{2w_0+2}   \ge e^{ -\bar c_0 \beta}$ for a $\beta$-independent $\bar c_0$. Next, we reason that the constant $c$ in their Claim 5.8 can be replaced by a $\beta$-independent $\bar c$. Indeed, all the steps in that proof are by monotonicity arguments, except when using their Lemma 5.4, we plug in our replacement~\eqref{eq:Lemma-5.4-equivalent} in its place, which has the $\beta$-independent $\bar C^*$ on the $1/h$. Thus, for $\pi_S^j$ as in their (5.26), we get for all $j$ that 
    \begin{align*}
        \pi_S^j(D_{j+1}^L \mid \sigma_{U_j} = \eta_{U_j})\ge \frac{1}{2} - w_j^{-10} \ge \frac{1}{2} - \bar c 2^{-j}\,.
    \end{align*}
    Now, instead of stopping at $K+ \frac{1}{2} \log_2 \ell$, we perform the recursion up to $j= (\frac{1}{2} +\frac{1}{100}) \log_2 \ell$ and observe that 
    \begin{align*}
        \pi_S(A_{(\frac{1}{2} + \frac{1}{10})\log_2 \ell} ) +  \pi_S(B_{( \frac{1}{2} + \frac{1}{100}) \log_2 \ell}) \ge e^{ - \bar c_0 \beta}\prod_{i=1}^{(\frac{1}{2} + \frac{1}{100}) \log_2 \ell} (\frac{1}{4} - \bar c 2^{-j}) \ge  \bar c'e^{ - \bar c_0 \beta}  \ell^{-1.1}\,.
    \end{align*}
    Finally, we show that 
    \begin{align*}
        \pi_S(B_{1+(\frac{1}{2} + \frac{1}{100})\log_2 \ell}\mid A_{(\frac{1}{2} + \frac{1}{100})\log_2 \ell}) \ge \frac{1}{2}\,.
    \end{align*}
    After the same monotonicity arguments as those for this step in~\cite{LMST}, this reduces to the complement of the probability of a contour reaching height $w_{j+1} - w_j > \ell^{\frac{1}{2} + \frac{1}{100}}$. By Lemma~\ref{lem:maximal-fluctuation-bound}, since $\ell \ge e^{\beta/5}$ and $\kappa_\beta \ge 1/3$ by~\eqref{eq:sharp-triangle-ineq}, this probability is going to $1$ in a $\beta$-independent manner for large $\ell$, which can be ensured by taking $\beta_0$ big. 
\end{proof}

\subsection{$\beta$-independent analogue of Proposition 4.4 of~\cite{LMST}}

As a corollary of Lemma~\ref{lem:thm:5.1-replacement}, we replace the first of the two main equilibrium estimates used in~\cite{LMST}, Proposition 4.4 of~\cite{LMST}, with the following version which identifies the exact exponential rate for the height fluctuations of a slope-$0$ Ising interface when looking at the moderate deviations regime. (When looking at the large deviations regime, the rate will actually cease to be exponential in $\beta$.) The key distinction in our proposition is the typical interface height exhibits the Gaussian tails at $e^{-\beta}\sqrt{\ell}$ rather than $\sqrt{\ell}$ reflecting the exponentially decaying variance when $\beta$ gets large.   

We borrow the notation of~\cite{LMST} that for rectangular domains, boundary conditions $(-,-,+,-)$ are used to denote ones that are all-$+$ on the south side, and all-$-$ on the other three sides. 

\begin{prop}[Replacement for Proposition 4.4 of~\cite{LMST}]\label{prop:vertical-maximum-bound-replacement}
    There exists $c_1,c_2$ such that for any $\beta>\beta_0$ the following holds. Let $R$ be a rectangle of width $\ell$ and height at least $e^{-\beta} \bar \alpha \sqrt{\ell}$. Then, for any $ \delta \bar \alpha \le \sqrt{\beta \log \ell}$, and every $\ell \ge \beta e^{2\beta}$, we have 
    \begin{align*}
        \pi^{(-,-,+,-)}(\lambda(\sigma) \text{ reaches }\delta e^{- \beta} \bar \alpha \sqrt{\ell})\le \ell^{c_1}e^{ - c_2 (\delta \bar \alpha)^2}\,.
    \end{align*}
\end{prop}

\begin{proof}
    The proof of this goes by following the proof of Proposition 4.4 of~\cite{LMST}, and replacing the two point function estimates applied by those of Lemma~\ref{lem:OZ-asymptotics-beta-indep}. We bound the numerator of their (6.1) as there by 
    \begin{align*}
        \sum_{\lambda \subset R^*: \partial \lambda = \{u,v\}\,,\,\lambda \text{ reaches }H_{\delta \alpha \sqrt{\ell}}^*} q_{R^*}(\lambda)  \le \sum_{z\in H_{\delta \alpha \sqrt{\ell}}^*}\sum_{\lambda\subset R^*: \partial \lambda = \{u,v\}\,,z\in \lambda} q_{R^*}(\lambda) \le \pi^*_{R^*}(\sigma_u \sigma_z)\pi^*_{R^*}(\sigma_z \sigma_v)
    \end{align*}
    which by the GKS inequality is at most $\pi^*_{S^*}(\sigma_u \sigma_z) \pi_{S^*}^*(\sigma_z\sigma_v)$. Having changed the domain to $S^*$, one follows the numerator and denominator bound from the proof of Lemma~\ref{lem:maximal-fluctuation-bound} (which use our Lemma~\ref{lem:OZ-asymptotics-beta-indep} instead of those of their Lemma 2.1 which did not quantify $\beta$-dependencies), to get 
    \begin{align*}
        \frac{\sum_{\lambda \subset R^*: \partial \lambda = \{u,v\}\,,\, \lambda \cap H^*_{\delta \alpha \sqrt{\ell}} \ne \emptyset} q_{R^*}(\lambda)}{\sum_{\lambda \subset S^*: \partial \lambda = \{u,v\}} q_{S^*}(\lambda)} \le C^3 e^{3\beta} \ell^{3/2}  \max_{w\in \lambda \cap H^*_{\delta \alpha \sqrt{\ell}} }\!\exp\Big( \!- \tau(u-w) - \tau(v-w) + \tau(u-v)\Big)\,.
    \end{align*}
    By convexity of the surface tension function as a function of a vector (not just an angle), 
    \begin{align*}
        \min (\tau(u-w)  + \tau(v-w)) = \tau(u - w_*) + \tau(v-w_*)
    \end{align*}
    where $w_*$ is the midpoint of $u,v$ plus height $h$, i.e., $u+ (\frac{L}{2},h) = v-(\frac{L}{2},h)$. For that point, we can write by second order Taylor expansion and $\tau'(0)= 0$, that 
    \begin{align*}
        \tau(u-w_*) = |u-w_*| \tau(\theta_{uw_*}) = |u-w_*| \tau''(\xi_{uw_*})\theta_{uw_*}^2
        \end{align*}
        for angle $\theta_{uw_*}$ formed by the vector $w_* -u$, and for some $\xi_{uw_*}\in [-\theta_{uw_*}, \theta_{uw_*}]$. Thus, if we have $\frac{2h}{\ell}\le ce^{-2\beta}$ for small $c$, we also will have $\theta_{uw_*} \le ce^{-2\beta}$ and using ~\eqref{eq:tau''-upper-bound} from Lemma~\ref{lem:tau-double-prime-near-zero-angle}, will get
        \begin{align*}
            \tau(u-w_*) + \tau(v-w_*) \ge (|u-w_*| + |v-w_*|)(\tau(0) + \frac{1}{8}e^{2\beta}\theta_{uw_*}^2) \ge \frac{1}{16} e^{2\beta} \frac{h^2}{n} + \tau(0) |u-v|\,.
        \end{align*}
        Since the height of $w$ is $\delta e^{-\beta} \bar \alpha \sqrt{\ell}$, we find that $2h/\ell = \delta e^{- \beta} \bar \alpha \ell^{-1/2}$. Since $\ell \ge \beta e^{ 2\beta}$, we see that this is at most $\delta  \beta^{-1/2} e^{ - 2\beta} \bar \alpha$. Now suppose that $\bar \alpha = \sqrt{C \log \ell}$ for a $\beta$-independent $C$. Since $\log \ell \le 3\beta$ (say) for large $\beta$, we are left with $2h/\ell \le 3\delta \beta^{-1/2} \sqrt{C}$. So long as $\delta \bar \alpha \le \sqrt{\beta \log \ell
        }$ therefore, e.g., this angle will indeed lie within $ce^{-2\beta}$ for small $c$ as long as $\beta>\beta_0$.

    Finally, the remaining factor to estimate is the following ratio, which was in their (6.3): 
    \begin{align}\label{eq:floor-domain-to-infinite-strip}
        \Big( \sum_{\lambda \subset R^*:\partial \lambda = \{u,v\}}q_{S^*}(\lambda) \Big/ \sum_{\lambda \subset S^*: \partial \lambda = \{u,v\}} q_{S^*}(\lambda)\Big) \ge e^{ - C \beta} \frac{1}{\ell^{1.1}}\,.
    \end{align}
    where this now used our replacement Lemma~\ref{lem:thm:5.1-replacement}, and the constant $C$ is uniform over large $\beta$. 
\end{proof}

By a simple monotonicity argument, we also get an upper bound on the vertical oscillations of an interface in smaller domains. 

    \begin{cor}\label{cor:short-rectangle-vertical-oscillation}
        In the context of Proposition~\ref{prop:vertical-maximum-bound-replacement}, suppose $\ell \le \beta^2 e^{2\beta}$. For large $\bar \kappa$ (independent of $\beta$) and all $\beta>\beta_0$, we have 
        \begin{align*}
            \pi^{(-,-,+,-)} (\lambda \text{ reaches height $\sqrt{\bar \kappa} \beta$}) \le  e^{ - c_2 \bar \kappa \beta/2}
        \end{align*}
    \end{cor}
    \begin{proof}
        By monotonicity in boundary conditions, the interface is only higher if we increase the domain from having width $\ell$ to $\ell' = \beta e^{2\beta}$. Then, taking $\delta \bar \alpha = \sqrt{\bar \kappa \beta}$ and applying Proposition~\ref{prop:vertical-maximum-bound-replacement}, we get a probability bound of $\ell'^{c_1} e^{ - c_2 \bar \kappa \beta}$. For large $\bar \kappa$ (independent of $\beta$) this is evidently bounded by $e^{ - c_2 \bar \kappa \beta/2}$. 
    \end{proof}

    Let us also at this point state a corollary that is the analogue of Lemma~\ref{lem:maximal-fluctuation-bound} in the presence of a floor (but not necessarily with the sharp exponential rate because it will be applied beyond the diffusive scale). 

    \begin{lem}\label{lem:maximal-vertical-oscillations-with-floor}
        Consider a $2N\times 2N$ box with boundary conditions that are $+$ on the bottom side, and $-$ on the other three sides, for $N\ge e^{\beta/5}$. There exists a $C$ (independent of $\beta$) such that for all $\beta>\beta_0$, for all $h$, 
        \begin{align*}
            \pi^{(-,-,+,-)}(\lambda(\sigma) \text{ reaches $H_h^*$}) \le C N^C \exp( - \kappa_\beta (\tfrac{h^2}{b-a-1} \wedge h))\,.
        \end{align*}
    \end{lem}

    \begin{proof}
        Up to changing the domain to the infinite strip $S^*$ of width $2N$, this is the same as the estimate of Lemma~\ref{lem:maximal-fluctuation-bound}. We claim that changing the domain to drop the floor incurs only a polynomial in $N$ cost. Indeed, writing the probability in terms of the random line function $q$ as in the proof of Proposition~\ref{prop:horizontal-splitting-beta-uniform}, and changing the domain to $S^*$ as there using the GKS inequality and monotonicities of the $q$ function, the only additional cost beyond the probability bound of Lemma~\ref{lem:maximal-fluctuation-bound} is division by the right-hand side of~\eqref{eq:floor-domain-to-infinite-strip} which is bounded by a uniform in $\beta$ polynomial of $N$. 
    \end{proof}

\subsection{$\beta$-independent analogue of Proposition 4.5 of~\cite{LMST}}

That last main equilibrium estimate to get a $\beta$-uniform version of, is the following. For a rectangle $R$, we follow the notation of~\cite{LMST} with the corners of $R$ labeled clockwise starting from $NW$ as $x,y,y',x'$. Also, there is an interval $\Delta$ with endpoints $u,v$  of length $s\bar \alpha^2$ centered along the south boundary of $R$. We use $(-,+,\Delta)$ to denote boundary conditions that are $-$ on the north boundary and on $\Delta$, and plus elsewhere. The event $\mathcal V$ is the event that the endpoints $u,v$ of $\Delta$ are connected to the corners $x,y$ through two contours that are confined to the left and right halves of $R$ respectively.

\begin{prop}[Replacement for Proposition 4.5 of~\cite{LMST}]\label{prop:horizontal-splitting-beta-uniform}
    There exists $C,c_1,c_2,s_0>0$ (independent of $\beta$) such that for all $\beta>\beta_0$ the following holds. If $R$ is an $\ell \times e^{ - \beta} \bar \alpha \sqrt{\ell}$ rectangle with $\sqrt{C \log \ell} \le\bar \alpha \le (1/ s)\sqrt{\ell}$, and $\Delta$ has length $s \bar \alpha^2$ for $s \ge s_0$, we have 
    \begin{align*}
        \pi^{(-,+,\Delta)}(\mathcal V^c) \le \ell^{c_1} e^{ - c_2 \bar\alpha^2}\,.
    \end{align*}
\end{prop}
\begin{proof}
    The key distinction to Proposition 4.5 of~\cite{LMST} is that we will take our rectangles $R$ to be of size $\ell \times e^{-\beta} \bar \alpha \sqrt{\ell}$.

    The first main lemma of the proof shows that (due to the boundary modification on $\Delta$) it is likely that the two interface contours are connecting $x\leftrightarrow u$ and $y\leftrightarrow v$ as opposed to $x\leftrightarrow y$ and $u\leftrightarrow v$. The main difference is the $e^{-\beta}$ scaling of the rectangle height, and the $\beta$-independence of~$\bar \alpha$.
    
    \begin{lem}[Replacement for Lemma 6.1 of~\cite{LMST}]\label{lem:Lem6.1-replacement}
        For $R$ and $\Delta$ as in Proposition~\ref{prop:horizontal-splitting-beta-uniform}, there exists $c_3,c_4$ and $s_0$ (independent of $\beta)$ such that if $\beta>\beta_0, s\ge s_0$ and $\ell \ge \beta e^{2\beta}$, and $\bar \alpha \ge \sqrt{C \log \ell}$ for $C$ large independent of $\beta$, then 
        \begin{align*}
            \pi_R^\eta (\partial \lambda_1 =\{x,y\}\,,\,\partial \lambda_2 = \{u,v\}) \le \ell^{c_3} e^{ - c_4 \bar \alpha^2}\,.
        \end{align*}
    \end{lem}

\begin{proof}
    The first three displays up to the one preceding (6.5) proceed as in the proof of Lemma 6.1 of~\cite{LMST}, since these are all exact equalities/inequalities. When their Lemma 2.1 is applied to get their (6.5), we apply instead Lemma~\ref{lem:OZ-asymptotics-beta-indep}, to get for $\Psi_1$ as in their (6.4), that 
    \begin{align*}
        \Psi_1 \le Ce^{\beta} \frac{1}{\sqrt{\ell} } e^{- \tau_\beta(0) \ell} \cdot C e^{\beta} \frac{1}{\sqrt{s\bar \alpha^2}} e^{ - \tau_\beta(0) s\bar \alpha^2}\,.
    \end{align*}
    For the lower bound on the denominator $\Psi_2$, the inequalities that follow in~\cite{LMST} until (6.6) are unchanged as they are simply correlation inequalities; therefore, our first aim is to replace (6.6) and show that 
    \begin{align}\label{eq:(6.6)-equivalent}
        \frac{\bar X}{\bar Y}  \ge 1- (\ell/4)^{c_1} \exp( - c_2 \bar \alpha ^2)\,,
    \end{align}
    where $\bar X: = \sum_{\lambda \subset \mathcal G_1, \partial \lambda = \{x,z\}} q_{\bar S_*}(\lambda)$ and $\bar Y:= \sum_{\lambda \subset \bar S\,,\,\partial\lambda = \{x,z\}} q_{\bar S_*}(\lambda)$ as there. While there, this followed from the entropic repulsion bound of their Proposition 4.4, now our Lemma~\ref{prop:vertical-maximum-bound-replacement} ensures that even though the height of the rectangle is $e^{ - \beta} \bar \alpha \sqrt{\ell}$, the bound still holds, as $\ell \ge \beta e^{2\beta}$. In particular, as long as $\ell \ge \beta e^{2\beta}$, we have for $\beta$ large and $\bar \alpha \ge \sqrt{C \log \ell}$ that the probability on the right-hand side of~\eqref{eq:(6.6)-equivalent} is at least $1/2$ say. 

    Defining $Y$ as in their proof as relaxing to contours in all of $S$ (the doubly infinite strip), by Lemma~\ref{lem:thm:5.1-replacement} (replacing their use of Theorem~5.1), we have 
    \begin{align*}
        \frac{\bar Y}{Y} \ge e^{ - C \beta} \ell^{-1.1}\,.
    \end{align*}
    On the other hand, replacing their use of Formula 2.22 of~\cite{GreenbergIoffe}, by our Lemma~\ref{lem:OZ-asymptotics-beta-indep}, 
    \begin{align*}
        Y \ge \frac{e^{-\beta}}{C} \frac{1}{\sqrt{\ell}} \exp( - \tau_{\beta}(x-z))\,.
    \end{align*}
    Combining, we deduce the equivalent to their (6.9): 
    \begin{align}\label{eq:Psi2-lower-bound}
        \Psi_2 \ge C'^{-1} e^{-2\beta} \ell^{-3} \exp( - 2\tau_\beta(x-u))
    \end{align}
    To conclude the proof, as there, we must compare
    \begin{align*}
       - (\ell + s\bar \alpha ^2) \tau_\beta(0) \qquad \text{to} \qquad  2\tau_\beta(x-u))\,.
    \end{align*}
    Since $\tau_\beta(\theta)$ is analytic and even, we get from part (1) of Lemma~\ref{lem:tau-double-prime-near-zero-angle} that for all $\theta$, by second-order Taylor expansion in remainder form,
    \begin{align*}
        |\tau_\beta(\theta)  - \tau_\beta(0)| \le \frac{1}{2} e^{2\beta} \theta^2\,.
    \end{align*}
    Thus, since $\theta \le \arctan ( (e^{-\beta}\bar \alpha \sqrt{\ell})/(\ell/2))\le 2e^{-\beta} \bar \alpha \ell^{-1/2}$ it follows that 
     \begin{align*}
         \tau_\beta(x-u) \le \tau_\beta(0)|x-u| + \frac{1}{2}e^{2\beta} (e^{-\beta}\bar \alpha)^2 = \tau_\beta(0) |x-u|  + \frac{1}{2}\bar \alpha ^2
     \end{align*}
On the other hand, one has by expanding $|x-u| = \sqrt{(\ell - s \bar \alpha^2)^2/4  + e^{ -2\beta}\bar \alpha^2 \ell}$, that 
    \begin{align*}
         \ell + s\bar \alpha^2 - 2|x-u| \ge   (2s-  1)\bar\alpha^2 
     \end{align*}
     which is at least $s\bar \alpha ^2/2$ if $s\ge 1$. Combining these, we get 
     \begin{align*}
         \frac{\Psi_1}{\Psi_2} \le C'' e^{ 3 \beta} \ell^6 \exp( -\tau_\beta(0) (s/2) \bar\alpha ^2  + \bar \alpha^2)
     \end{align*}
    which gives the claimed bound since $\tau_\beta(0) \ge \beta$ per Lemma~\ref{lem:uniform-in-beta-bounds-on-surface-tension-and-stiffness}, as long as $s \ge \frac{C}{\beta}$ for large universal constant $C$. 
\end{proof}

The next lemma in the proof of Proposition~\ref{prop:horizontal-splitting-beta-uniform} is the following that confines the two interfaces to the left and right halves of $R$ respectively. 

\begin{lem}[Replacement for Lemma 6.2 of~\cite{LMST}]\label{lem:Lem6.2-replacement}
    Let $R_l, R_r$ be left and right halves of $R$. There exist $c_5, c_6$ (independent of $\beta$) such that for all $\beta>\beta_0$ and $\ell \ge \beta e^{2\beta}$, and $\sqrt{C \log \ell} \le\bar \alpha \le (1/s)\sqrt{\ell}$ for $C$ large independent of $\beta$,
    \begin{align*}
        \pi_R^\eta(\lambda_1 \subset R_l, \lambda_2 \subset R_r \mid \partial \lambda_1 = \{x,u\}\,,\,\partial\lambda_2 = \{y,v\}) \ge 1- \ell^{c_5} e^{ - c_6 s \bar \alpha^2}
    \end{align*}
\end{lem}

\begin{proof}
    The first two displays of the proof of Lemma 6.2 in~\cite{LMST} are exact inequalities, and therefore unchanged. When they applied their Lemma 6.1, we apply our Lemma~\ref{lem:Lem6.1-replacement}, to get
    \begin{align*}
        \frac{\Phi_2}{\Psi_2} \ge 1- \ell^{c_3} e^{ - c_4 \bar \alpha^2}
    \end{align*}
    for $\Phi_1,\Phi_2$ defined as in their proof of Lemma 6.2 of~\cite{LMST}. Using the lower bound on $\Psi_2$ from~\eqref{eq:Psi2-lower-bound}, 
    \begin{align*}
        \Phi_2 \ge (1- \ell^{c_3} e^{ - c_4\bar \alpha^2}) \ell^{ - 8} e^{ - \tau_\beta(x-u) - \tau_\beta(y-v)}\,.
    \end{align*}
    For the upper bound on $\Phi_1$, the steps in~\cite{LMST} are unchanged (as again, they only depend on exact, non-asymptotic, inequalities) up to 
    \begin{align*}
        \Phi_1 \le \sum_{z\in I} \exp( - \tau_\beta(z-x) - \tau_\beta(z-u) - \tau_\beta(y-v))
    \end{align*}
    Now when applying the sharp triangle inequality to lower bound the surface tension terms above, we use the sharp triangle inequality with uniform in large $\beta$ constant from Lemma~\ref{lem:uniform-in-beta-bounds-on-surface-tension-and-stiffness} to get 
    \begin{align*}
        \tau_\beta(z-x) + \tau_\beta(u-z) - \tau_\beta(u-x) \ge \frac{1}{3} (|z-x| + |u-z| - |u-x|)\,.
    \end{align*}
    Combining these bounds, and using that $\bar\alpha \ge \sqrt{C\log \ell}$ with large $\beta$-independent $C$, so that $(1-\ell^{c_3} e^{ - c_4 \bar \alpha^2})\ge 1/2$, we get that 
    \begin{align*}
        \frac{\Phi_1}{\Phi_2} \le 2 \ell^{8} \exp ( - \frac{1}{3} (|z-x| + |u-z| - |u-x|))\,.
    \end{align*}
    From this point, the geometric steps to upper bound this parallel~\cite{LMST}: firstly, note that
    \begin{align*}
        \min_{z\in I} (|z-x| + |u-z| - |u-x|) \le |v-x| -  |u-x|\,;
    \end{align*}
    recalling the dimensions of $R$ as now $e^{  - \beta} \bar \alpha \sqrt{\ell} \times \ell$ and $|u-v| = s\bar \alpha^2$, 
    \begin{align*}
        |v-x|^2 = \frac{1}{4} (\ell + s\bar \alpha^2)^2 + e^{ -\beta} \bar \alpha^2\ell \qquad |u-x|^2 = \frac{1}{4} (\ell - s\bar\alpha^2)^2 + e^{ -\beta} \bar \alpha^2 \ell\,.
    \end{align*}
    As long as $\bar \alpha \le (1/s) \sqrt{\ell}$, we get $|u-x| + |v-x|\le 2\ell$ and thus
    \begin{align*}
        |v-x| - |u-x| = \frac{s\bar \alpha^2 \ell}{|u-x| + |v-x|} \ge \frac{1}{2} s\bar \alpha ^2\,.
    \end{align*}
    Plugging this in, and summing over the $|I|\le \ell$ values that $z$ can take, $\frac{\Phi_1}{\Phi_2} \le 2\ell^9 \exp( - \frac{1}{6} s\bar \alpha ^2)$\end{proof}

    Lemmas~\ref{lem:Lem6.1-replacement} and~\ref{lem:Lem6.2-replacement} immediately imply Proposition~\ref{prop:horizontal-splitting-beta-uniform}. 
\end{proof}

\subsection{The $\beta$-uniform version of the recursive scheme}

The following will replace Theorem 4.2 of~\cite{LMST}, which is the main recursive scheme used to give quasi-polynomial mixing. Let $N$ be a large integer, let $L=L_N = 2^N -1 \ge \beta^2 e^{2\beta}$ and choose $N_0$ to be the smallest integer such that $L_{N_0}  \ge (\beta^2 e^{2\beta} )\vee e^{2\beta} (  \lfloor \log L\rfloor^3)$. 

For intermediate $n\in [N_0,N]$, define the rectangles $R_n$, $Q_n$ to have sides parallel to the coordinate axes of length $(L_n, \kappa_N\sqrt{L_n})$ and $(L_n, \kappa_N \sqrt{L_{n+1}})$ respectively,  where $L_n = 2^n-1$ and $\kappa_N = e^{-\beta}\sqrt{ \bar \kappa N}$ where $\bar \kappa$ is a constant independent of $\beta$ to be chosen later. Note that $\kappa_N \le \bar C e^{- \beta} \sqrt{\log L}$ for a $\bar C(\bar \kappa)$ (independent of $\beta$). Observe that this is much smaller than $\sqrt{L_n}$ for all $n\in [N_0,N]$, and in turn we have that $\kappa_N \sqrt{L_{n+1}}$ is much smaller than $L_n$.

\begin{definition}[Definition~3.1 of~\cite{LMST}]
    A distribution $\mathbf{P}$ of boundary conditions for a rectangle $R$ (either $R_n$ or $Q_n$) is in $\mathcal D(R)$ if the marginal on its north, west, and east boundaries is stochastically below $\pi_{\beta,\mathbb Z^2}^-$ and its marginal on south boundary stochastically dominates $\pi_{\beta,\mathbb Z^2}^+$. 
\end{definition}

\begin{definition}[Definition~3.2 of~\cite{LMST}]
    For $n\in \mathbb N$, $\delta>0$, $t>0$, consider the Ising model in $R_n$ with random boundary condition $\tau \sim \mathbf{P}$. We say $\mathcal A(L_n,t_n,\delta_n)$ holds if 
    \begin{align*}
        \mathbf{E}||\mu_{t_n}^\pm - \pi^\tau\|\le \delta_n \qquad \text{for all $\mathbf P\in \mathcal D(R_{L_n})$}\,.
    \end{align*}
    The statement $\mathcal B(L_n,t_n,\delta_n)$ is defined with $Q_{L_n}$ replacing $R_{L_n}$. 
\end{definition}

\begin{prop}[The starting point; analogue of Proposition 4.1 of~\cite{LMST}]\label{prop:starting-point}
    There exists $C>0$ such that for all $\beta$, for any $\ell \times h$ rectangle $R$ with $h \ge \log \ell$, with any boundary conditions $\tau$, has mixing time at most $e^{ 8 \beta h}$. As a consequence, for all $n \in [N_0,N]$, we have 
    \begin{align*}
        \mathcal A (L_n, t, e^{ - t e^{ - 8\beta \kappa_N \sqrt{L_n}}})  \qquad \text{and} \qquad \mathcal B(L_n, t, e^{-t e^{- 16\beta \kappa_N \sqrt{L_n}}}) \quad \text{hold}\,.
    \end{align*}
\end{prop}

\begin{proof}
    A classical canonical paths bound (e.g., Proposition 1.1 of~\cite{BKMP-trees}) gives that the inverse spectral gap  of the Ising Glauber dynamics on a rectangle $\ell \times h$ for $h\le \ell$ is $\ell h\exp( 4\beta h)$; changing this into a mixing time bound and absorbing the polynomial prefactors into the exponent since $h\ge \log \ell$ gives the first claim. The second claim then holds by exponential decay of total variation distance after mixing, and then averaging over the random boundary conditions.   
\end{proof}

\begin{thm}[The inductive step; analogue of Theorem 4.2 of~\cite{LMST}]\label{thm:inductive-step}
    Fix $s$ (in the definition of $\Delta$) sufficiently large. There exist constants $c_1,c_2,c_3$ and $\bar \kappa_0$ such that for all $\beta>\beta_0$, $\bar \kappa \ge\bar\kappa_0$, all $L\ge \beta^2 e^{2\beta}$ and any $n\in [N_0,N]$, 
    \begin{align*}
        \mathcal A(L_n,t_n,\delta_n) \implies \mathcal B(L_n, t_n',\delta_n') \implies \mathcal A(L_{n+1}, t_{n+1},\delta_{n+1})\,,
    \end{align*}
    where 
    \begin{align*}
        \delta_n' &= c_1\big(\delta_n + 
       L_N^{-c_2 \bar \kappa} + L_n^2 e^{ - c_2 \log t_n}\big) \quad &&;\quad &t_n' & = 2t_n \,,\\ 
        \delta_{n+1} &= c_3 \big( \delta_n + L_N^{-c_2 \bar \kappa} \big) \quad & &; \quad& t_{n+1} & = e^{ c_3 \beta \bar \kappa N} t_n\,.
    \end{align*}
\end{thm}

\begin{cor}[Solving for the final scale; analogue of Corollary 4.3 of~\cite{LMST}]\label{cor:final-scale}
    In the setting of Theorem~\ref{thm:inductive-step}, there exists $c>0$ such that for all $\beta>\beta_0$, if $t_N = e^{ c \beta \bar \kappa N^2}$ and $\delta_N = c e^ { - c^{-1} \bar \kappa N}$, then for all $N$ such that $2^N\ge \beta^2 e^{2\beta}$, the statement $\mathcal A(L_N,t_N,\delta_N)$ holds. 
    
    At the same time, if $L\in [e^{\beta/10}, \beta^2 e^{2\beta}]$, then $\mathcal A(L, t, e^{ - t e^{- c \beta(\log L)^2}})$ holds for all $t$.
\end{cor}

\begin{proof}
    For the case $L\ge \beta^2 e^{2\beta}$, we apply Theorem~\ref{thm:inductive-step} as follows. Choose $t_{N_0} = e^{ c' \bar \kappa N^2}$ for some $c'$ independent of $\beta$. Applying Proposition~\ref{prop:starting-point}, since $t_{N_0} \gg e^{ 8 \beta \kappa_N \sqrt{L_{N_0}}} = e^{8 \beta e^{-\beta} \bar \kappa N^2}$, we get that $\mathcal A(L_{N_0}, t_{N_0},\delta_{N_0})$ holds for $\delta_{N_0}= e^{ - c' \bar \kappa N^2/2}$ for a $c'$ (independent of $\beta$). Theorem~\ref{thm:inductive-step} then implies that $\mathcal A(L_N, t_N,\delta_N)$ holds for $t_N = t_{N_0} 2^{(N-N_0)}e^{ c_3 \beta \bar \kappa N(N-N_0)}$ which is at most $e^{ c \beta \bar \kappa N^2}$ for some $c$ independent of $\beta$, and $\delta_{N} = (c_1 \vee c_3)^{N-N_0} (\delta_{N_0} +  L_N^{-c_2 \bar \kappa} + L_N^2 e^{ - c_2 c' \bar\kappa N^2})$. In turn, $\delta_N$ is at most the claimed $ce^{ - c^{-1}\bar\kappa N}$ for a $\beta$-independent $c$ so long as $\bar \kappa$ is sufficiently large (depending on $c_1,c_2,c_3$ but not on $\beta$). 

    For the second part of the corollary, if $L \in [e^{\beta/10}, \beta^2 e^{2\beta}]$. we have as in Proposition~\ref{prop:starting-point} that the mixing time for the $L \times \beta^2$ rectangle $\tilde R_L$ is at most $\exp(8\beta^3) \le \exp(16 \beta (\log L)^2)$. In particular, for every $L \in [e^{\beta/10},\beta^2 e^{2\beta}]$, one has $\mathcal A(L, t, e^{ - t e^{c \beta (\log L)^2}})$ for $\beta$-independent $c$. 
\end{proof}

\subsection{Proof of Theorem~\ref{thm:inductive-step}}

We now need to justify the proof of Theorem~\ref{thm:inductive-step} with the corresponding uniform-in-$\beta$ analogues: i.e., with our Propositions~\ref{prop:vertical-maximum-bound-replacement} and~\ref{prop:horizontal-splitting-beta-uniform} replacing Propositions 4.4--4.5 of~\cite{LMST}. In the $\beta$-dependent case, that implication is actually found in the older~\cite{MaTo}, and therefore we will explain those steps and how they interact with large $\beta$. These will be quite a bit less subtle than the preceding steps because the main static estimates in this section are (a) using couplings to domain enlargements to go from random boundary conditions in $\mathcal D(R)$ to deterministic plus/minus ones, where the closeness of the two measures is bounded by Peierls bounds and therefore strictly improving in $\beta$; (b) a vertical reduction of $L\times L$ boxes to $L\times \tilde O(\sqrt{L})$ rectangles, which uses upper bounds on fluctuations of horizontal interfaces (already shown to be improving in $\beta$ in Proposition~\ref{prop:vertical-maximum-bound-replacement}).

Our Theorem~\ref{thm:inductive-step}, which is the analogue of Theorem 4.2 in~\cite{LMST} is also the analogue of Theorem 3.2 of~\cite{MaTo} and that latter is whose proof we will now be adapting. In this section, we therefore adopt the notation of the proof of Theorem 3.2 in~\cite{MaTo}. We use $\mu_t^{\pm}$ to denote the law of the Ising dynamics at time $t$ started from all $+$ or $-$ respectively.

\medskip
\noindent\textbf{Proof of Theorem~\ref{thm:inductive-step}: first implication.}

\medskip
\noindent \emph{Analogue of Proof of Theorem 3.2: part (1), (i) in~\cite{MaTo}}. We begin with part (1) of Theorem~\ref{thm:inductive-step}, meaning the implication from $\mathcal A$ at scale $n$ to $\mathcal B$ at scale $n$, started from the all-plus initialization. The arguments are non-asymptotic and uniform in $\beta$ until the step bounding the four terms in ~(3.14) of~\cite{MaTo}. The first term of that equation is at most $\delta_n$ by the assumption of $\mathcal A(L_n,t_n,\delta_n)$. The second term only uses monotonicity to also reduce it to the assumption, and therefore bound it by $\delta_n$. The third and fourth terms on right of (3.14) of~\cite{MaTo} are the ones that need some care in tracking $\beta$-dependencies. As they note, the two go by the same argument and thus we present only the modification for the fourth term.

\subsubsection*{Domain enlargements and applying Proposition~\ref{prop:vertical-maximum-bound-replacement}}
We need to give the analogue of Claim 3.6 of~\cite{MaTo}. This will be where some coupling steps are utilized to change out random boundary conditions in $\mathcal D(Q_{L_n})$ to exactly $(-,-,+,-)$ boundary conditions, and then apply our $\beta$-uniform equilibrium estimate, Proposition~\ref{prop:vertical-maximum-bound-replacement}. Recall the notation of~\cite{MaTo} from their Definition 3.5 that $\tau \sim \mathbf{P}$ is a randomly drawn boundary, and $(\tau,-)$ is the boundary condition that is $\tau$ on the north, east, west sides of a rectangle and $-$ on the south. 

\begin{claim}[Analogue of Claim 3.6 of~\cite{MaTo}]\label{clm:Claim-3.6-analogue}
    For all large $\bar \kappa$ (independent of $\beta$), for all $\beta>\beta_0$, 
    \begin{align*}
        \mathbf {E} [ \pi^\tau(\sigma_x = +) - \pi^{\tau,-}(\sigma_x = +)] \le   e^{ - c \bar \kappa N} + e^{ - \beta L_n/2}
    \end{align*}
\end{claim}
\begin{proof}
    The inequalities up to (3.18) in~\cite{MaTo} are all by monotonicity and have no $\beta$-dependence. Now recall that we use the notation $E_{L_n}(Q_{L_n})$ for the enlargement (in all but the south direction) of $Q_{L_n}$ by $L_n$. If, as there, $\pi_\infty^{(-,-)}$ is $\pi_\infty^-$ conditioned on minuses on the north, east, west boundaries of $E_{L_n} (Q_{L_n})$, then we use that the exponential decay in the minus phase is bounded by a Peierls bound (and therefore is non-asymptotic and improving with $\beta$) as follows. 
    
    \begin{lem}\label{lem:pure-phase-decay-of-correlations}
    Let $\pi_\infty^-$ be the infinite volume minus measure, and fix some finite sets $U,V$. For $\beta>\beta_0$, we have 
    \begin{align*}
        \|\pi_\infty^- (\sigma(V) \in \cdot) - \pi_\infty^-(\sigma(V) \in \cdot \mid \sigma(U)\equiv -1)\|_{\tv}\le |U| e^{ - \beta d(U,V)}\,.
    \end{align*}
\end{lem}
\begin{proof}
    This is a consequence of the standard Peierls bound, together with a monotone coupling. Namely, the total-variation distance is bounded by the probability that under $\pi_\infty^-$, there is a $*$-connected $+1$-path connecting $U$ to $V$, whose probability is bounded by the right-hand side. 
\end{proof}

 Therefore, we get the following refinement of (3.18): for $\beta>\beta_0$, 
    \begin{align*}
        \pi_{\infty}^-[\pi^{\tau,+}(\Gamma^c)] \le \pi_\infty^{(-,-)}[\pi^{\tau,+}(\Gamma^c)] + e^{- \beta {L_n}/2}\,.
    \end{align*}
    The next steps use monotonicity to add back the $+$ boundary conditions on the south of $E_{L_n}(Q_{L_n})$ to reduce the estimate to $\pi^{(-,-,+,-)}_{E_{L_n}(Q_{L_n})}$. Their estimate (3.20) under $\pi^{(-,-,+,-)}_{E_{L_n}(Q_{L_n})}$ is exactly what was bounded in a uniform manner in Proposition~\ref{prop:vertical-maximum-bound-replacement}. Namely, applying Proposition~\ref{prop:vertical-maximum-bound-replacement} with $\ell = 2L_n$ and  $\delta \bar \alpha \sqrt{L_n} = \sqrt{\bar \kappa N} (\sqrt{L_{n+1}} - \sqrt{L_n})$, where we note that $\delta \bar \alpha \ge \frac{1}{2} \sqrt{\bar \kappa N} \ge \frac{1}{2} \sqrt{\bar \kappa} \sqrt{\log L_N}$ to get 
    \begin{align*}
        \pi_{E_L(Q_L)}^{(-,-,+,-)}(\gamma \text{ reaches height of South border of $A$}) \le (2L_n)^{c_1}e^{- c_2 \bar \kappa \log {L_N}} \le (L_N)^{-c_2'\bar \kappa}
    \end{align*}
    with the last inequality holding so long as $\bar\kappa$ is sufficiently large (independent of $\beta$). 
\end{proof}

Taking a union bound over $x$ in Claim~\ref{clm:Claim-3.6-analogue}, and the bounds on the other terms in (3.14) of~\cite{MaTo}, we will have obtained the desired for some other $c_2'$ independent of $\beta$:
\begin{align*}
    \mathbf{E}\|\mu_{2t}^+ - \pi^\tau\|\le 2\delta + 2L_N^{- c_2' \bar \kappa}\,.
\end{align*}

\medskip
\noindent \emph{Analogue of Proof of Theorem 3.2: part (1), (ii) in~\cite{MaTo}}. This corresponds to the minus initialization. As with the plus initialization, the censoring scheme is the same as in~\cite{MaTo}, all monotonicity inequalities are unchanged, and we arrive at their (3.22). The first term in that expression is at most $\delta$ by the assumption and the third term is at most $L_N^{- c_2'\bar\kappa}$ by the same argument as given above for the plus initialization. For the second term, the important quantity to consider is 
\begin{align*}
    \sum_{x\in A} \pi^\tau [\nu_2^{\eta_{A^c}}(\sigma_x =-) - \pi^\tau (\sigma_x =-)]
\end{align*}
Given $x\in A$ and $r\in \mathbb N$, let $K_r$ be the intersection of $A$ with a square of side length $2r+1$ (where our choice of $r$ will possibly be $\beta$ dependent now). By replacing their use of the crude mixing time bound (3.1) with the $\beta$-dependent form of it in the first part of Proposition~\ref{prop:starting-point}, we deduce that the following analogue of their equation preceding Claim 3.8 holds: 
\begin{align*}
    \nu_2^{\eta_{A^c}}(\sigma_x = -) - \pi^\tau(\sigma_x = -) \le e^{ - t e^{ - 8 \beta r}} + [\pi_r^{\tau,\eta_{A^c}}(\sigma_x = -) - \pi^\tau(\sigma_x = -)]
\end{align*}
where $\pi_r^{\tau, \eta_{A^c}}$ is the measure on $K_r$ with its induced boundary conditions. 

\begin{claim}[Analogue of Claim 3.8 of~\cite{MaTo}]\label{clm:Claim-3.8-analogue}
    For $\beta>\beta_0$ one has 
    \begin{align*}
        \mathbf{E}[\pi^\tau[\pi_r^{\tau,\eta_{A^c}}(\sigma_x =-)] - \pi^\tau(\sigma_x=-)] \le e^{ - \beta r/2 } + 2L_N^{- c_2'\bar\kappa}\,.
    \end{align*}
\end{claim}

Assuming the claim, if we choose $r = \frac{1}{8\beta} (\log t - \log \log t)$, we find that 
\begin{align*}
    \mathbf E[\pi^\tau[\|\nu_2^{\eta_{A^c}} - \pi_A^{\tau,\eta_{A^c}}\|]] \le L_n^2(e^{ - t e^{- 8\beta r}}+ e^{ - \beta r/2}) + 2L_N^{-c_2' \bar \kappa} \le L_n^2e^{ - \frac{1}{20}\log t} + 2L_N^{- c_2 ' \bar \kappa}
\end{align*}
as long as $t$ is at least a large ($\beta$-independent) constant, which it necessarily will be. 

\begin{proof}[\textbf{\emph{Proof of Claim~\ref{clm:Claim-3.8-analogue}}}]
    Again, the monotonicity steps remain intact as they are non-quantitative. Following the proof in~\cite{MaTo}, when replacing $\mathbf{E}[\pi^\tau(\Gamma^c)]$ by $\pi_{E_L(Q_L)}^{(-,-,+,-)}(\Gamma^c)$ in the first step, the error is in fact $e^{-\beta {L_n}/2}$ per Lemma~\ref{lem:pure-phase-decay-of-correlations} as in the proof of Claim~\ref{clm:Claim-3.6-analogue}. When replacing $\pi_{E_{L_n}(Q_{L_n})}^{(-,-,+,-)}$ with $\pi_{E_{L_n}(Q_{L_n})}^-$, by the argument of the next step in Claim~\ref{clm:Claim-3.6-analogue}, the error is at most  $L_N^{-c_2' \bar \kappa}$. Finally, again utilizing Lemma~\ref{lem:pure-phase-decay-of-correlations} for the minus-phase decay of correlations, 
the last inequality there is replaced by $\pi_{E_{L_n}(Q_{L_n})}^-(\Gamma^c) \le\pi_\infty(\Gamma^c) \le e^{ - \beta r/2}$. 
\end{proof}

\medskip
\noindent\textbf{Proof of Theorem~\ref{thm:inductive-step}: second implication.}
We begin with modifying the boundary conditions on the stretch $\Delta$ along the base of the rectangle $R_{L_n}$. This boundary condition modification introduces some $\beta$-dependencies that the below tracks. After presenting these equivalents of Section~2 material of~\cite{MaTo}, we return to conclude the proof of the second implication of Theorem~\ref{thm:inductive-step}.

\subsubsection*{Boundary condition modifications}

Following the presentation of~\cite{MaTo} chronologically and verifying the $\beta$-dependencies therein, we begin with the following bounds comparing the mixing times with modified boundary conditions. 

\begin{lem}[E.g., Lemma 2.8 of~\cite{MaTo}]
    Starting from a domain $\Lambda$ with boundary conditions $\tau$, if we change the boundary conditions on a subset $\Delta \subset \partial \Lambda$ of size $|\Delta|$ to get boundary conditions $\tau^{\Delta}$, then 
    \begin{align*}
        \tmix(\Lambda, \tau) \le  4\beta |E(\Lambda)|\log 2\cdot e^{4\beta |\Delta|}\cdot  \tmix(\Lambda, \tau^\Delta)
    \end{align*}
\end{lem}
\begin{proof}
    In their Lemma 2.8, the factor $M$ is bounded by $e^{4\beta |\Delta|}$ if $\tau$ and $\tau^\Delta$ differ on $|\Delta|$ many vertices. The constant $c$ is bounded by $\log (2/\min_{\sigma}\pi(\sigma))\le 4\beta|E(\Lambda)| \log2$. 
\end{proof}

As in~\cite{MaTo}, let $d^\pm(t) = \|\mu_t^\pm -\pi^\tau\|$ and let $\gamma(t) = \max\{d^+(t),d^-(t)\}$, and use $\Delta$ superscripts to denote the same things after the boundary modification to $\Delta$. Using the above explicit constants and following the next steps identically, we arrive at the following. 

\begin{lem}[Lemma 2.9 of~\cite{MaTo}]
    For distribution $\mathbf{P}$ on boundary conditions on $\partial \Lambda$, we have, 
    \begin{align*}
        \mathbf{E}[\gamma(t)] \le e^{ - e^{\beta |\Delta|}} + 8 \mathbf{E}^\Delta\Big[\gamma \Big(\frac{t}{4\beta |E(\Lambda)| \log2 e^{4\beta|\Delta|}}\Big)\Big]\,.
    \end{align*}
\end{lem}

We then arrive at the following corollary describing how the events $\mathcal A$ and $\mathcal B$ change under boundary conditions modifications on $\Delta \subset \partial R_L$.  

\begin{cor}\label{cor:boundary-modification}
    Let $R_{L_n}$ be as earlier, i.e., $L_n \times e^{ - \beta} \sqrt{\bar \kappa N} \sqrt{L_n}$, and let $|\Delta| = s\bar \kappa N$. Assume that 
    \begin{align*}
        \mathbf{E}^\Delta[\|\mu_t^\pm - \pi^\tau\|] \le \delta \qquad \forall \mathbf{P}\in \mathcal D(R_{L_n})\,.
    \end{align*}
    Then $\mathcal A(L_n,\tilde {t}_n,\tilde \delta _n)$ holds with $\tilde \delta_n = 8\delta_n + e^{ - e^{8\beta s\bar \kappa N}}$ and $\tilde t_n = 4\beta |E(\Lambda)|  \log 2 \cdot t_n \cdot  e^{4\beta s \bar \kappa N}$. 
    Analogously, $\mathcal A(L_n,t_n,\delta_n)$ implies $\mathbf{E}^\Delta[\|\mu_{t_n}^\pm - \pi^\tau\|]\le \tilde \delta_n$, and similar statements hold if we replace $R_{L_n}$ by $Q_{L_n}$ and $\mathcal A(L_n,\tilde t_n,\tilde \delta_n)$ by $\mathcal B(L_n,\tilde t_n',\tilde \delta_n')$. 
\end{cor}

With the above in hand, we return to the task of proving the second implication of Theorem~\ref{thm:inductive-step}. We begin by observing that by Corollary~\ref{cor:boundary-modification}, up to an additional factor of $e^{5 \beta s \bar \kappa N}$  in the time (assuming $s,\bar\kappa$ are sufficiently large $\beta$-independent constants), which gets absorbed into the choice of $c_3$ in $t_{n+1}$, we can modify the boundary conditions on $\Delta$ of size $s \bar \kappa N$. This also incurs a multiplicative factor of $8$ and an additive $e^{ - 8 \beta s \bar \kappa N}$ (which are naturally absorbed by the $c_3$ in the definition of $\delta_{n+1}$.  In other words, it is sufficient to bound $\mathbf{E}^\Delta[ \|\mu_{2t_{n+1}}^\pm - \pi^\tau\|]$ with some $\beta$-independent choice of $c_3$, in order to have proved the bound for $\mathbf{E}[\|\mu_{2t_{n+1}}^\pm - \pi^\tau\|] \le \delta_{n+1}$ with a different (still $\beta$-independent) choice of $c_3$.  

\medskip
\noindent \emph{Analogue of Proof of Theorem 3.2: part (2), (i) in~\cite{MaTo}}. In this case, the dynamics on $R_{L_{n+1}}$ begins from all-plus. The choice of the censoring scheme, and the monotonicity inequalities being unchanged, we arrive at equation (3.29) of~\cite{MaTo}. For the first term, since $t_{n+1}\ge 2\tilde{t}_n$ where $\tilde{t}_n$ is defined as in Corollary~\ref{cor:boundary-modification}, the first term in (3.29) is at most $\tilde \delta_{n} = 8 \delta_n + e^{ - e^{8\beta s \bar\kappa N}}$. For the second term, again the only quantitative (i.e., not monotonicity or exact equality/inequality estimate used) is their Corollary~2.10, which is replaced by our Corollary~\ref{cor:boundary-modification} to give that the second term is also at most $2\tilde \delta_n$. 

For the more delicate third and fourth terms in (3.29) of~\cite{MaTo}, again the first several steps are only using monotonicity. Since they are essentially the same argument, we focus on the third one (as in~\cite{MaTo}). The enlargement $\bar A$ is defined similarly to in that paper, meaning it doubles the height of $A$. The exponential decay of correlations in $\pi_\infty^-$ is bounded by a Peierls argument as in Lemma~\ref{lem:pure-phase-decay-of-correlations}, and thus (3.32) is replaced by 
\begin{align*}
    \mathbf{E}^\Delta[\pi_A^{\tau,+}(\Gamma^c)] \le e^{ - \beta e^{-\beta} \sqrt{\bar \kappa N} \sqrt{L_n}/2} + \pi_{\bar A}^{(-,+,\Delta)} (\Gamma^c)\,.
\end{align*}
Because we are assuming $L_n \ge e^{2\beta} (\lfloor \log L_N\rfloor^3)$, and $N \ge \frac{1}{2} \log L_N$, the first exponential is at most $e^{ - c \beta \sqrt{\bar \kappa} \log L_N}$ for a $\beta$-independent constant $c$.

For the term $\pi_{\bar A}^{(-,+,\Delta)}(\Gamma^c)$, which was bounded in Claim 3.10 in~\cite{MaTo}, for us the quantity is exactly what is bounded in a uniform-in-$\beta$ manner by our Proposition~\ref{prop:horizontal-splitting-beta-uniform}, with the choice of $\bar \alpha = \sqrt{\bar \kappa N}$, so long as the constant $s$ in the modified region $\Delta$ is large enough (that constant is $\beta$-independent and only appears in the $c_3$),  we end up with 
\begin{align*}
    \pi_{\bar A}^{(-,+,\Delta)} ( \Gamma^c) \le L_n^{c_1} e^{ - c_2 \bar \kappa N} \,.
\end{align*}
This is at most $L_N^{-c_2' \bar \kappa}$ as long as $\bar \kappa$ is a sufficiently large ($\beta$-independent) constant. The fourth term of their (3.29) handled identically (only the domain has width $2L_n+1$ instead of $L_n$).

\medskip
\noindent \emph{Analogue of Proof of Theorem 3.2: part (2), (ii) in~\cite{MaTo}}. Finally, suppose in the second implication that the initialization is all-minus. As in~\cite{MaTo}, this term is handled identically to the plus initialization modulo the obvious changes.

\begin{remark}\label{rem:general-side-lengths}
    Though the proof was presented for readability for integer side-lengths of the form $2^n-1$, it is fairly straightforward to replace with all integer side-lengths. Indeed, this is discussed in Remark 3.12 of~\cite{MaTo}, and the tweaks described therein apply mutatis mutandis (its tweaks would not induce any extra $\beta$ dependencies, as the families of domains at scale $n$ on which we recurse are within uniform-in-$\beta$ factors, in fact factors of $2$, of the specific domains $R_{L_n},Q_{L_n}$). 
\end{remark}

\subsection{Proof of uniform-in-$\beta$ quasipolynomial mixing time}

The final step to go through is to show that given the mixing time on $L \times \tilde O(\sqrt{L})$ regions, with $(-,-,+,-)$ boundary conditions, one can bound the mixing time on $L \times L$ boxes $\Lambda_L$. 

The below will follow the proof of Theorem 1.6 in~\cite{MaTo}. 

\begin{proof}[\textbf{\emph{Proof of Theorem~\ref{thm:beta-independent-2D-mixing-time}}}]
    Suppose that $L \ge e^{ \beta/10}$, as otherwise the bound has been shown in Lemma~\ref{lem:small-n-mixing-time}, possibly up to a change of the $\beta$-independent constant $C$. 

    \medskip
    \noindent \emph{Mixing time with ``$(-,-,+,-)$" boundary conditions}. 
    The main step is to establish for $$t_L = \exp(  \bar A\beta(\log L)^2)$$ for a large $\beta$-independent constant $\bar A$ to be determined, that for all $L\ge e^{\beta/10}$, we have  
    \begin{align}\label{eq:wts---+-bc}
        \mathbf E[\|\mu_{t_L}^\pm -\pi^\tau\|] \le L^{-6}\,.
    \end{align}
Define the quantity $H_0$ that will play the role of $L^{1/2 + \varepsilon'}$ from their paper as 
    \begin{align*}
        H_0(L) = \max\{\beta^{2}, e^{ - \beta} \sqrt{\bar \kappa \log L}\sqrt{L}\}
    \end{align*}
    Consider the evolution from the all-plus initialization. For all $i$, let 
    \begin{align*}
        h_i = H_0(L) +  i(H_0(2L+1) - H_0(L))
    \end{align*}
    and let $k$ be such that $h_{k-1} = L$. Evidently, $k\le \sqrt{L}$. Let $\Lambda_L^i$ be defined as the rectangle of height $h_i$ of the same base as $\Lambda_L$. 
The analogue of Lemma 4.1 of~\cite{MaTo} will be the following. 
    \begin{lem}[Analogue of Lemma 4.1 of~\cite{MaTo}]\label{lem:plus-start-final-step}
        The following holds for $0\le i\le k-1$ and for a $\beta$-independent constant $C$. If $\tau\sim \mathbf P$ for $\mathbf{P}\in \mathcal D(\Lambda_L^i)$, 
        \begin{align*}
            \mathbf{E}[\|\mu_{(i+1) t_L/k}^{+,i} - \pi_{\Lambda_L^i}^\tau \|] \le C(1+i) L^{-8}\,,
        \end{align*}
        where $\mu^{+,i}_t$ is the law of Ising dynamics in $\Lambda_L^i$ started from all plus. 
    \end{lem}
    Let us first defer the proof of the lemma, and conclude by also showing that $\mathbf{E}\|\mu_{t_L}^- - \pi^\tau\|\le C L^{-8}$. Using the same censoring scheme and Definition 4.2 there, we arrive at their equation (4.5). For the first term in their (4.5), in the case $L \ge \beta^2 e^{2\beta}$, Corollary~\ref{cor:final-scale} implies as long as $\bar A$ is large compared to $\bar \kappa$, it is at most $L^{-10}$. In the case $L\le \beta^2 e^{2\beta}$, by the second part of Corollary~\ref{cor:final-scale}, the mixing time is at most $\exp( 16 \beta^3)$, which since $L\ge e^{\beta/10}$, is at most $e^{ C \beta (\log L)^2}$ for some $\beta$-independent constant $C$, so as long as $\bar A$ is large, this term is also at most $L^{-10}$. For the third term in their (4.5), it is identically handled to the their third term in (3.22), which we already argued in ``Analogue of Proof of Theorem 3.2: part (1), (ii) in~\cite{MaTo}" was at most $L^{-10}$ if $L\ge \beta^2 e^{2\beta}$; if $L \le \beta^2 e^{2\beta}$, then it is similarly bounded because Corollary~\ref{cor:short-rectangle-vertical-oscillation} replaces the control on the vertical oscillations, yielding still error at most $L^{-10}$ for large $\bar \kappa$ (the height will still be less than $\beta^2$ for large $\beta$). Finally, the second term in their (4.5) is bounded like the second term in their (3.22), which we handled in our discussion around Claim~\ref{clm:Claim-3.8-analogue}.  In the case where $L\le \beta^2 e^{2\beta}$, the only step altered in the proof of Claim~\ref{clm:Claim-3.8-analogue} is the interface bound used to replace $\pi_{E_L(Q_L)}^{(-,-,+,-)}$ by $\pi_{E_L(Q_L)}^-$ for which we again use Corollary~\ref{cor:short-rectangle-vertical-oscillation} to still get $L^{-10}$ error.  

    If we now combine with the Lemma~\ref{lem:plus-start-final-step} with $i=k-1$, we get the claimed bound of~\eqref{eq:wts---+-bc}. 

    \begin{proof}[\textbf{\emph{Proof of Lemma~\ref{lem:plus-start-final-step}}}]
        The proof is inductive, with the $i=0$ case being exactly Corollary~\ref{cor:final-scale}. Now following the proof of~\cite{MaTo}, we arrive at their equation (4.9), and take expectation with respect to $\mathbf{P}$. The first term is bounded by Corollary~\ref{cor:final-scale} as there, by $L^{-10}$. The third and fourth terms are each bounded by $L^{-8}$ (the proof being identical to the bounds on the third and fourth terms in (3.14) of~\cite{MaTo}, which was exactly what we handled in our Claim~\ref{clm:Claim-3.6-analogue}). (For the case $L\in [e^{\beta/10},\beta^2 e^{2\beta}]$ that argument only used Peierls arguments which are unaffected, and the vertical interface fluctuation bound for which Corollary~\ref{cor:short-rectangle-vertical-oscillation} applies). Finally, we have the second term in (3.14) of~\cite{MaTo}, which is bounded  by the inductive hypothesis by $C i L^{-8}$.  
    \end{proof}

    \emph{Mixing time with ``$-$" boundary conditions}. At last we need one more argument to boost the above into a bound on the mixing time with all-minus boundary conditions. At this point, the modifications and making quantitative of all the estimates have been done before, but we go through them again for completeness. 
    We begin with the case of the all-plus initialization.  All the steps are unchanged (only using censoring, monotonicity, etc) up to the bound of (4.14) in~\cite{MaTo}. For the second term in (4.14), since the overlap of the two domains $\Lambda_L^-$ and $\Lambda_L^+$ is order $L$, the bound on interface fluctuations of Lemma~\ref{lem:maximal-fluctuation-bound} is applicable (in fact even a trivial $\epsilon_\beta L$ bound on the length of the interface already suffices) to give that this term is at most $e^{ - \beta L/5}$, say. The first term in their (4.14) is at most $L^{-6}$ by application of Lemma~\ref{lem:plus-start-final-step} in place of their Lemma 4.1. 

    Next we investigate the evolution from the all-minus initialization. This will be akin to the argument immediately preceding Claim~\ref{clm:Claim-3.8-analogue}. For given $x\in \Lambda_L$, letting $K_r$ be the $2r+1$ side-length box centered at $x$ as there, we follow the steps of their proof to arrive at the analogue of their (4.17): 
    \begin{align*}
        \mathbf{E}\|\mu_t^- - \pi^\tau\|\le \sum_{x\in \Lambda_L} \Big( \mathbf{E}\|\mu_{K_r,t}^{\tau,-} - \pi_{K_r}^{\tau,-}\| + e^{ - \beta r/2} \Big)\,.
    \end{align*}
    Using the $\beta$-dependent crude mixing time bound of Proposition~\ref{prop:starting-point}, choosing $t = t_L$ and $r = \frac{1}{8\beta} (\log t - \log \log t)$, the above is at most 
    \begin{align*}
        L^2 e^{ - \frac{1}{20} \log t_L }  \le L^2 e^{ - \beta (\log L)^2} \le L^{-10}
    \end{align*}
    for large $\beta$. This concludes the bound of mixing time with boundary conditions that are stochastically below the minus infinite-volume measure, which implies the same for the all-plus boundary condition by spin-flip symmetry. 
\end{proof}

\bibliographystyle{plain}
\bibliography{references}

\end{document}